\documentclass[3p]{elsarticle}

\usepackage[hidelinks]{hyperref}
\usepackage[utf8]{inputenc}

\hypersetup{
    colorlinks=true,
    linkcolor=ceruleanblue,
    filecolor=ceruleanblue,
    urlcolor=ceruleanblue,
    citecolor=ceruleanblue,
}

\usepackage{graphicx}
\usepackage{amsmath,amssymb}
\usepackage{color,amscd,amsthm}
\usepackage{enumitem,fancyhdr,dsfont,amsfonts}
\usepackage{tikz}
\usepackage{float,cancel}
\usepackage{stmaryrd}
\usepackage{mathabx} 
\usepackage[nameinlink]{cleveref}
\usepackage[normalem]{ulem}

\usepackage{lineno}

\setlist{itemsep=1ex, topsep=0ex}

\usetikzlibrary{arrows, intersections}
\pgfdeclarelayer{bg}
\pgfdeclarelayer{fg}
\pgfdeclarelayer{crossing over}
\pgfsetlayers{main,bg,crossing over,fg}

\usepackage{multicol,multirow} 

\usetikzlibrary{arrows, intersections}
\pgfdeclarelayer{bg}
\pgfdeclarelayer{fg}
\pgfdeclarelayer{crossing over}
\pgfsetlayers{main,bg,crossing over,fg}

    \definecolor{carrotorange}{rgb}{0.93, 0.57, 0.13}
    \definecolor{dodger}{rgb}{0.0,0.5,1.0}

\definecolor{ceruleanblue}{rgb}{0.16, 0.32, 0.75}

\definecolor{brown}{rgb}{0.43, 0.21, 0.1}

\newenvironment{to consider!}[1][Miro]  
    {\color{dodger}\leavevmode Alternative:\marginpar{#1 says:}}
    {}

\def\ml{l\kern-0.035cm\char39\kern-0.03cm}

\newcounter{enuAlph}

\theoremstyle{plain}
\newtheorem{theorem}{Theorem}[section]
\newtheorem{lemma}[theorem]{Lemma}

\newtheorem{corollary}[theorem]{Corollary}
\newtheorem{prop}[theorem]{Proposition}
\newtheorem{fact}[theorem]{Fact}

\newtheorem{fct}[theorem]{Fact}
\newtheorem{teorema}[enuAlph]{Theorem}
\newtheorem{question}[theorem]{Question}

\theoremstyle{definition}
\newtheorem{definition}[theorem]{Definition}
\newtheorem{example}[theorem]{Example}
\newtheorem{notation}[theorem]{Notation}
\newtheorem{remark}[theorem]{Remark}

\theoremstyle{remark}

\numberwithin{equation}{theorem}

\DeclareMathOperator{\cf}{cf}

\newcommand{\baire}[1]{{}^{\omega}#1}

\newcommand{\Cor}{\mathbb{C}}

\newcommand{\oneone}{\textrm{\upshape1-1}}

\newcommand{\cohen}[1]{V^{\mathbb{C}_{#1}}}

\newcommand{\fin}{\text{\rm Fin}}
\newcommand{\covh}[1]{\mathtt{cov}^*(#1)}

\newcommand{\I}{{ I}}
\newcommand{\J}{{ J}}

\newcommand{\calS}{{\mathcal{S}}}
\newcommand{\U}{{\mathcal{U}}}
\newcommand{\V}{{\mathcal{V}}}
\newcommand{\A}{{\mathcal{A}}}

\newcommand{\F}{{\mathcal{F}}}

\newcommand{\OO}{{\mathcal{O}}}
\newcommand{\PP}{{\mathcal{P}}}

\newcommand{\bb}{\mathfrak{b}}
\newcommand{\dd}{\mathfrak{d}}
\newcommand{\pp}{\mathfrak{p}}
\newcommand{\cc}{\mathfrak{c}}

\newcommand{\ppk}[1]{\pp_\mathrm{K}(#1)}

\newcommand{\add}[1]{\mathtt{add}(#1)}
\newcommand{\non}[1]{\mathtt{non}(#1)}
\newcommand{\cov}[1]{\mathtt{cov}(#1)}
\newcommand{\cof}[1]{\mathtt{cof}(#1)}
\newcommand{\nonst}{\mathtt{non}^*}
\newcommand{\covst}{\mathtt{cov}^*}

\newcommand{\covm}{\cov{\mathcal{M}}}
\newcommand{\covn}{\cov{\mathcal{N}}}
\newcommand{\nonm}{\non{\mathcal{M}}}
\newcommand{\nonn}{\non{\mathcal{N}}}
\newcommand{\addm}{\add{\mathcal{M}}}
\newcommand{\addn}{\add{\mathcal{N}}}
\newcommand{\cofm}{\cof{\mathcal{M}}}
\newcommand{\cofn}{\cof{\mathcal{N}}}

\newcommand{\sla}[1]{\mathfrak{sl_e}(#1)}
\newcommand{\slamg}[1]{\sla{\star,#1}}

\newcommand{\dsla}[1]{\mathfrak{sl_t}(#1)}

\newcommand{\dslaml}[1]{\dsla{\star,#1}}
\newcommand{\dslago}[1]{\dsla{#1,\fin}}

\newcommand{\sone}[2]{{\rm S}_1(#1,#2)}
\newcommand{\soneoo}{\sone{\OO}{\OO}}

\newcommand{\sonef}[3]{{\rm S}_1^{#1}(#2,#3)}

\newcommand{\isonegg}[2]{\sone{#1\text{-}\Gamma}{#2\text{-}\Gamma}}
\newcommand{\isoneggfi}[1]{\sone{\Gamma}{#1\text{-}\Gamma}}

\newcommand{\isonemg}[1]{\sone{\Omega}{#1\text{-}\Gamma}}

\newcommand{\isoneglfi}[1]{\sone{\Gamma}{#1\text{-}\Lambda}}

\newcommand{\isonefgg}[3]{\sonef{#1}{#2\text{-}\Gamma}{#3\text{-}\Gamma}}
\newcommand{\isonefgl}[3]{\sonef{#1}{#2\text{-}\Gamma}{#3\text{-}\Lambda}}
\newcommand{\isonefmg}[2]{\sonef{#1}{\Omega}{#2\text{-}\Gamma}}
\newcommand{\isonefgsg}[3]{\sonef{#1}{\Gamma_{#2}}{#3\text{-}\Gamma}}

\newcommand{\GammaB}[1]{#1\text{-}\Gamma}
\newcommand{\Gammac}[1]{\Gamma^{#1}}
\newcommand{\Gammah}[1]{\Gamma_{#1}}
\newcommand{\LambdaB}[1]{{#1}\text{-}\Lambda}

\newcommand{\cLambda}[2]{\mathfrak{sl_e}(#1,#2)}

\newcommand{\tsl}[2]{\mathfrak{sl_t}(#1,#2)}
\newcommand{\schema}[2]{\mathrm{S}_1(#1,#2)}

\newcommand{\la}{\langle}
\newcommand{\ra}{\rangle}

\newcommand{\be}{\mathfrak{b}}
\newcommand{\de}{\mathfrak{d}}
\newcommand{\pfrak}{\mathfrak{p}}

\newcommand{\Por}{\mathbb{P}}
\newcommand{\Qnm}{\dot{\mathbb{Q}}}
\newcommand{\Dor}{\mathbb{D}}
\newcommand{\Eor}{\mathbb{E}}

\definecolor{cadmiumorange}{rgb}{0.93, 0.53, 0.18}

    \DeclareMathOperator{\dom}{dom}

    \DeclareMathOperator{\ran}{ran}

    \newcommand{\thzfc}{\mathrm{ZFC}}

    \newcommand{\Ewf}{\mathcal{E}}

    \newcommand{\Iwf}{\mathcal{I}}

    \newcommand{\Mwf}{\mathcal{M}}
    \newcommand{\Nwf}{\mathcal{N}}
    
    \newcommand{\Pwf}{\mathcal{P}}
    \newcommand{\Pcal}{\mathcal{P}}
    \newcommand{\Scal}{\mathcal{S}}
    \newcommand{\Swf}{\mathcal{S}}

    \newcommand{\bfrak}{\mathfrak{b}}
    \newcommand{\cfrak}{\mathfrak{c}}
    \newcommand{\dfrak}{\mathfrak{d}}

    \newcommand{\kfrak}{\mathfrak{k}}
    \newcommand{\lfrak}{\mathfrak{l}}

    \newcommand{\pK}{\mathfrak{p}_{\rm K}}
    
    \newcommand{\sfrak}{\mathfrak{s}}

    \newcommand{\Cbf}{\mathbf{C}}

    \newcommand{\Lbf}{\mathbf{L}}
    \newcommand{\Mbf}{\mathbf{M}}

    \newcommand{\Ibb}{\mathbb{I}}

    \newcommand{\Qor}{\mathbb{Q}}

    \newcommand{\menos}{\smallsetminus}
    
    \DeclareMathOperator{\pts}{\mathcal{P}}

    \newcommand{\Prop}{\mathcal{P}}
    
    \newcommand{\Fin}{\mathrm{Fin}}
    \newcommand{\sbf}{\mathbf{s}}

    \newcommand{\match}{\sqsubset^{\rm m}}

    \newcommand{\imp}{\Rightarrow}

    \DeclareMathOperator{\rk}{\rk}
    \newcommand{\frestr}{{\upharpoonright}}

    \DeclareMathOperator{\Fn}{Fn}

    \newcommand{\leqT}{\leq_{\mathrm{T}}}
    \newcommand{\eqT}{\cong_{\mathrm{T}}}
    \newcommand{\leqK}{\leq_{\mathrm{K}}}
    \newcommand\subsetdot{\mathrel{\ooalign{$\subset$\cr
  \hidewidth\hbox{$\cdot\mkern3mu$}\cr}}} 

    \newcommand{\Ed}{\mathbf{Ed}}
    \newcommand{\Rbf}{\mathbf{R}}

    \newcommand{\Srm}{\mathrm{S}}
    \newcommand{\rr}{\text{-}}

    \newcommand{\seqn}[2]{\langle#1:\allowbreak\,#2\rangle}

    \newcommand{\bigset}[2]{\left\{#1:\,#2\right\}}
    \newcommand{\set}[2]{\{#1:\allowbreak\,#2\}}

    \newcommand{\vfa}{\mathfrak{v}^\forall}
    \newcommand{\cfa}{\mathfrak{c}^\forall}
    \newcommand{\vxt}{\mathfrak{v}^\exists}
    \newcommand{\cxt}{\mathfrak{c}^\exists}

    \newcommand{\Lc}{\mathbf{Lc}}
    \newcommand{\aLc}{\mathbf{aLc}}
    \newcommand{\pLc}{\mathbf{pLc}}
    
    \newcommand{\pL}{\mathbf{pL}}

    \newcommand{\blc}{\mathfrak{b}^{\mathrm{Lc}}}
    \newcommand{\dlc}{\mathfrak{d}^{\mathrm{Lc}}}
    \newcommand{\balc}{\mathfrak{b}^{\mathrm{aLc}}}
    \newcommand{\dalc}{\mathfrak{d}^{\mathrm{aLc}}}

    \newcommand{\id}{\mathrm{id}}
    \newcommand{\dual}{\mathrm{d}}
    \newcommand{\dc}{\mathrm{dc}}
    \newcommand{\dfil}[1]{#1^\dual}

    \newcommand{\Baire}{{}^\omega\omega}

    \DeclareMathOperator{\cnt}{ct}

\newcommand{\Kat}{\mathrm{K}}
\newcommand{\KB}{\mathrm{KB}}
\newcommand{\mKat}{\mathrm{\overline{K}}}
\newcommand{\mKB}{\mathrm{\overline{KB}}}

\newcommand{\RB}{\mathrm{RB}}
\newcommand{\RK}{\mathrm{RK}}

\newcommand{\slalome}{\mathfrak{sl}_\mathrm{e}}
\newcommand{\slalomt}{\mathfrak{sl}_\mathrm{t}}

\newcommand{\ccup}{{\scriptscriptstyle\sqcup}}
\newcommand{\ccap}{{\scriptscriptstyle\sqcap}}

\makeatletter
\newcommand{\startlist}{\ \@beginparpenalty=10000}
\makeatother

\journal{Dissertationes Mathematicae}

\begin{document}

\begin{frontmatter}
\title{Slalom numbers} 
\author{Miguel A.~Cardona\fnref{fn1,fn2}}
\address{Einstein Institute of Mathematics,
The Hebrew University of Jerusalem, Givat Ram, Jerusalem, 9190401, Israel}
\address{Faculty of Engineering, 
Institución Universitaria Pascual Bravo, 
Calle 73 No. 73A - 226, Medellín, Colombia}

\ead{miguel.cardona@mail.huji.ac.il}
\author{Viera Gavalov\' a\fnref{fn1}}
\address{Department of Applied Mathematics and Bussiness Informatics, Faculty of Economics of the Technical University of Ko\v{s}ice, N\v emcovej 32, 040 01 Ko\v{s}ice, Slovakia}
\ead{viera.gavalova@tuke.sk}
\author{Diego A.~Mej\'ia\fnref{fn3}}
\address{Graduate School of System Informatics, Kobe University, 1-1 Rokkodai-cho, Nada-ku, Kobe, Hyogo 657-8501, Japan}
\ead{damejiag@people.kobe-u.ac.jp}
\author{Miroslav Repick\' y\fnref{fn1,fn4}}
\address{Mathematical Institute, Slovak Academy of Sciences, Gre\v s\' akova 6, 040 01 Ko\v sice, Slovak Republic.}
\ead{repicky@saske.sk}
\author{Jaroslav \v Supina\fnref{fn1,fn5}}
 \address{Institute of Mathematics, P.J. \v{S}af\'arik University in Ko\v sice, Jesenn\'a 5, 040 01 Ko\v{s}ice, Slovakia}
\ead{jaroslav.supina@upjs.sk}
\fntext[fn1]{Supported by the Slovak Research and Development Agency under the Contract no. APVV-20-0045.}
\fntext[fn2]{Supported by Pavol Jozef \v{S}af\'arik University in Ko\v{s}ice at a postdoctoral position, and by Israel Science Foundation for partial support of this research by grant 2320/23 (2023-2027).}
\fntext[fn3]{Supported by the Grant-in-Aid for Scientific Research (C) 23K03198, Japan Society for the Promotion of Science.}
\fntext[fn4]{Supported by the grant VEGA 2/0104/24 of the Slovak Grant Agency VEGA.}
\fntext[fn5]{Supported by the grant VEGA 1/0657/22 of the Slovak Grant Agency VEGA.}
\begin{abstract}

{\footnotesize The paper is an extensive and systematic study of cardinal invariants we call slalom numbers, describing the combinatorics of sequences of sets of natural numbers. Our general approach, based on relational systems, covers many such cardinal characteristics, including localization and anti-localization cardinals. We show that most of the slalom numbers are connected to topological selection principles, in particular, we obtain the representation of the uniformity of meager and the cofinality of measure. Considering instances of slalom numbers parametrized by ideals on natural numbers, we focus on monotonicity properties with respect to ideal orderings and computational formulas for the disjoint sum of ideals. Hence, we get such formulas for several pseudo-intersection numbers as well as for the bounding and dominating numbers parametrized with ideals. Based on the effect of adding a Cohen real, we get many consistent constellations of different values of slalom numbers.}   

\end{abstract}
\end{frontmatter}




\section{Introduction}\label{Sintro}

The notion of slalom, as a function $\omega\to[\omega]^{<\aleph_0}$, appeared implicitly in~\cite{Ba1984} to prove that the additivity of measure is below the additivity of category. Later, Bartoszy\'nski~\cite{Ba1987} introduced the notion explicitly. It has been proven to be highly important in studying combinatorial properties of measure and category, namely, to characterize and approximate classical cardinal invariants of the continuum, like those in Cicho\'n's diagram. In recent literature, the slalom-based cardinal invariants are usually called \emph{localization and anti-localization cardinals}. The papers \cite{CM19,CM23} are deep surveys on these invariants with a~long list of research sources dedicated to their studies, starting in the 80's and continuing to present-day results and modern treatment \cite{Mi1982,pawli,GS93,KS12,KO14,BrM,CKM}. We list some sources in \Cref{remnotation} to compare all known notations. The localization and anti-localization cardinals are instances of what we denominate \emph{slalom numbers} or \emph{slalom invariants}. 

Many classical cardinal invariants of the continuum have been studied, in a~more general form, parametrized by an ideal $J$ on the natural numbers, see e.g.~\cite{BM,FaSo10,Hr,BorFar,Bakke,FI_KW_22,RaSt,Su22} (throughout this text, we convey that an ideal contains all the finite sets). Very recently, the second and third authors~\cite{GaMe} developed a~version of the Lebesgue measure zero ideal $\Nwf$ and the $\sigma$-ideal $\Ewf$ generated by $F_\sigma$-measure zero sets modulo ideals on the natural numbers, and studied their associated cardinal invariants. These are denoted by $\Nwf_J$ and $\Nwf^*_J$, respectively, for any ideal $J$ on the natural numbers.

In this paper, we propose a~general framework to define slalom numbers parametrized with ideals. Within this framework, we prove general theorems about their connections and show several applications to particular cases that have already appeared in previous research, as well as consistency results. We also study \emph{selection principles} under this framework.

\subsubsection*{Instances of slalom numbers} 

Considering a~function $h\in\baire\omega$ and an ideal $J$ on natural numbers, the paper focuses on slalom invariants of the following form~\cite{SoDiz}:
\begin{align*}
    \slalomt(h,J) & =\min \bigset{|\calS|}{\mathcal{S}\subseteq\prod_{n\in \omega}[\omega]^{\leq h(n)}\wedge(\forall x\in \baire\omega)(\exists s\in\mathcal{S})\ \set{n\in \omega}{x(n)\notin s(n)}\in J},\\[0.1cm]
    \slalome(h,J) & =\min \bigset{|\calS|}{\mathcal{S}\subseteq\prod_{n\in \omega}[\omega]^{\leq h(n)}\wedge(\forall x\in \baire\omega)(\exists s\in\mathcal{S})\ \set{n\in \omega}{x(n)\in s(n)}\notin J}. 
\end{align*}
The classical instances of these numbers are obtained with $J=\Fin$, the ideal of finite sets of natural numbers. 
Well-known results on the latter slalom numbers by Bartoszy\'nski \cite{Ba1987,Ba1984} and Miller~\cite{Mi1982} state that $\cofn=\slalomt(g,\Fin)$ when $\lim_{n\to \infty}g(n)=\infty$, and $\nonm=\slalome(h,\Fin)$ when $h(n)\ge 1$ for all but finitely many $n\in\omega$. The dual forms of these slalom numbers characterize $\addn$ and $\covm$, as well.

Considering $\mathcal{S}\subseteq\baire\I$ for an ideal $I$ on $\omega$ instead of $\mathcal{S}\subseteq\prod_{n\in \omega}[\omega]^{\leq h(n)}$ in the definitions of $\slalomt(h,J)$ and $\slalome(h,J)$ above, we obtain $\slalomt(I,J)$ and $\slalome(I,J)$, see \Cref{sec:LpL} for details. Moreover, we study two more cardinals $\slalomt(\star,J)$, $\slalome(\star,J)$, allowing $\mathcal{S}$ to be more general \cite{SJ,SoSu,SV19,Su22}. Basic relations among the invariants are depicted in \Cref{simplediag}.

\begin{figure}[ht]
\[
\begin{array}{*{5}{@{}c}@{}}
\slalomt(\star,J)&{}\rightarrow{}
&\slalomt(I,J)&{}\rightarrow{}
&\slalomt(h,J)\\[2pt]
{\uparrow}&&{\uparrow}&&{\uparrow}\\[2pt]
\slalome(\star,J)&{}\rightarrow{}
&\slalome(I,J)&{}\rightarrow{}
&\slalome(h,J)
\end{array}
\]
\caption{Diagram of inequalities between slalom numbers. An arrow denotes that ZFC proves $\leq$.}\label{simplediag}
\end{figure}

These generalize more classical cardinal invariants, like the \emph{dominating number} $\dfrak = \slalomt(\Fin,\Fin)$, the \emph{bounding number} $\bfrak = \slalome(\Fin,\Fin)$, and the \emph{pseudo-intersection number} $\pfrak = \slalome(\star,\Fin)$ \cite{SJ,Su22}. We even obtain $\covm = \slalomt(\star,\Fin)$ (see \cite{Su22} and \Cref{sltstarcovM}). The ideal versions of the dominating and unbounding numbers are $\dfrak_J = \slalome(\Fin,J)$ and $\bfrak_J = \slalomt(\Fin,J)$, i.e., those with respect to the relation $x\leq^{J^\dual} y$ iff $\set{n\in\omega}{x(n)>y(n)}\in J$ on $\Baire$. The study of $\dfrak_J$ and $\bfrak_J$ dates back to at least the 1980s, when dominating numbers modulo ultrafilters (i.e., maximal ideals) were used by R.~Canjar \cite{Canjar2} in the context of nonstandard arithmetic (to study the cofinality of ultrapowers of the natural numbers). Further research appears, for instance, in \cite{BM,TsZd08,FaSo10}.

The connections between the slalom numbers introduced so far are illustrated in \Cref{BasicDiaIntro} (see \cite{SoDiz} and~\cite{Su22} for the diagram without the top row).

\begin{figure}[ht]
\begin{center}
\begin{tikzpicture}[scale=0.8]
\node (ale) at (-7, -3.5) {$\aleph_1$};
\node (a) at (-5, -3.5) {$\pp$};
\node (as) at (-5, -1) {$\sla{\I,\fin}$};
\node (b) at (-5, 1.5) {$\bb$};
\node (ba) at (-5, 4) {$\nonm$};
\node (aa) at (-2, -3.5) {$\slamg{\J}$};
\node (aas) at (-2, -1) {$\sla{\I,\J}$};
\node (bb) at (-2, 1.5) {$\bb_\J$};
\node (bba) at (-2, 4) {$\sla{h,\J}$};
\node (c) at (1, -3.5) {$\dslaml{\J}$};
\node (cs) at (1, -1) {$\dsla{\I,\J}$};
\node (f) at (1, 1.5) {$\dd_\J$};
\node (csa) at (1, 4) {$\dsla{h,\J}$};
\node (xpf) at (4, 1.5) {$\dd$};
\node (xpc) at (4, -3.5) {$\covm$};
\node (xpcs) at (4, -1) {$\dslago{\I}$};
\node (xpfa) at (4, 4) {$\cofn$};
\node (cont) at (6, 4) {$\cc$};
\foreach \from/\to in {aas/bb,aas/cs,cs/f,aas/cs} \draw [line width=.15cm,
white] (\from) -- (\to);
\foreach \from/\to in {ale/a,a/as, aa/aas, c/cs, b/bb, a/aa, aa/c, bb/f, as/b, aas/bb, cs/f, as/aas, aas/cs,f/xpf,cs/xpcs,c/xpc,xpcs/xpf,xpc/xpcs, b/ba,bb/bba,f/csa,xpf/xpfa,ba/bba,bba/csa,csa/xpfa,xpfa/cont} \draw [->] (\from) -- (\to);

\end{tikzpicture}
\end{center}
\caption{Relations among particular cases of slalom numbers.}
\label{BasicDiaIntro}
\end{figure}

We also look at combinatorial notions related with pseudo-intersection modulo ideals~\cite{BorFar,REP21a,REP21b,Su22}, and their counterparts, $\lfrak_\Kat(\star,J)$ and $\lfrak_\Kat(I,J)$, which are original in this paper. Here, $\leqK$~denotes the \emph{Kat\v{e}tov order.}
\begin{align*}
    \pfrak_\Kat(\star,J) & :=\min\set{|A|}{A\subseteq\PP(\omega)\text{ generates an ideal and } A\not\leq_\Kat\J},\\
    \lfrak_\Kat(\star,J) & :=\min\set{|A|}{A\subseteq\PP(\omega)\text{ generates an ideal and } A\not\leq_\Kat J^{\rm dc}},\\
    \pfrak_\Kat(I,J) & :=\min\set{|A|}{A\subseteq I\text{ and } A\not\leq_\Kat\J},\\
    \lfrak_\Kat(I,J) & :=\min\set{|A|}{A\subseteq I\text{ and } A\not\leq_\Kat J^{\rm dc}}.
\end{align*}

These invariants are upper bounds of the slalom numbers $\slalomt(I,J)$ and $\slalome(I,J)$, as illustrated in \Cref{fig:slpk}.

\begin{figure}[ht]
\[
\begin{array}{*{3}{@{}c}@{}}
\slalomt(I,J)&{}\rightarrow{}&\lfrak_\Kat(I,J)\\[2pt]
{\uparrow}&&{\uparrow}\\[2pt]
\slalome(I,J)&{}\rightarrow{}&\pfrak_\Kat(I,J)
\end{array}
\]
\caption{Further connections between slalom numbers}\label{fig:slpk}
\end{figure}

It is known that $\pfrak_\Kat(\star,\Fin) = \pfrak$ and $\pfrak_\Kat(I,\Fin) = \covst(I)$, where $\covst(I)$ is a~well-known idealized pseudo-{\hskip0pt}intersection number, introduced for maximal ideals (dually for ultrafilters) in~\cite{BreShe99} under the notation~$\pi\pp(\U)$, whose current notation comes from \cite{HeHru07}.

\subsection*{General framework} 

We develop a~general framework to define slalom numbers, 
using a~single general definition that describes all the slalom numbers presented before (and much more), see \Cref{def:Lc}. 
This framework allows general theorems that imply connections between slalom numbers, developed mainly in~\Cref{sec:LpL}. For instance, in \Cref{sec:partcases}, we derive monotonicity properties of slalom numbers with respect to several orders of ideals, like the \emph{Kat\v{e}tov order} and the \emph{Kat\v{e}tov-Blass order}. As a~consequence, $\star$-slalom numbers get characterized:

\begin{teorema}[{\Cref{G2}}]
Let $J$ be an ideal on $\omega$. Then
\begin{enumerate}[label=\rm(\alph*)]
\item $\slalomt(\star,J)=\min\set{\slalomt(I,J)}{I \text{ is an ideal on }\omega}$,

\item $\slalome(\star,J)=\min\set{\slalome(I,J)}{I \text{ is an ideal on }\omega}$.
\end{enumerate}
\end{teorema}

Concerning meager ideals, 
B.~Tsaban and L.~Zdomskyy \cite{TsZd08} have shown that $\bb_\J=\bb$ for any meager ideal~$\J$ on $\omega$. Later on, B.~Farkas and L.~Soukup \cite{FaSo10} have essentially complemented that by establishing $\dd_\J=\dd$ as well. Similarly, the second and third authors~\cite{GaMe} show that $\Nwf_J=\Nwf$ and $\Nwf^*_J = \Ewf$ when $J$ is meager. 
We prove similar results for slalom numbers. 
The following result summarizes \Cref{sltstarcovM}, \Cref{C2.7}, and \Cref{L3.4}. 

\begin{teorema}
Let $I$ and $J$ be ideals on $\omega$, $J$ with the Baire property, and $h\in{}^\omega\omega$. Then
\begin{enumerate}[label=\normalfont(\arabic*)]
   \item If $h\ge1$, then\/ $\slalome(h,J)=\slalome(h,\Fin)=\nonm$.

    \item If\/ $\lim_{n\in\omega}h(n)=\infty$, then\/ $\slalomt(h,J)=\slalomt(h,\Fin)=\cofn$.

    \item $\slalomt(I,J) = \slalomt(I,\Fin)$.

    \item $\slalomt(\star,J)= \slalomt(\star,\Fin)=\covm$.

    \item $\lfrak_\Kat(I,J)=\lfrak_\Kat(\star,J)=\infty$ (i.e., undefined).
\end{enumerate}
\end{teorema}

To prove this theorem, we use Mathias', Jalali-Naini's, and Talagrand's characterization of the Baire property with the Rudin-Blass order,\footnote{See, e.g.,~\cite{farah}. Note that an~ideal~$J$ on $\omega$ has the~Baire property if and only if $J$ is meager.} to which we apply our monotonicity results. The latter is deeply investigated in \Cref{sec:LpL} and \Cref{sec:partcases}.

\subsection*{Disjoint sum of ideals}

In \Cref{sec:sumI}, we prove characterizations of slalom numbers modulo \emph{disjoint sum of ideals}. Similar results appear in \cite{FI_KW_22,GaMe}.

\begin{teorema}[\Cref{L6.2}, \Cref{L6.3}, and \Cref{L6.4}]\label{thm:i1}
Let $I_0$, $I_1$, $J_0$ and $J_1$ be ideals on $\omega$. Then:
\begin{enumerate}[label=\rm(\alph*)]

\item $\slalomt(I_0\oplus I_1,J_0) = \slalomt(I_0\cap I_1,J_0) = \max\{\slalomt(I_0,J_0),\slalomt(I_1,J_0)\}$.

\item $\slalome(I_0\oplus I_1,\Fin) = \slalome(I_0\cap I_1,\Fin) = \max\{\slalome(I_0,\Fin),\slalomt(I_1,\Fin)\}$.

\item $\lfrak_\Kat(I_0\oplus I_1,J_0) = \lfrak_\Kat(I_0\cap I_1,J_0) = \max\{\lfrak_\Kat(I_0,J_0),\lfrak_\Kat(I_1,J_0)\}$.

\item $\covst(I_0\oplus I_1)= \covst(I_0\cap I_1) = 
\max\{\covst(I_0),\covst(I_1)\}$.

\item $\slalomt(I_0,J_0\cap J_1) = \slalomt(I_0,J_0\oplus J_1) = \max\{\slalomt(I_0,J_0),\slalomt(I_0,J_1)\}$.

\item $\slalome(I_0,J_0\cap J_1)\leq \slalome(I_0,J_0\oplus J_1) = \min\{\slalome(I_0,J_0),\slalome(I_0,J_1)\}$.

\item $\lfrak_\Kat(I_0,J_0\cap J_1) = \lfrak_\Kat(I_0,J_0\oplus J_1) = \max\{\lfrak_\Kat(I_0,J_0),\lfrak_\Kat(I_0,J_1)\}$.

\item $\pfrak_\Kat(I_0,J_0\cap J_1)\leq \pfrak_\Kat(I_0,J_0\oplus J_1) = \min\{\pfrak_\Kat(I_0,J_0),\pfrak_\Kat(I_0,J_1)\}$.

\item $\slalomt(\star,J_0\oplus J_1) = \slalomt(\star,J_0\cap J_1)  = \max\{\slalomt(\star,J_0),\slalomt(\star,J_1)\}$.

\item $\slalome(\star,J_0\oplus J_1) = \min\{\slalome(\star,J_0),\slalome(\star,J_1)\}$.

\item $\lfrak_\Kat(\star,J_0\oplus J_1) = \lfrak_\Kat(\star,J_0\cap J_1) = \max\{\lfrak_\Kat(\star,J_0),\lfrak_\Kat(\star,J_1)\}$.

\item $\pfrak_\Kat(\star,J_0\oplus J_1) = \min\{\pfrak_\Kat(\star,J_0),\pfrak_\Kat(\star,J_1)\}$.

\item $\dfrak_{J_0\oplus J_1}=
\max\{\dfrak_{J_0},\dfrak_{J_1}\}$.

\item $\bfrak_{J_0\oplus J_1}=
\min\{\bfrak_{J_0},\bfrak_{J_1}\}$.
\end{enumerate}
\end{teorema}

These results are relevant to obtain examples of ideals with different slalom numbers as part of our consistency results.

\subsection*{Selection principles}

Systematic studies of selection principles were initiated in \cite{Comb1,Comb2}. The latter presents the list of uniformity numbers (or \emph{critial cardinalities}) of studied selection principles in a~form of standard cardinal invariants. For more recent sources, see \cite{BStr,Osip18,Buky18}. In \cite{SJ,SoSu,Su22}, it was pointed out that slalom numbers tend to be uniformity numbers of some selection principles. In \Cref{sec:selection}, we propose a~very general selection principle and derive its critical cardinality using the general definition of slalom number (\Cref{G}). Many well-known and so far-unknown critical cardinalities are derived as direct consequences of this result. Below we present the new results. To show the flavor of the result, let us define selection principle $\mathrm{S}_1(\Gamma_h,\OO)$ for a~function $h\in\Baire$, introduced first in \cite{SoDiz}. A~topological space~$X$ \emph{satisfies the selection principle} $\mathrm{S}_1(\Gamma_h,\OO)$ if for each $\seqn{\seqn{V_{n,m}}{m\in\omega}}{n\in\omega}$ with $V_{n,m}$ being open subsets of~$X$ such that $|\set{m}{x\not\in V_{n,m}}|\leq h(n)$ for each $x\in X$, there is a $d\in\Baire$ with $\set{V_{n,d(n)}}{n\in\omega}$ being an open cover of~$X$.

The following particular cases of \Cref{G} are formulated in \Cref{critical}.

\begin{teorema}
If $h\in\Baire$ then\/ 
$\non{\mathrm{S}_1(\Gamma_{b,h},\GammaB{{J}})}=\cLambda{b,h}{{J}}$ and\/ $\non{\mathrm{S}_1(\Gamma_{b,h},\LambdaB{J})}=\tsl{b,h}{{J}}$. As a~consequence,
\begin{alignat*}{2}
    \non{\mathrm{S}_1(\Gamma_h,\Gamma)} & = \nonm &\quad&\text{when $h\geq^*1$,\quad and}\\
    \non{\mathrm{S}_1(\Gamma_h,\OO)} & =\cofn & & \text{when $h\to\infty$.}
\end{alignat*}
\end{teorema}

The latter two equalities as well as \Cref{SjednaABcard1} were obtained in the frame of \cite{SoDiz}. The same applies to the second part of \Cref{sec:selection} that treats properties of topological spaces possessing the investigated selection principles.

\begin{figure}[ht]
\centering
\begin{tikzpicture}[scale=0.8]
\small{
\node (a1) at (-4.5, -3) {$\schema{\Omega}{\Gamma}$};
\node (a2) at (0, -3) {$\schema{\Omega}{\GammaB{{J}}}$};
\node (a3) at (4.5, -3) {$\schema{\Omega}{\LambdaB{{J}}}$};
\node (a4) at (9, -3) {$\schema{\mathcal{O}}{\mathcal{O}}$};
\node (aa1) at (-4.5, -3.7) {\footnotesize$\mathfrak{p}$};
\node (aa2) at (0, -3.7) {\footnotesize$\cLambda{\star}{{J}}$};
\node (aa3) at (4.5, -3.7) {\footnotesize$\tsl{\star}{{J}}$};
\node (aa4) at (9, -3.7) {\footnotesize$\covm$};

\node (b1) at (-4.5, 0) {$\schema{\GammaB{{I}}}{\Gamma}$};
\node (b2) at (0, 0) {$\schema{\GammaB{{I}}}{\GammaB{{J}}}$};
\node (b3) at (4.5, 0) {$\schema{\GammaB{{I}}}{\LambdaB{{J}}}$};
\node (b4) at (9, 0) {$\schema{\GammaB{{I}}}{\mathcal{O}}$};
\node (bb1) at (-4.5, -0.7) {\footnotesize$\min \{\covh{{I}}, \mathfrak{b}\}$};
\node (bb2) at (0, -0.7) {\footnotesize$\cLambda{{I}}{{J}}$};
\node (bb3) at (4.5, -0.7) {\footnotesize$\tsl{{I}}{{J}}$};
\node (bb4) at (9, -0.7) {\footnotesize$\tsl{{I}}{\fin}$};

\node (c1) at (-4.5, 3) {$\schema{\Gamma}{\Gamma}$};
\node (c2) at (0, 3) {$\schema{\Gamma}{\GammaB{{J}}}$};
\node (c3) at (4.5, 3) {$\schema{\Gamma}{\LambdaB{{J}}}$};
\node (c4) at (9, 3) {$\schema{\Gamma}{\mathcal{O}}$};
\node (cc1) at (-4.5, 2.3) {\footnotesize$\mathfrak{b}$};
\node (cc2) at (0, 2.3) {\footnotesize$\be_{{J}}$};
\node (cc3) at (4.5, 2.3) {\footnotesize$\de_{{J}}$};
\node (cc4) at (9, 2.3) {\footnotesize$\mathfrak{d}$};

\node (d1) at (-4.5, 6) {$\schema{\Gammah{h}}{\Gamma}$};
\node (d2) at (0, 6) {$\schema{\Gammah{h}}{\GammaB{{J}}}$};
\node (d3) at (4.5, 6) {$\schema{\Gammah{h}}{\LambdaB{{J}}}$};
\node (d4) at (9, 6) {$\schema{\Gammah{h}}{\mathcal{O}}$};
\node (dd1) at (-4.5, 5.3) {\footnotesize$\nonm$};
\node (dd2) at (0, 5.3) {\footnotesize$\cLambda{h}{{J}}$};
\node (dd3) at (4.5, 5.3) {\footnotesize$\tsl{h}{{J}}$};
\node (dd4) at (9, 5.3) {\footnotesize$\cofn$};
}
\foreach \from/\to in {a1/a2, a2/a3, a3/a4, b1/b2, b2/b3, b3/b4, c1/c2, c2/c3, c3/c4, d1/d2,d2/d3, d3/d4, a1/bb1,b1/cc1, c1/dd1, a2/bb2, b2/cc2, c2/dd2, a3/bb3, b3/cc3, c3/dd3, a4/bb4, b4/cc4, c4/dd4}
\draw [->] (\from) -- (\to);
\end{tikzpicture}
\caption{Critical cardinality of some selection principles.}
\label{SjednaABcard1}
\end{figure}

\subsection*{Consistency results} 

In \Cref{sec:forcing}, we construct forcing models to prove consistency constellations of our slalom numbers. These models are motivated by Canjar's result~\cite{Canjar2}, which states that after adding $\lambda$ many Cohen reals, there exists a (maximal) ideal $J_\kappa$ for any uncountable regular cardinal $\kappa\leq\lambda$ such that $\bfrak_{J_\kappa}=\dfrak_{J_\kappa}=\kappa$. 
We expand this result to show the effect of Cohen reals on the slalom numbers parametrized by ideals. This allows us to present a~strong iteration theorem to control slalom numbers using Cohen reals added at intermediate steps (\Cref{seq:sn}).
One consequence is that, in Cohen model, we have many different slalom numbers.

\begin{teorema}[\Cref{aplcohen}]\label{mCohen}
    Let $\lambda = \lambda^{\aleph_0}$ be an infinite cardinal. Then, after adding $\lambda$-many Cohen reals:
    \begin{enumerate}[label=\normalfont(\alph*)]
        \item\label{mCohena} Any uncountable regular cardinal $\kappa$ satisfying $\lambda^{<\kappa} = \lambda$ is a~slalom number of the form\/ $\slalome(\star,J)= \slalomt^\perp(h,J)=\slalomt(h,J)$ (for all $J$-unbounded $h$) for some maximal ideal $J$ on $\omega$. (This corresponds to the two central columns of \Cref{Cohen_var_val:b})

        \item\label{mCohenb} For any regular $\aleph_1\leq\kappa_1\leq\kappa_2$, if $\lambda^{<\kappa_2} = \lambda$ then there is some ideal $J$ on $\omega$ such that\/ $\slalome(\star,J) = \slalomt^\perp(h,J) = \slalome(h,J) = \kappa_1$ and\/ $\slalomt(\star,J) = \slalome^\perp(h,J) = \slalomt(h,J) =\kappa_2$ for all $h\in\Baire$ satisfying $\lim^J h = \infty$. In particular, the four columns of \Cref{Cohen_var_val:b} can be pairwise different.
    \end{enumerate}

\begin{figure}[ht]
\centering
\begin{tikzpicture}[scale=0.6,every node/.style={scale=0.6}]
\node (ale) at (-7, -3.5) {$\aleph_1$};
\node (a) at (-5, -3.5) {$\pp$};
\node (as) at (-5, -1) {$\sla{\I,\fin}$};
\node (b) at (-5, 1.5) {$\bb$};
\node (ba) at (-5, 4) {$\nonm$};
\node (aa) at (-2, -3.5) {$\slamg{\J}$};
\node (aas) at (-2, -1) {$\sla{\I,\J}$};
\node (bb) at (-2, 1.5) {$\bb_\J$};
\node (bba) at (-2, 4) {$\sla{h,\J}$};
\node (c) at (1, -3.5) {$\dslaml{\J}$};
\node (cs) at (1, -1) {$\dsla{\I,\J}$};
\node (f) at (1, 1.5) {$\dd_\J$};
\node (csa) at (1, 4) {$\dsla{h,\J}$};
\node (xpf) at (4, 1.5) {$\dd$};
\node (xpc) at (4, -3.5) {$\covm$};
\node (xpcs) at (4, -1) {$\dslago{\I}$};
\node (xpfa) at (4, 4) {$\cofn$};
\node (cont) at (6, 4) {$\cc$};
\foreach \from/\to in {aas/bb,aas/cs,cs/f,aas/cs} \draw [line width=.15cm,
white] (\from) -- (\to);
\foreach \from/\to in {ale/a,a/as, aa/aas, c/cs, b/bb, a/aa, aa/c, bb/f, as/b, aas/bb, cs/f, as/aas, aas/cs,f/xpf,cs/xpcs,c/xpc,xpcs/xpf,xpc/xpcs, b/ba,bb/bba,f/csa,xpf/xpfa,ba/bba,bba/csa,csa/xpfa,xpfa/cont} \draw [->] (\from) -- (\to);

\draw[color=blue,line width=.05cm] (-3.5,-4)--(-3.5,5);
\draw[color=blue,line width=.05cm] (2.5,-4)--(2.5,5);
\draw[circle, fill=cadmiumorange,color=cadmiumorange] (-5.8,0) circle (0.4);
\draw[circle, fill=cadmiumorange,color=cadmiumorange] (4.8,0) circle (0.4);
\draw[circle, fill=cadmiumorange,color=cadmiumorange] (-0.5,0) circle (0.4);
\node at (-0.5,0) {$\kappa$};
\node at (4.8,0) {$\lambda$};
\node at (-5.8,0) {$\aleph_1$};
\end{tikzpicture}
\begin{tikzpicture}[scale=0.6,every node/.style={scale=0.6}]
\node (ale) at (-7, -3.5) {$\aleph_1$};
\node (a) at (-5, -3.5) {$\pp$};
\node (as) at (-5, -1) {$\sla{\I,\fin}$};
\node (b) at (-5, 1.5) {$\bb$};
\node (ba) at (-5, 4) {$\nonm$};
\node (aa) at (-2, -3.5) {$\slamg{\J}$};
\node (aas) at (-2, -1) {$\sla{\I,\J}$};
\node (bb) at (-2, 1.5) {$\bb_\J$};
\node (bba) at (-2, 4) {$\sla{h,\J}$};
\node (c) at (1, -3.5) {$\dslaml{\J}$};
\node (cs) at (1, -1) {$\dsla{\I,\J}$};
\node (f) at (1, 1.5) {$\dd_\J$};
\node (csa) at (1, 4) {$\dsla{h,\J}$};
\node (xpf) at (4, 1.5) {$\dd$};
\node (xpc) at (4, -3.5) {$\covm$};
\node (xpcs) at (4, -1) {$\dslago{\I}$};
\node (xpfa) at (4, 4) {$\cofn$};
\node (cont) at (6, 4) {$\cc$};
\foreach \from/\to in {aas/bb,aas/cs,cs/f,aas/cs} \draw [line width=.15cm,
white] (\from) -- (\to);
\foreach \from/\to in {ale/a,a/as, aa/aas, c/cs, b/bb, a/aa, aa/c, bb/f, as/b, aas/bb, cs/f, as/aas, aas/cs,f/xpf,cs/xpcs,c/xpc,xpcs/xpf,xpc/xpcs, b/ba,bb/bba,f/csa,xpf/xpfa,ba/bba,bba/csa,csa/xpfa,xpfa/cont} \draw [->] (\from) -- (\to);

\draw[color=blue,line width=.05cm] (-3.5,-4)--(-3.5,5);
\draw[color=blue,line width=.05cm] (-0.5,-4)--(-0.5,5);
\draw[color=blue,line width=.05cm] (2.5,-4)--(2.5,5);
\draw[circle, fill=cadmiumorange,color=cadmiumorange] (-5.8,0) circle (0.4);
\draw[circle, fill=cadmiumorange,color=cadmiumorange] (4.8,0) circle (0.4);
\draw[circle, fill=cadmiumorange,color=cadmiumorange] (-2.65,0) circle (0.4);
\draw[circle, fill=cadmiumorange,color=cadmiumorange] (0.35,0) circle (0.4);
\node at (-2.65,0) {$\kappa_1$};
\node at (0.35,0) {$\kappa_2$};
\node at (4.8,0) {$\lambda$};
\node at (-5.8,0) {$\aleph_1$};
\end{tikzpicture}
\caption{Effect of adding $\lambda$~many Cohen reals.}
\label{Cohen_var_val:b}
\end{figure}
\end{teorema}

The ideals satisfying~\ref{mCohenb} are obtained as sums of ideals from~\ref{mCohena}, where we use the characterization of the slalom numbers for sum of ideals (see~\Cref{thm:i1}). 

The general result \Cref{seq:sn} can be applied to any iteration adding Cohen reals. For more applications, we consider models obtained by FS (finite support) iterations and more sophisticated techniques like matrix iterations and coherent systems of finite support iterations~\cite{FFMM,mejvert}. We bring forcing constructions from~\cite{mejiamatrix,mejvert,BCM,GKMSsplit} and use our powerful theorem to prove the behavior of slalom numbers in these models. 

In the final section, we present some open problems and discussions.

\section{Preliminaries}

We introduce basic notation.
\begin{enumerate}[label=\rm(T\arabic*)]
    \item For $A\subseteq \PP(M)$, denote $A^{\rm c}:= \pts(M) \menos A$ and $A^\dual := \set{M\smallsetminus a}{a\in A}$.
    
    \item An \emph{ideal on $M$} is a~family $I\subseteq \PP(M)$ that is closed under taking subsets, closed under finite unions, containing all finite subsets of~$M$ but with $M\notin\I$. A~$\sigma$-ideal on~$M$, usually considered on a~Polish space~$M$, is an ideal on~$M$ which is closed under countable unions. 
    
    \item We say that $A$ has the \emph{finite union property (FUP)} whenever there is an ideal $I$ on~$M$ such that $A\subseteq I$.
    
    \item For an ideal $J$ on $M$, denote $J^+:= J^{\rm c} =\Pcal(M)\smallsetminus J$ (the collection of \emph{$J$-positive sets}),
    $J^\dual$ is the dual filter of $I$ and
    $J^{\rm dc}=\Pcal(M)\smallsetminus J^\dual = \set{M\smallsetminus a}{a\in J^+}$. We often extend this notation to arbitrary collections $J\subseteq \pts(M)$ that are not ideals.

    \item For a~function $\varphi\colon M\to N$ and $A\subseteq\pts(M)$, denote
    $\varphi^\rightarrow(A):=\set{y\subseteq N}{\varphi^{-1}\llbracket y\rrbracket\in A}$.

    \item Let $\sqsubset$ be a~relation. If $x$ and $y$ are two functions with the same domain $w$, denote 
    $\|x \sqsubset y\|:=\set{i\in w}{x(i) \sqsubset y(i)}.$
\end{enumerate}

We say that $\Rbf=\la X, Y, {\sqsubset}\ra$ is a~\textit{relational system} if it consists of two non-empty sets $X$ and $Y$ and a~relation~$\sqsubset$.
\begin{enumerate}[label=(\arabic*)]
    \item A~set $F\subseteq X$ is \emph{$\Rbf$-bounded} if $(\exists y\in Y) (\forall x\in F)\ x \sqsubset y$. 
    \item A~set $D\subseteq Y$ is \emph{$\Rbf$-dominating} if $(\forall x\in X) (\exists y\in D)\ x \sqsubset y$. 
\end{enumerate}

We associate two cardinal characteristics with this relational system $\Rbf$: 
\begin{align*}
\bfrak(\Rbf)&:=\min\set{|F|}{\text{$F\subseteq X$ is $\Rbf$-unbounded}},
&&\text{the \emph{unbounding number of\/ $\Rbf$}, and}\\
\dfrak(\Rbf)&:=\min\set{|D|}{\text{$D\subseteq Y$ is $\Rbf$-dominating}},
&&\text{the \emph{dominating number of\/ $\Rbf$}.}
\end{align*}
The dual of $\Rbf$ is defined by $\Rbf^\perp:=\la Y, X, {\sqsubset^\perp}\ra$ where $y\sqsubset^\perp x$ iff $x\not\sqsubset y$. Note that $\bfrak(\Rbf^\perp)=\dfrak(\Rbf)$ and $\dfrak(\Rbf^\perp)=\bfrak(\Rbf)$.

The cardinal $\bfrak(\Rbf)$ may be undefined, in which case we write $\bfrak(\Rbf) = \infty$, likewise for $\dfrak(\Rbf)$. Concretely, $\bfrak(\Rbf) = \infty$ iff $\dfrak(\Rbf) =1$; and $\dfrak(\Rbf)= \infty$ iff $\bfrak(\Rbf) =1$.

The cardinal characteristics associated with an ideal can be characterized by relational systems.

\begin{example}\label{exm:Iwf}
For $\Iwf\subseteq\pts(X)$, define the relational systems:
\begin{enumerate}[label=\rm(\arabic*)]
    \item $\Iwf:=\la\Iwf,\Iwf,{\subseteq}\ra$, which is a~directed partial order when $\Iwf$ is closed under unions (e.g.\ an ideal).

    \item $\Cbf_\Iwf:=\la X,\Iwf,{\in}\ra$.
\end{enumerate}
Whenever $\Iwf$ is an ideal on $X$,
\begin{multicols}{2}
\begin{enumerate}[label=\rm(\alph*)]
    \item $\bfrak(\Iwf)=\add\Iwf$, \emph{the additivity of $\Iwf$}.

    \item $\dfrak(\Iwf)=\cof\Iwf$, \emph{the cofinality of~$\Iwf$}.
    
    \item $\dfrak(\Cbf_\Iwf)=\cov\Iwf$, \emph{the covering of~$\Iwf$}. 

    \item $\bfrak(\Cbf_\Iwf)=\non\Iwf$, \emph{the uniformity of $\Iwf$}.
\end{enumerate}
\end{multicols}
\end{example}

The Tukey connection is a~practical tool to determine relations between cardinal characteristics. 
Let $\Rbf=\la X,Y,{\sqsubset}\ra$ and $\Rbf'=\la X',Y',{\sqsubset'}\ra$ be two relational systems. We say that 
$(\Psi_-,\Psi_+)\colon\Rbf\to\Rbf'$ 
is a~\emph{Tukey connection from $\Rbf$ into $\Rbf'$} if 
 $\Psi_-\colon X\to X'$ and $\Psi_+\colon Y'\to Y$ are functions such that
 \[
 (\forall x\in X)(\forall y'\in Y')\ 
 \Psi_-(x) \sqsubset' y' \Rightarrow x \sqsubset \Psi_+(y').
 \]
The \emph{Tukey order} between relational systems is defined by
$\Rbf\leqT\Rbf'$ iff there is a~Tukey connection from $\Rbf$ into $\Rbf'$. \emph{Tukey equivalence} is defined by $\Rbf\eqT\Rbf'$ iff $\Rbf\leqT\Rbf'$ and $\Rbf'\leqT\Rbf$.

\begin{fct}\label{fct:Tukey}
Assume that\/ $\Rbf=\la X,Y,{\sqsubset}\ra$ and\/ $\Rbf'=\la X',Y',{\sqsubset'}\ra$ are relational systems and that\/ $(\Psi_-,\Psi_+)\colon\allowbreak \Rbf\to\Rbf'$ is a~Tukey connection.
\begin{enumerate}[label=\rm(\alph*)]
    \item $(\Psi_+,\Psi_-)\colon(\Rbf')^\perp\to\Rbf^\perp$ is a~Tukey connection.
    \item If $E\subseteq X$ is\/ $\Rbf$-unbounded then\/ $\Psi_-[E]$ is\/ $\Rbf'$-unbounded.
    \item If $D'\subseteq Y'$ is\/ $\Rbf'$-dominating, then\/ $\Psi_+[D']$ is\/ $\Rbf$-dominating.
    \qed
\end{enumerate}
\end{fct}

\begin{corollary}\label{cor:Tukeyval}
Let\/ $\Rbf=\la X, Y,{\sqsubset}\ra$ and\/ $\Rbf'=\la X', Y',{\sqsubset'}\ra$ be relational systems. Then
\begin{enumerate}[label=\rm(\alph*)]
    \item $\Rbf\leqT\Rbf'$ implies\/ $(\Rbf')^\perp\leqT\Rbf^\perp$.
    \item\label{Tukeyval:b}
    $\Rbf\leqT\Rbf'$ implies\/ $\bfrak(\Rbf')\leq\bfrak(\Rbf)$ and\/ $\dfrak(\Rbf)\leq\dfrak(\Rbf')$.
    \item $\Rbf\eqT\Rbf'$ implies\/ $\bfrak(\Rbf')=\bfrak(\Rbf)$ and\/ $\dfrak(\Rbf)=\dfrak(\Rbf')$.
    \qed
\end{enumerate}
\end{corollary}

We use a~couple of types of products of relational systems for our main results. 

\begin{definition}\label{def:prodrel}
Let $\overline{\Rbf}=\seqn{\Rbf_i}{i\in K}$ be a~sequence of relational systems $\Rbf_i = \la X_i,Y_i,{\sqsubset_i}\ra$. Define:
\begin{enumerate}[label=\rm(P\arabic*)]
    \item $\bigotimes\overline{\Rbf} = \bigotimes_{i\in K} \Rbf_i:=\left\la\prod_{i\in K}X_i, \prod_{i\in K}Y_i, {\sqsubset^\otimes}\right\ra$ where $x \sqsubset^\otimes y$ iff $x_i\sqsubset_i y_i$ for all $i\in K$.

    \item $\bigboxtimes\overline{\Rbf} = \bigboxtimes_{i\in K} \Rbf_i:=\left\la\prod_{i\in K}X_i, \prod_{i\in K}Y_i, {\sqsubset^\boxtimes}\right\ra$ where $x \sqsubset^\boxtimes y$ iff $x_i\sqsubset_i y_i$ for some $i\in K$.
\end{enumerate}
For two relational systems $\Rbf$ and $\Rbf'$, write $\Rbf\otimes \Rbf'$ and $\Rbf\boxtimes \Rbf'$. 

When $\Rbf_i = \Rbf$ for all $i\in K$, we write ${}^K\Rbf := \bigotimes \overline{\Rbf}$.

Notice that $\bigboxtimes\overline{\Rbf} =\left(\bigotimes_{i\in K}\Rbf_i^\perp\right){}^\perp$.
\end{definition}

\begin{fact}[{\cite{CarMej23}}]\label{products}
Let\/ $\overline{\Rbf}$ be as in \Cref{def:prodrel}. Then
\begin{enumerate}[label=\rm(\alph*)]
    \item $\sup_{i\in K}\dfrak(\Rbf_i)\leq\dfrak(\bigotimes \overline{\Rbf})\leq\prod_{i\in K}\dfrak(\Rbf_i)$ and\/ $\bfrak(\bigotimes \overline{\Rbf})=\min_{i\in K}\bfrak(\Rbf_i)$.
    \item $\dfrak(\bigboxtimes \overline{\Rbf})=\min_{i\in K}\dfrak(\Rbf_i)$ and\/ $\sup_{i\in K}\bfrak(\Rbf_i)\leq \bfrak(\bigboxtimes \overline{\Rbf})\leq \prod_{i\in K}\bfrak(\Rbf_i)$.
    \qed
\end{enumerate} 
\end{fact}

\section{Slalom numbers}\label{sec:LpL}

In this section, we present a general framework for slalom numbers and prove general theorems for connections between them.

\begin{definition}\label{def:Lc}
Let $A\subseteq \pts(a)$ and let $\sqsubset$ be a~relation. For two functions $x$ and $s$ with domain $a$, define the relation
\[
x \sqsubset^A s
\text{ iff }
\|x \sqsubset s\|=\set{i\in a}{x(i)\sqsubset s(i)}\in A.
\]
Most of the time we use this notation when $a=\omega$.

For non-empty sets $D$ and $E$ of functions with domain $\omega$ (or some other set $a$ in general), we consider the relational system $\la D,E,{\in^A}\ra$ and denote its associated cardinal characteristics by 
\[
\dfrak_A(D,E):=\dfrak(D,E,{\in^A})\text{ and }\bfrak_A(D,E):=\bfrak(D,E,{\in^A}).
\]
We will refer to any cardinal characteristic of this form as a~\emph{slalom number}. Any function $s\in E$ is called an \emph{slalom} since $s(i)$ is a set typically contained in some domain where $x(i)$ lives for all $x\in D$. The expresion ``$x \in^A s$" indicates that $x$ passes through the slalom $s$ precisely at coordinates on a set in $A$. 
We concentrate mostly on the $\dfrak$-numbers and the following situation:
\begin{enumerate}[label=\rm(\arabic*)]
    \item\label{not:b} For some sequence $b=\seqn{b(i)}{i<\omega}$ of non-empty sets, we consider $D\subseteq \prod b:=\prod_{n<\omega}b(n)$. Most of the time $D = \prod b$ but there are some exceptions, e.g.\ when $D$ is the set of $\oneone$ functions in $\Baire$. We typically assume $E\subseteq\prod_{i<\omega}\pts(b(i))$ for simplicity, but this is not required because $\la D, E, {\in^A}\ra \eqT \la D, E', {\in^A}\ra$ where $E'$ is the collection of all sequences of the form $\seqn{s(n)\cap b(n)}{n<\omega}$ for some $s\in E$, see \Cref{restE}.

    In the case $D = \prod b$, we replace $D$ by $b$ in the notation for the relational system and its cardinal characteristics, i.e., $\la b, E,{\in^A}\ra := \la \prod b, E,{\in^A}\ra$, $\dfrak_A(b,E):= \dfrak_A(\prod b,E)$, and likewise for the $\bfrak$-number. When $b$ is the constant sequence of a~set~$a$, we replace $b$ by $a$ in the previous notation, e.g.\ we write 
    $\dfrak_A(a,E)$ and $\bfrak_A(a,E)$; when $a=\omega$, we omit $\omega$ in the cardinal characteristics, i.e., we just write $\dfrak_A(E)$ and~$\bfrak_A(E)$.

   We are interested in the case when $b(i)$ is countable for all $i<\omega$, but we do not need to assume this all the time.

    \item\label{not:ideal} The set $A$ is associated to an ideal $J$ on the natural numbers. In fact, we are only interested in the case when $A = \dfil{J}$ (the dual filter of $J$), or $A = J^+$ (the collection of positive sets). Considering each case, we denote:
    \begin{align}\label{eq:Lc}
       \Lc_J(D,E) & := \la D,E,{\in^{\dfil{J}}}\ra,  &
       \pLc_J(D,E) & := \la D,E,{\in^{J^+}}\ra, \nonumber\\
       \slalomt(D,E,J) & := \dfrak_{\dfil{J}}(D,E), &
       \slalome(D,E,J) & := \dfrak_{J^+}(D,E),\\
       \slalomt^\perp(D,E,J) & := \bfrak_{\dfil{J}}(D,E), &
       \slalome^\perp(D,E,J) & := \bfrak_{J^+}(D,E). \nonumber
    \end{align}
    We call these relational systems \emph{localization} and \emph{pseudo-localization}, respectively. 
    Like in~\ref{not:b}, we replace $D$ by $b$ when $D=\prod b$ and, in addition, we omit $b$ when it is the constant sequence $\omega$, i.e., we write $\slalome(E,J)$, $\slalomt(E,J)$. We allow this notation when $J$ is not an ideal.

    \item 
    For the set $E$, we consider the following when $D\subseteq\prod b$:
    \begin{enumerate}[label=(\arabic{enumi}\alph*)]
        \item $E= \prod\bar{I} = \prod_{n<\omega}I_n$ when $\bar{I}=\seqn{I_n}{n<\omega}$ is a~sequence of ideals on $\omega$, or more generally, each~$I_n\subseteq\pts(b(n))$ for some set $b(n)$ (not necessarily an ideal). In this case, we replace~$E$ by~$\bar{I}$ in~\eqref{eq:Lc}, and by~$I$ in case $\bar{I}$~is the constant sequence of $I$, i.e., we write $\slalome(D,I,J)$ and $\slalomt(D,I,J)$. The most common particular case is the one with $D=\Baire$, i.e., $\slalome(I,J)$ and $\slalomt(I,J)$.

        \item For some $h\in\Baire$, let $E=\Scal(b,h):=\prod_{n\in \omega} [b(n)]^{\leq h(n)}$. We replace $E$ by~$h$ in~\eqref{eq:Lc}, i.e., we write $\slalome(D,h,J)$ and $\slalomt(D,h,J)$. We often consider $\slalome(h,J)$ and $\slalomt(h,J)$ with $D=\Baire$. Another relevant particular case is 
        \begin{equation*}
        \slalomt(b,h,J) = \slalomt\left(\prod b,\Scal(b,h),J\right)\text{ and }\slalome(b,h,J) = \slalome\left(\prod b,\Scal(b,h),J\right).    
        \end{equation*}
         When $b$~is the constant sequence of a~set~$a$, we write $\Scal(a,h)$; when $b$~is the constant sequence of~$\omega$, we write~$\Scal(h)$. Notice that $\la D,\Swf(b,h), {\in^A}\ra \eqT \la D, \Swf(a,h),{\in^A}\ra$ when $a$ contains $\bigcup_{i<\omega}b(i)$, see \Cref{restE}.

        \item\label{it:constI} For some ideal $I$ on $\omega$, or on some set $a$, consider the collection $\cnt(I)$ of constant functions in ${}^\omega I$ (or in ${}^a I$). Denote:
        \begin{align}\label{eq:LK}
        \Lbf_{I,J}(D) & := \Lc_J(D,\cnt(I)),  &
        \pL_{I,J}(D) & := \pLc_J(D,\cnt(I)),\\
        \lfrak_D(I,J) & := \slalomt(D,\cnt(I),J), &
        \pfrak_D(I,J) & := \slalome(D,\cnt(I),J).
        \nonumber
        \end{align}
    Like in~\ref{not:b}, for the relational systems we omit $D$ when it is~$\Baire$, but for the cardinal characteristics $\lfrak_D(I,J)$ and $\pfrak_D(I,J)$, 
    we use a~different notation, usually associated with some property:
    in the case $D=\Baire$, we denote these numbers by $\lfrak_\Kat(I,J)$ and $\pK(I,J)$, respectively (in connection with the Kat\v{e}tov ordering); in the case that $D$ is the set of all finite-to-one functions from $\omega$ into~$\omega$, we write $\lfrak_{\rm KB}(I,J)$ and $\pfrak_{\rm KB}(I,J)$ (in connection with the Kat\v{e}tov-Blass ordering); and when $D$ is the set of all one-to-one functions from $\omega$ into $\omega$, we write $\lfrak_\oneone(I,J)$ and $\pfrak_\oneone(I,J)$, respectively. This notation appears in \cite{BorFar,Su22,BaSuZd}.
    \end{enumerate}
\end{enumerate}
We sometimes extend the notation presented above for arbitrary families $I$ and $J$ instead of ideals and consider similar definitions where the domain of the slaloms is some other set instead of $\omega$. Likewise for the slalom number defined below.

A more general approach to define slalom numbers like $\dfrak_A(D,E)$ is the following: for some collection (or property) $\Prop$ of sets of functions with domain $\omega$, define
\[
\dfrak_A(D,\Prop) := \min\set{|S|}{S\text{ satisfies }\Prop \text{ and } (\forall x\in D) (\exists s\in S)\  x\in^A s}.
\]
When $\Prop$ is ``$S \subseteq E$'', the cardinal above is just $\dfrak_A(D,E)$. Note that, in general, we do not have a~relational system for this more general setting.
In relation with~\ref{not:ideal}, we denote
    \begin{align*}
       \slalomt(D,\Prop,J) & := \dfrak_{\dfil{J}}(D,\Prop), &
       \slalome(D,\Prop,J) & := \dfrak_{J^+}(D,\Prop)
    \end{align*}
and, like in~\ref{not:b}, we use $b$ or omit $D$ when it is $\Baire$.
Two relevant properties are $\star$ and $\Prop_{\cnt}$:
\begin{itemize}
    \item
    A set $S$ satisfies~$\star$ if $S\subseteq \prod b$ and, for any $i<\omega$, the collection $\set{s(i)}{s\in S}$ has the FUP in $b(i)$.
    
    \item
    A set $S$ satisfies $\Prop_{\cnt}$ if  $S\subseteq{}^\omega\pts(a)$ is a~set of constant $\omega$-sequences and $S$ satisfies $\star$ in $a$.
\end{itemize}

In particular, we get definitions for $\slalomt(D,\star,J)$, $\slalome(D,\star,J)$, $\slalomt(b,\star,J)$, $\slalome(b,\star,J)$, $\slalomt(\star,J)$, $\slalome(\star,J)$, and, in relation with~\ref{it:constI}, we denote
\begin{align*}
\lfrak_D(\star,J)&:=\slalomt(D,\Prop_{\cnt},J),&
\pfrak_D(\star,J)&:=\slalome(D,\Prop_{\cnt},J).
\end{align*}
\end{definition}

The latter slalom numbers are easy to characterize.

\begin{lemma}\label{minpklk}
    Let $a$ be an infinite set and $D\subseteq {}^\omega a$.  
    Then\/ 
    \[\dfrak_A(D,\Prop_{\cnt}) = \min\set{\dfrak_A(D,\cnt(I))}{\text{$I$ is an ideal on $a$}}.\] 
    In particular, whenever $J\subseteq\pts(\omega)$,
    \begin{align*}
        \pfrak_D(\star,J) = & \min\set{\pfrak_D(I,J)}{\text{$I$ is an ideal on $a$}},\\
        \lfrak_D(\star,J) = & \min\set{\lfrak_D(I,J)}{\text{$I$ is an ideal on $a$}}.
    \end{align*}
\end{lemma}

\begin{proof}
    Notice that, for any set $S$ of constant $\omega$-sequences with value in $\pts(a)$, $S$  satisfies property $\star$ iff there is some ideal $I$ on $a$ such that $S\subseteq \cnt(I)$. This fact allows us to easily prove the result.
\end{proof}

In the following section, we prove similar results for $\slalomt(\star,J)$ and $\slalome(\star,J)$.

\begin{remark}\label{remnotation}
    We list some notation used in other references.
    \begin{enumerate}[label=\rm\arabic*.]
        \item \emph{Austria--Israel notation,} see \cite{GS93,K08,KS09,KS12,KM,CKM}:
        \begin{align*}
            \cfa_{b,h} & = \slalomt(b,h,\Fin), & \cxt_{b,h} & = \slalome(b,h,\Fin),\\ 
            \vfa_{b,h} & = \slalomt^\perp(b,h,\Fin), &
            \vxt_{b,h} & = \slalome^\perp(b,h,\Fin).
        \end{align*}

        \item \emph{Higher cardinal characteristics notation,} see \cite{BBSM,BS22,TvdV25}:
        \begin{align*}
            \dfrak_h(\in^*) & = \slalomt(h,\Fin), & \bfrak_h(\in^*) & = \slalomt^\perp(h,\Fin)\\
            \dfrak_\omega^{b,h}(\in^*) & = \slalomt(b,h,\Fin), & \bfrak_\omega^{b,h}(\in^*) & = \slalomt^\perp(b,h,\Fin)\\
            \dfrak_\omega^{b,h}(\cancel{\ni^\infty}) & = \slalome(b,h,\Fin), & \bfrak_\omega^{b,h}(\cancel{\ni^\infty}) & = \slalome^\perp(b,h,\Fin).
        \end{align*}

        \item \emph{Colombian-expats notation,} see \cite{CM19,CM23,Car4E}:
        \begin{align*}
            \dlc_{b,h} & = \slalomt(b,h,\Fin), & \balc_{b,h} & = \slalome(b,h,\Fin),\\ 
            \blc_{b,h} & = \slalomt^\perp(b,h,\Fin), &
            \dalc_{b,h} & = \slalome^\perp(b,h,\Fin).
        \end{align*}

        \item \emph{Earlier Slovak notation,}\footnote{The \emph{current} Slovak notation is the one we use in this paper.} see \cite{SJ,SoSu,REP21a,REP21b}:
        \begin{align*}
            \kfrak_{I,J} & = \pfrak_\Kat(I,J), & \lambda(I,J) & = \slalome(I,J),\\ 
            \kfrak^*_{I,J} & = \pfrak_\KB(I,J). &
            &
        \end{align*}           
    \end{enumerate}
\end{remark}

We obtain general connections and inequalities as follows.

\begin{lemma}\label{lem:gen}
Let $D$, $D'$, $E$, $E'$ be sets of functions with domain $\omega$, let $\Prop$ and $\Prop'$ be properties of sets of functions with domain $\omega$, and let $A$ and $A'$ be subsets of $\pts(\omega)$.
\begin{enumerate}[label=\rm(\alph*)]
    \item\label{it:monD} If $D \subseteq D'$ then\/ $\la D,E,{\in^A}\ra\leqT \la D',E,{\in^A}\ra$. 

    In particular, $\dfrak_A(D,E)\leq\dfrak_A(D',E)$ and\/ $\bfrak_A(D',E)\leq\bfrak_A(D,E)$, even more\/ $\dfrak_A(D,\Prop)\leq \dfrak_A(D',\Prop)$.

    \item\label{it:monE} If $E \subseteq E'$ then\/ $\la D,E',{\in^A}\ra\leqT\la D,E,{\in^A}\ra$.

    In particular, $\dfrak_A(D,E')\leq\dfrak_A(D,E)$ and\/ $\bfrak_A(D,E)\leq\bfrak_A(D,E')$.

    \item\label{it:monP} If $\Prop \subseteq \Prop'$ then\/ $\dfrak_A(D,\Prop')\leq\dfrak_A(D,\Prop)$.

    \item\label{it:monA} If $A\subseteq A'$ then\/ $\la D,E,{\in^{A'}}\ra\leqT \la D,E,{\in^A}\ra$. Even more, $\dfrak_{A'}(D,\Prop)\leq \dfrak_A(D,\Prop)$. 
\end{enumerate}
\end{lemma}

\begin{proof}
The inclusion maps $\id_D\colon D\to D'$ and $\id_E\colon E\to E'$, as well as identity maps, can be used to construct the Tukey connections for \ref{it:monD}, \ref{it:monE} and~\ref{it:monA}. For the latter, note that $A\subseteq A'$ implies that $x\in^A s\imp x\in^{A'} s$. The inequalities using $\Prop$ and~\ref{it:monP} are easy to check.
\end{proof}

\begin{corollary}\label{cor:gen}
Let $J\subseteq\pts(\omega)$.
\begin{enumerate}[label=\normalfont(\alph*)]
    \item $\Lc_J(D,I)\leqT \Lbf_{I,J}(D)$, in particular, $\slalomt(D,I,J)\leq \lfrak_D(I,J)$.

    \item $\pLc_J(D,I)\leqT \pL_{I,J}(D)\leqT \Lbf_{I,J}(D)$, in particular, $\slalome(D,I,J)\leq \pfrak_D(I,J) \leq \lfrak_D(I,J)$.

    \item $\pLc_J(D,E)\leqT \Lc_J(D,E)$, in particular, $\slalome(D,E,J)\leq \slalomt(D,E,J)$. Even more, $\slalome(D,\Prop,J)\leq\slalomt(D,\Prop,J)$ for any property $\Prop$.
    \qed
\end{enumerate}
\end{corollary}

\begin{notation}
Let $J\subseteq\pts(\omega)$ (typically an ideal).
\begin{enumerate}[label=(\arabic*)]
    \item For two sets $D$ and $D'$ of functions with domain $\omega$, write
    \[
    D\sqsubseteq^J D' \text{ iff } (\forall x\in D) (\exists x'\in D')\ x'=^{\dfil{J}} x.
    \]
    \item For two sets $E$ and $E'$ of functions with domain $\omega$, write
    \[
    E\Subset^J E' \text{ iff } (\forall s\in E)(\exists s'\in E') s\subseteq^{\dfil{J}} s'.
    \]
    \item Let $\Prop$ and $\Prop'$ be two properties for sets of functions with domain $\omega$. We write $\Prop \imp^J \Prop'$ if, for any $S$ satisfying $\Prop$, there is some $S'$ satisfying $\Prop'$ such that $S\Subset^J S'$.
\end{enumerate}
\end{notation}

We use the previous notation to improve \Cref{lem:gen} when using ideals.

\begin{lemma}\label{lem:monJ}
Let $D$, $D'$, $E$, $E'$ be sets of functions with domain $\omega$, let $\Prop$ and $\Prop'$ be properties for sets of functions with domain $\omega$, and let $J$ be an ideal on $\omega$.
\begin{enumerate}[label=\rm(\alph*)]
    \item\label{mon:JD} If $D \sqsubseteq^J D'$ then 
    \begin{align*}
    \Lc_J(D,E) &\leqT \Lc_J(D',E), & \pLc_J(D,E) &\leqT \pLc_J(D',E),\\ 
    \slalomt(D,\Prop,J) &\leq \slalomt(D',\Prop,J), & \slalome(D,\Prop,J) &\leq \slalome(D',\Prop,J).    
    \end{align*}

    \item\label{mon:JE} If $E \Subset^J E'$ then 
    \begin{align*}
    \Lc_J(D,E') &\leqT \Lc_J(D,E), & \pLc_J(D,E') &\leqT \pLc_J(D,E),\\ 
    \slalomt(D,E',J) &\leq \slalomt(D,E,J), & \slalome(D,E',J) &\leq \slalome(D,E,J).    
    \end{align*}

    \item\label{mon:JP} If $\Prop \imp^J \Prop'$ and $\Prop'$~is $\subseteq$-downward closed then\/ $\slalomt(D,\Prop',J)\leq \slalomt(D,\Prop, J)$ and\/ $\slalome(D,\Prop',J)\leq \slalome(D,\Prop, J)$.

    \item\label{mon:J} If $J\subseteq J'\subseteq\pts(\omega)$ then\footnote{Here, there is no need to assume that $J$ and $J'$ are ideals.}
    \begin{align*}
    \Lc_{J'}(D,E) &\leqT \Lc_J(D,E), & \pLc_J(D,E) &\leqT \pLc_{J'}(D,E),\\ 
    \slalomt(D,\Prop,J') &\leq \slalomt(D,\Prop,J), & \slalome(D,\Prop,J) &\leq \slalome(D,\Prop,J').  
    \end{align*}
\end{enumerate}
\end{lemma}

\begin{proof}
When $D\sqsubseteq^J D'$ and $E\Subset^J E'$, define the maps $f\colon D\to D'$ and $g\colon E\to E'$ such that, for $x\in D$ and $s\in E$, $f(x) =^{\dfil{J}} x$ and $s\subseteq^{\dfil{J}} g(s)$. These maps, along with identity maps, yield the Tukey connections for~\ref{mon:JD} and~\ref{mon:JE}. 
The inequalities at the bottom of~\ref{mon:JE} follow by \Cref{cor:Tukeyval}~\ref{Tukeyval:b}.

We show the inequalities at the bottom of~\ref{mon:JD}. Let $S$ be a~witness of $\slalomt(D',\Prop,J)$, so $S$ satisfies property~$\Prop$. If $x\in D$ then $f(x)\in D'$, so $f(x)\in^{\dfil{J}} s$ for some $s\in S$. Since $x=^{\dfil{J}} f(x)$, we get that $x\in^{\dfil{J}}s$. Hence, $S$~satisfies the properties of the definition of $\slalomt(D,\Prop,J)$, so $\slalomt(D,\Prop,J)\leq |S|=\slalomt(D',\Prop,J)$. A~similar argument guarantees $\slalome(D,\Prop,J)\leq \slalome(D',\Prop,J)$, just note that $f(x)\in^{J^+}s$ implies $x\in^{J^+} s$.


\ref{mon:JP}: Assume $\Prop\imp^J\Prop'$ and $\Prop'$~is $\subseteq$-downward closed. Let $S$ be a~witness of $\slalomt(D,\Prop,J)$. Since $S$ has property~$\Prop$, there is some~$S'$ with property~$\Prop'$ and some map $g'\colon S\to S'$ such that $s\subseteq^{\dfil{J}} g'(s)$ for all $s\in S$. Since $\Prop'$~is $\subseteq$-downward closed, we can assume that $S'=g'[S]$ and hence $|S'|\leq|S|$. If $x\in D$ then $x\in^{\dfil{J}} s$ for some $s\in S$, which implies that $x\in^{\dfil{J}} g'(s)$. Therefore, $S'$ satisfies the properties of the definition of $\slalomt(D,\Prop',J)$, so $\slalomt(D,\Prop',J)\leq |S'|\leq |S| = \slalomt(D,\Prop,J)$. The case for $\slalome$ is similar, just note that $x\in^{J^+}s$ implies $x\in^{J^+} g'(s)$.

\ref{mon:J}: Note that $J\subseteq J'$ implies $\dfil{J}\subseteq \dfil{J'}$ and ${J'}^{\rm c}\subseteq J^{\rm c}$, so the result follows by \Cref{lem:gen}~\ref{it:monA}.
\end{proof}

In connection with \Cref{lem:monJ}~\ref{mon:JE}, we obtain:

\begin{lemma}\label{restE}
If $D\subseteq\prod_{n<\omega}b(n)$, $E$~is a~set of functions with domain~$\omega$, and $A\subseteq \pts(\omega)$, then\/ $\la D,E,{\in^A}\ra\eqT\la D,E',{\in^A}\ra$ where $E'$ is the collection of all slaloms of the form $\seqn{s(n)\cap b(n)}{n<\omega}$ for some $s\in E$.
\end{lemma}
\begin{proof}
    Abusing notation, notice that $E'\Subset^{\{\emptyset\}}E$, so the proof of \Cref{lem:monJ}~\ref{mon:JE} can be used to show that $\la D,E,{\in^A}\ra \leqT \la D,E',{\in^A}\ra$. The converse Tukey connection is obtained by using the identity map of $D$ and $s\mapsto \seqn{s(n)\cap b(n)}{n<\omega}$.
\end{proof}

In the following section, we are going to review some order of ideals like the \emph{Katetov order} and the \emph{Katetov-Blass order}, with some variations, and study its effect on the slalom numbers. 
The following results are useful there.

\begin{lemma}\label{thm:RB1}
Let $f\colon\omega\to\omega$, denote $a_n:=f^{-1}\llbracket\{n\}\rrbracket$ for $n\in\omega$, and let $J$ and $J'$ be ideals on $\omega$.
Assume:
\begin{enumerate}[label=\rm(\roman*)]
    \item\label{it:RB1i} $D$ and $D'$ are two sets of functions with domain $\omega$ such that, whenever $x'\in D'$, $x'\circ f\in D$.
    \item\label{it:RB1ii} $E$ and $E'$ are two sets of functions with domain $\omega$ such that, for any $s\in E$, there is some function $s'\in E'$ such that\/ $\bigset{n<\omega}{\bigcup_{k\in a_n}s(k)\subseteq s'(n)}\in\dfil{J'}$.
\end{enumerate}
Then:
\begin{enumerate}[label=\rm(\alph*)]
    \item\label{it:RBincr} Whenever $J'\subseteq f^\to(J)$, $\pLc_{J'}(D',E')\leqT \pLc_J(D,E)$, so\/ $\slalome(D',E',J')\leq \slalome(D,E,J)$.
    \item\label{it:RBbincr} Whenever $f^\to(J)\subseteq J'$, $\Lc_{J'}(D',E')\leqT \Lc_J(D,E)$,  so\/ $\slalomt(D',E',J')\leq \slalomt(D,E,J)$.\qed
\end{enumerate}
\end{lemma}




\Cref{thm:RB1} is a~particular case of the following result when $F=\{f\}$.

\begin{lemma}\label{thm:RBuni}
Let $F\subseteq\Baire$,
let $A$ and~$A_f$ for $f\in F$ be $\subseteq$-upwards closed subsets
of~$\pts(\omega)$,
let $J$ and $J_f$ for $f\in F$ be ideals on~$\omega$, and let
$D$, $D_f$, $E$, $E_f$ for $f\in F$ be sets of functions with domain~$\omega$.
Denote
\begin{align*}
&A_F=\bigset{a\cap b}{\textstyle a\in J^\dual\text{ and }b\in\bigcap_{f\in F}f^\to(A_f)}\quad\text{and}\\
&E_F=\bigset{t_{\bar s}}{\textstyle\bar s\in\prod_{f\in F}E_f},
\quad\text{where}\quad
t_{\bar s}(n)=
\bigcup_{f\in F}\bigcup_{k\in f^{-1}\llbracket\{n\}\rrbracket}s_f(k),
\end{align*}
and assume that $E_F\Subset^J E$ and, for all $f\in F$, $\set{x\circ f}{x\in D}\sqsubseteq^{J_f} D_f$ and $a\cap b\in A_f$ for all $a\in \dfil{J_f}$ and $b\in A_f$.\footnote{The latter assumption ``for all $f\in F$, ..." can be replaced by $D\sqsubseteq^J \set{x\in\Baire}{(\forall f\in F)\ x\circ f\in D_f}$ and the proof is similar.}
\begin{enumerate}[label=\rm(\alph*)]
\item\label{RBuni-0}
If $A_F\subseteq A$, then\/
$\la D,E,{\in^A}\ra\leqT\bigotimes_{f\in F}\la D_f,E_f,{\in}^{A_f}\ra$.
\item\label{RBuni-1}
If\/ $\bigcap_{f\in F}f^\to(J_f)\subseteq J$, then\/ $\Lc_{J}(D,E)\leqT\bigotimes_{f\in F}\Lc_{J_f}(D_f,E_f)$;
in particular, \[\slalomt(D,E,J)\leq\prod_{f\in F}\slalomt(D_f,E_f,J_f).\]
\item\label{RBuni-2}
If $J\subseteq\bigcup_{f\in F}f^\to(J_f)$, then\/ $\pLc_{J}(D,E)\leqT \bigotimes_{f\in F}\pLc_{J_f}(D_f,E_f)$;
in particular, $\slalome(D,E,J)\leq\prod_{f\in F}\slalome(D_f,E_f,J_f)$.
\end{enumerate}
\end{lemma}

\begin{proof}
\ref{RBuni-0}:
We define $G_-\colon D\to\prod_{f\in F}D_f$ and
$G_+\colon\prod_{f\in F}E_f\to E$ as follows:
For every $x\in D$, $f\in F$ and $\bar s\in\prod_{f\in F}E_f$ fix $x_f\in D_f$
and $u_{\bar s}\in E$ such that $x_f=^{\dfil{J_f}}x\circ f$ and
$t_{\bar s}\subseteq^{J^\dual}u_{\bar s}$.
Define $G_-(x):=\seqn{x_f}{f\in F}$ and $G_+(\bar s):=u_{\bar s}$.
We prove that for every $x\in D$ and $\bar s\in\prod_{f\in F}E_f$,
$G_-(x)\sqsubset^\otimes\bar s$ implies $x\in^A G_+(\bar s)$.

Let $x\in D$ and $\bar s\in\prod_{f\in F}E_f$ and assume that
$x_f\in^{A_f}s_f$ for all $f\in F$, i.e., $\|x_f\in s_f\|\in A_f$. Since $\|x_f = x\circ f\|\in\dfil{J_f}$, we get that $\|x\circ f\in s_f\|\in A_f$. 
We claim that
$\|x\circ f\in s_f\|\subseteq f^{-1}\llbracket\|x\in t_{\bar s}\|\rrbracket$.
Indeed,
if $x(f(k))\in s_f(k)$ then, for $n=f(k)$,
$x(n)\in\bigcup_{i\in f^{-1}\llbracket\{n\}\rrbracket)}s_f(i)\subseteq t_{\bar s}(n)$ because
$k\in f^{-1}\llbracket\{n\}\rrbracket$.
Then $f^{-1}\llbracket\|x\in t_{\bar s}\|\rrbracket\in A_f$ because $A_f$~is
$\subseteq$-upwards closed, so $\|x\in t_{\bar s}\|\in f^\to(A_f)$.
We have seen that $\|x\in t_{\bar s}\|\in\bigcap_{f\in F}f^\to(A_f)$.
Since 
$\|t_{\bar s}\subseteq u_{\bar s}\|\cap\|x\in t_{\bar s}|\subseteq
\|x\in u_{\bar s}\|$, 
$\|t_{\bar s}\subseteq u_{\bar s}\|\in J^\dual$ and $A$~is
$\subseteq$-upwards closed, we get $\|x\in u_{\bar s}\|\in A$, i.e.,
$x\in^AG_+(\bar s)$.

\ref{RBuni-1}--\ref{RBuni-2}:
Apply~\ref{RBuni-0}
to $A_f=J_f^\dual$ and $A=J^\dual$ in case~\ref{RBuni-1}
and
to $A_f=J_f^+$ and $A=J^+$ in case~\ref{RBuni-2}. 
\end{proof}

\begin{lemma}\label{thm:RB2}
Let $f\in\Baire$, let $J$ and $J'$ be ideals on $\omega$, let
$A,A'\subseteq\pts(\omega)$, and let $D$, $D'$, $E$, $E'$ be sets of functions with domain~$\omega$.
Denote
\begin{align*}
&D_f=\set{x_f}{x\in D}
&&\text{where\quad $x_f(n)=x\frestr f^{-1}\llbracket\{n\}\rrbracket$, $n\in\omega$,}
\quad\text{and}\\
&E'^f=\set{s^f}{s\in E'}
&&\text{where\quad}s^f(k)=
\set{t(k)}
{\text{$t\in s(f(k))$ is a~function and $k\in\dom(t)$}},\
k\in\omega,
\end{align*}
and assume that $D_f\subseteq D'$ and $E'^f\subseteq E$.
\begin{enumerate}[label=\rm(\alph*)]
\item\label{it:RBgen}
If $A'\subseteq f^{\to}(A)$ and $A$ is\/ $\subseteq$-upwards closed, then\/
$\la D,E,{\in^A}\ra\leqT\la D',E',{\in^{A'}}\ra$;
in particular, $\dfrak_A(D,E)\leq \dfrak_{A'}(D',E')$.
\item\label{it:RBdcr} 
If $J'\subseteq f^\to(J)$ then\/ $\Lc_{J}(D,E)\leqT \Lc_{J'}(D',E')$ and\/
$\slalomt(D,E,J)\leq \slalomt(D',E',J')$.
\item\label{it:RBbdcr} 
If $f^\to(J)\subseteq J'$ then\/ $\pLc_{J}(D,E)\leqT \pLc_{J'}(D',E')$ and\/
$\slalome(D,E,J)\leq \slalome(D',E',J')$.
\end{enumerate}
\end{lemma}

\begin{proof}
\ref{it:RBgen}: 
Define $G_-\colon D\to D'$ and $G_+\colon E'\to E$ by $G_-(x):= x_f$ and
$G_+(s):= s^f$.
We show that $(G_-,G_+)$ is the desired Tukey connection.
Let $x\in D$, $s\in E'$, and assume that $G_-(x)\in^{A'}s$, i.e.,
$\|x_f\in s\|\in A'$. 
Then $f^{-1}\llbracket\|x_f\in s\|\rrbracket\in A$.
Note that $x_f(f(k))\in s(f(k))$ implies $x(k)\in s^f(k)$ because
$k\in f^{-1}\llbracket\{f(k)\}\rrbracket=\dom(x_f(f(k)))$.
Therefore $f^{-1}\llbracket\|x_f\in s\|\rrbracket\subseteq\|x\in s^f\|$.
Then $\|x\in s^f\|\in A$ (because $A$ is $\subseteq$-upwards closed), i.e.,
$x\in^{A}G_+(s)$.

\ref{it:RBdcr} follows directly by~\ref{it:RBgen} applied to $A=\dfil{J}$ and
$A'=\dfil{J'}$, and \ref{it:RBbdcr} follows by~\ref{it:RBgen} applied to $A=J^+$ and $A'= {J'}^+$.
\end{proof}

For the following results, we fix the following notation.

\begin{notation}\label{n:ast}
    When $\bar f = \seqn{f_n}{n<\omega}$ is a~sequence of functions and $x$ is a~function with domain $\omega$, let $\bar f\ast x$ be the function whose domain are those $n<\omega$ such that $x(n)\in \dom f_n$, and $(\bar f\ast x)(n):= f_n(x(n))$. Also let $\bar f\circledast x$ be the function with domain $\omega$ such that $(\bar f \circledast x)(n):= f^{-1}_n\llbracket x(n) \rrbracket$.

    Let $F$ be a~set of sequences $\bar f$ as above, and let $\bar s = \seqn{s_{\bar f}}{\bar f\in F}$ be a~sequence of functions with domain $\omega$. Then define the function $F \circledast\bar s$ with domain $\omega$, such that
    \[(F \circledast\bar s)(n) := \bigcap_{\bar h\in F}\bigcup_{\bar f\in F} h^{-1}_n\llbracket s_{\bar f}(n)\rrbracket.\]
    We also allow the notation $\bar f\ast D := \set{\bar f \ast x}{x\in D}$ and $F\circledast E := \set{F\circledast \bar s}{\bar s\in E}$ when $D$ is a~set of functions with domain $\omega$ and $E$ is a~set of sequences of the form $\bar s = \seqn{s_{\bar f}}{\bar f\in F}$ as above. Likewise for $\bar f\circledast D$.

    When $f$ is a~function, denote $f\ast D:= \set{f\circ x}{x\in D}$.
\end{notation}

\begin{lemma}\label{gen:decrK}
    Let $D$, $D'$, $E$, $E'$ be sets of functions with domain $\omega$, and $\bar f = \seqn{f_n}{n<\omega}$ a~sequence of functions. Assume that $\bar f\ast D\subseteq D'$ and $\bar f\circledast E' \subseteq E$. 
    Then\/ $\la D, E,{\in^A}\ra\leqT \la D', E',{\in^A}\ra$, in particular, $\dfrak_A(D,E)\leq \dfrak_A(D',E')$.
\end{lemma}

\begin{proof}
    Define $F\colon D\to D'$ and $G\colon E'\to E$ such that $F(x):= \bar f\ast x$ and $G(s'):=\bar f\circledast s'$. It is clear that $(F,G)$ is the desired Tukey connection.
\end{proof}

\begin{lemma}\label{Kinterc}
Let $F$ be a~finite set of sequences $\bar f=\seqn{f_n}{n<\omega}$, where each~$f_n$ is a~function, let
$D$, $D_{\bar f}$, $E$, $E_{\bar f}$ ($\bar f\in F$) be sets of functions with domain~$\omega$, and let $J$ be an ideal on~$\omega$. 
If for every $\bar f\in F$, $\bar f* D \sqsubseteq^JD_{\bar f}$,\footnote{So $\bar f \ast x$ has domain $\omega$ for all $\bar f\in F$ and $x\in D$.} 
and\/
$F\circledast\prod_{\bar f\in F}E_{\bar f}\Subset^JE$,
then\/ $\Lc_J(D,E)\leqT \bigotimes_{\bar f\in F}\Lc_J(D_{\bar f},E_{\bar f})$;
in particular,
$\slalomt(D,E,J)\leq \prod_{\bar f\in F}\slalomt(D_{\bar f},E_{\bar f},J)$.

Moreover, when $F=\{\bar f\}$, we also obtain $\pLc_J(D,E) \leqT \pLc_J(D_{\bar f},E_{\bar f})$ and $\slalome(D,E,J) \leq \slalome(D_{\bar f},E_{\bar f},J)$.
\end{lemma}

\begin{proof}
Define
$\Phi_-\colon D\to\prod_{\bar f\in F}D_{\bar f}$ and
$\Phi_+\colon\prod_{\bar f\in F}E_{\bar f}\to E$ as follows:
For $x\in D$ choose
$\Phi_-(x)=\seqn{x_{\bar f}}{\bar f\in F}\in\prod_{\bar f\in F}D_{\bar f}$
so that $\bar f*x=^{J^\dual}x_{\bar f}$ for every $\bar f\in F$.
For $\bar s\in\prod_{\bar f\in F}E_{\bar f}$ choose
$\Phi_+(\bar s)=t_{\bar s}\in E$ such that
$F\circledast\bar s\subseteq^{J^\dual}t_{\bar s}$.
Assume that $\Phi_-(x)\sqsubset^{\otimes}\bar s$, i.e., for every $\bar h\in F$,
$\|x_{\bar h}\in s_{\bar h}\|\in J^\dual$.
Then
$a_{\bar h}:= \|\bar h\ast x = x_{\bar h}\|\cap\allowbreak\|x_{\bar h}\in s_{\bar h}\|
\in J^\dual$
and
$a_{\bar h}\subseteq\bigset{n<\omega}
{x(n)\in\bigcup_{\bar f\in F}h_n^{-1}\llbracket s_{\bar f}(n)\rrbracket}$.
Since $F$~is finite, $\bigcap_{\bar h\in F}a_{\bar h}\in J^\dual$ and so
$x\in^{J^\dual}F\circledast\bar s$.
Then $x\in^{J^\dual}\Phi_+(\bar s)$ because
$F\circledast\bar s\subseteq^{J^\dual}t_{\bar s}$.

When $F=\{\bar f\}$, we also obtain $\pLc_J(D,E) \leqT \pLc_J(D_{\bar f},E_{\bar f})$ with the same proof (just noting that $a_{\bar f} \in J^+$).
\end{proof}

As a~consequence of the previous, we can infer:

\begin{corollary}\label{Kinterc1}
    Let $F$ and $J$ be as before, and let $\bar I^{\bar f}=\seqn{I^{\bar f}_n}{n<\omega}$ ($\bar f\in F$) and $\bar I'=\seqn{I'_n}{n<\omega}$ be sequences of families of sets.\footnote{In most cases, sequences of ideals.} If\/ $\bar f\ast D \sqsubseteq^J D_{\bar f}$ for all $\bar f\in F$ and, for all $n<\omega$ and $\bar a\in\prod_{\bar f\in F}I^{\bar f}_n$, $\bigcap_{\bar h\in F}h^{-1}_n\llbracket \bigcup_{\bar f\in F}a_{\bar f}\rrbracket\in I'_n$, then\/ 
    $\Lc_J(D',\bar I')\leqT\bigotimes_{\bar f\in F}\Lc_J(D_{\bar f},\bar I^{\bar f})$. 
    In particular, $\slalomt(D',\bar I',J)\leq \prod_{\bar f\in F}\slalomt(D_{\bar f},\bar I^{\bar f},J)$.

    In addition, when $F=\{\bar f\}$, $\pLc_J(D',\bar I') \leqT \pLc_J(D_{\bar f}, \bar I^{\bar f})$ and $\slalome(D',\bar I',J) \leq \slalome(D_{\bar f}, \bar I^{\bar f}, J)$.
    \qed
\end{corollary}

\begin{corollary}\label{Kinterc2}
    Let $F$ and $J$ be as before, and $\bar I=\seqn{I_n}{n<\omega}$ and $\bar I'=\seqn{I'_n}{n<\omega}$ sequences of families of sets where each $I_n$ is an ideal on some $b(n)$. If\/ $\bar f\ast D \sqsubseteq^J D_{\bar f}$ for all $\bar f\in F$ and $I_n\subseteq\set{a\subseteq b(n)}{\bigcap_{\bar h\in F}h^{-1}_n\llbracket a\rrbracket\in I'_n}$ for all $n<\omega$, then\/ 
    $\Lc_J(D',\bar I')\leqT \bigotimes_{\bar f\in F}\Lc_J(D_{\bar f},\bar I)$. 
    In particular, $\slalomt(D',\bar I',J)\leq \prod_{\bar f\in F}\slalomt(D_{\bar f},\bar I,J)$.

    In addition, when $F=\{\bar f\}$, $\pLc_J(D',\bar I') \leqT \pLc_J(D_{\bar f}, \bar I)$ and $\slalome(D',\bar I',J) \leq \slalome(D_{\bar f}, \bar I, J)$.
    \qed
\end{corollary}

\begin{corollary}\label{Kinterc3}
    Let $F$ be a~finite set of functions with domain $\omega$, $J$ an ideal on $\omega$, and let $I'$ and $I_f$ ($f\in F$) be families of sets. If\/$f\ast D \sqsubseteq^J D_{f}$ for all $ f\in F$ and, for all $\bar a\in \prod_{f\in F}I_f$, $\bigcap_{h\in F}h^{-1}\llbracket \bigcup_{f\in F}a_{f}\rrbracket\in I'$, then\/ $\Lc_J(D',I')\leqT \bigotimes_{f\in F}\Lc_J(D_{f}, I_{f})$ and\/ $\Lbf_{I',J}(D)\leqT \bigotimes_{f\in F}\Lbf_{I_f,J}(D_f)$. In particular, $\slalomt(D',I',J)\leq \prod_{f\in F}\slalomt(D_f,I_f,J)$ and\/ $\lfrak_{D'}(I',J)\leq \prod_{d\in F}\lfrak_{D_f}(I_f,J)$.

    In addition, when $F=\{f\}$, $\pLc_J(D',I') \leqT \pLc_J(D_{f}, I_f)$, $\pL_{I',J}(D') \leqT \pL_{I_f,J}(D_{f})$, $\slalome(D', I',J) \leq \slalome(D_{f}, I, J)$ and $\pfrak_{D'}(I',J)\leq \pfrak_{D_f}(I_f,J)$.
    \qed
\end{corollary}


\begin{corollary}\label{Kinterc4}
    Let $F$, $J$, and $I'$ be as before, and let $I$ be an ideal on some set $a_*$. If\/ $f\ast D \sqsubseteq^J D_{f}$ for all $f\in F$ and $I\subseteq\set{a\subseteq a_*}{\bigcap_{h\in F}h^{-1}\llbracket a\rrbracket\in I'}$, then\/ $\Lc_J(D',I')\leqT \bigotimes_{f\in F}\Lc_J(D_{f}, I)$ and\/ $\Lbf_{I',J}(D)\leqT \bigotimes_{f\in F}\Lbf_{I,J}(D_f)$. In particular, $\slalomt(D',I',J)\leq \prod_{f\in F}\slalomt(D_f,I,J)$ and\/ $\lfrak_{D'}(I',J)\leq \prod_{f\in F}\lfrak_{D_f}(I,J)$.

    In addition, when $F=\{f\}$, $\pLc_J(D',I') \leqT \pLc_J(D_{f}, I)$, $\pL_{I',J}(D') \leqT \pL_{I,J}(D_{f})$, $\slalome(D', I',J) \leq \slalome(D_{f}, I, J)$ and $\pfrak_{D'}(I',J)\leq \pfrak_{D_f}(I,J)$.\qed
\end{corollary}

We close this section with a~general result that uses point-wise intersections of slaloms. This will be used in \Cref{L6.2}.

\begin{lemma}\label{slint}
    Let $E_0$, $E_1$ and $E$ be sets of slaloms and assume that, for any $s\in E_0$ and $s'\in E_1$, $s\wedge s' := \seqn{s(n)\cap s'(n)}{n<\omega}\in E$. Let $D$ be a~set of functions with domain $\omega$ and $J$ an ideal on $\omega$. Then:
    \begin{enumerate}[label = \normalfont (\alph*)]
        \item\label{slint1} $\Lc_J(D,E) \leqT \Lc_J(D,E_0)\otimes \Lc_J(D,E_1)$.
        \item\label{slint2} $\pLc_J(D,E) \leqT \Lc_J(D,E_0)\otimes \pLc_J(D,E_1)$.
        \item\label{slint3} Assume that $D\subseteq \baire{a}$ and, for any $s'\in E_1$, there is some $s''\in E_1$ such that $s' \subseteq^{\Fin^\dual} s''$ and $s''(n)\subseteq s''(n+1)$ for all $n<\omega$. Then 
        $\slalome(D,E,\Fin)\leq \slalome(D,E_0,\Fin)\cdot \slalome(\baire{a},E_1,\Fin)$. 
    \end{enumerate}
\end{lemma}
\begin{proof}
    The maps $x\mapsto (x,x)$ and $(s,s')\mapsto s\wedge s'$ gives the Tukey connections for~\ref{slint1} and~\ref{slint2}. 

    \ref{slint3}: Assume that $S_0\subseteq E_0$ and $S_1\subseteq E_1$ are witnesses of $\slalome(D,E_0,\Fin)$ and $\slalome(\baire{a},E_1,\Fin)$, respectively. By the assumption, we can also assume that, for every $s'\in S_1$, $s'(n)\subseteq s'(n+1)$ for all $n\in\omega$. Then $\set{s\wedge s'}{s\in S_0,\ s'\in S_1}$ is dominating in $\pLc_\Fin(D,E)$. 
    Indeed, for every $x\in D$ there is an $s\in S_0$ such that $x\in^{\fin^+}s$. Pick an increasing function $g\in\baire\omega$ such that $x(g(k))\in s(g(k))$ for all $k\in\omega$. Then, there is an $s'\in S_1$ such that $x\circ g\in^{\fin^+}s'$. However, $s'(k)\subseteq s'(g(k))$ for each $k\in\omega$. Hence, $x\circ g\in^{\fin^+}s'\circ g$, and so $x\in^{\fin^+}s\wedge s'$ as well.

    A dual argument shows that $\min\{\slalome^\perp(D,E_0,\Fin),\slalome^\perp(\baire{a},E_1,\Fin)\} \leq \slalome^\perp(D,E,\Fin)$.
\end{proof}

\begin{remark}\label{rem:gengen}
   Except for \Cref{slint}~\ref{slint3}, $\omega$ is immaterial for the development of the theory and results of this section. In general, the sets of functions and slaloms we have dealt with can have domain some infinite (possibly uncountable) set $w$, and $J$ can be assumed to be an ideal on $w$.
\end{remark}

\section{Partial orderings of ideals}\label{sec:partcases}

We shall be interested mainly in the particular cases of slalom invariants $\slalome(I,J)$, $\slalomt(I,J)$, $\slalome(\star,J)$, $\slalomt(\star,J)$, $\slalome(b,h,J)$, and $\slalomt(b,h,J)$ (with special attention to $\prod b = \Baire$) where $I$ and $J$ are ideals on $\omega$ and $h\in\Baire$. Let us recall that 
\begin{align*}
    \covst(I) & = \min\set{|A|}{A\subseteq I\text{ and }(\forall a\in[\omega]^\omega)(\exists b\in A)\ |a\cap b|=\omega} = \dfrak([\omega]^{\aleph_0}, I,{\not\perp^\fin}),\\
    \nonst(I)  & = \min\set{|E|}{E\subseteq [\omega]^{\aleph_0}\text{ and }(\forall b\in I)(\exists a\in E)\ |a\cap b|<\aleph_0} = \bfrak([\omega]^{\aleph_0}, I,{\not\perp^\fin}),\\
    \bfrak_J & =\bfrak(\baire\omega,\baire\omega,{\leq^{J^\dual}}), \text{ and}\\
    \dfrak_J & =\dfrak(\baire\omega,\baire\omega,{\leq^{J^\dual}}),
\end{align*}
where $a\perp^\Fin b$ iff $a\cap b$ is finite. Observe that 
$\covst(I)=\pfrak_\Kat(I,\Fin)$, 
$\bfrak_J=\slalome(\Fin,J)$, and
$\dfrak_J=\slalomt(\Fin,J)$, even more, $\la [\omega]^{\aleph_0}, I,{\not\perp^\fin}\ra \eqT \pL(I,\Fin)$ and $\la \Baire, \Baire,{\in^{J^\dual}}\ra \eqT \Lc_J(\Fin)\eqT \pLc_J(\Fin)^\perp$. 
The obvious inequalities (which follow from \Cref{lem:gen} and \Cref{cor:gen}) between the studied slalom numbers are summarized in the following
diagram in which $\rightarrow$ denotes~$\le$ and $h\geq^* 1$ in $\Baire$:
\begin{align*}
\begin{array}{*{11}{@{}c}@{}}
\slalomt(\star,J)&{}\rightarrow{}
&\slalomt(I,J)&{}\rightarrow{}&\dfrak_J&{}\rightarrow{}
&\slalomt(h,J)&\hskip4em
&\slalomt(I,J)&{}\rightarrow{}&\lfrak_\Kat(I,J)\\[2pt]
{\uparrow}&&{\uparrow}&&{\uparrow}&&{\uparrow}&&{\uparrow}&&{\uparrow}\\[2pt]
\slalome(\star,J)&{}\rightarrow{}
&\slalome(I,J)&{}\rightarrow{}&\bfrak_J&{}\rightarrow{}
&\slalome(h,J)&
&\slalome(I,J)&{}\rightarrow{}&\pfrak_\Kat(I,J)
\end{array}
\end{align*}

All these cardinals are uncountable as a~consequence of \Cref{S4_eq}~\ref{S4_eq:03}, even above the \emph{pseudo-intersection number $\pfrak$} (see also \Cref{G2}, \Cref{L2.6}~\ref{L2.6b} and \Cref{BasicDia}), and below $\cfrak$ when well-defined. The only cardinals that may not be well-defined are $\pfrak_\Kat(I,J)$ and $\lfrak_\Kat(I,J)$, see \Cref{undef:pl} and~\ref{L2.1}.

\begin{definition}\label{def:leqK}
Let $a$ and $a'$ be sets, $I\subseteq \pts(a)$ and $J\subseteq \pts(a')$ (usually ideals). 
The family of all $J$-to-one functions in ${}^{a'} a$ is denoted by $({}^{a'} a)^J$, i.e., $f\in ({}^{a'} a)^J$ iff $f\colon a'\to a$ and $f^{-1}\llbracket \{n\}\rrbracket\in J$ for all $n\in a$. We consider the following partial quasi-orderings (usually between ideals):\footnote{See \cite{GaMe} for applications of the dual Kat\v{e}tov-Blass ordering~$\le_\mKB$.}
\begin{align*}
I\le_\RK J &\text{ iff }(\exists f\in{}^{a'} a)\
I=f^\rightarrow(J)&&\text{(Rudin-Keisler)},\\
I\le_\Kat J &\text{ iff }(\exists f\in{}^{a'} a)\
I\subseteq f^\rightarrow(J)&&\text{(Kat\v{e}tov)},\\
I\le_\mKat J&\text{ iff }(\exists f\in{}^{a} a')\ f^\rightarrow(I)\subseteq J
&&\text{(dual Kat\v{e}tov)},\\
I\le_\RB J &\text{ iff }(\exists f\in({}^{a'} a)^\Fin)\
I=f^\rightarrow(J)&&\text{(Rudin-Blass)},\\
I\le_\KB J &\text{ iff }(\exists f\in({}^{a'} a)^\Fin)\
I\subseteq f^\rightarrow(J)&&\text{(Kat\v{e}tov-Blass)},\\
I\le_\mKB J & \text{ iff }(\exists f\in({}^{a} a')^\Fin)\
f^\rightarrow(I)\subseteq J&&\text{(dual Kat\v{e}tov-Blass)}.
\end{align*}
Note that if $a=a'$ and $I\subseteq J$ then $I\le_\KB J$ and $I\le_\mKB J$, witnessed by the identity map. Moreover, ${\le_\RB}\subseteq{\le_\KB}\subseteq{\le_\Kat}$ and ${\le_\RB}\subseteq{\ge_\mKB}$. In addition, to present some monotonicity results we shall need the following partial ordering.
\[\textstyle
I\le_\Kat^\ccap J\text{ iff }
(\exists F\in[{}^{a'} a]^{<\omega})\
I\subseteq\set{w\subseteq a}{\bigcap_{f\in F}f^{-1}\llbracket w\rrbracket\in J}.
\]
It is clear that
${\le_\Kat} \subseteq{\le_\Kat^\ccap}$, i.e., ${\le_\Kat^\ccap}$ is a~generalization of ${\le_\Kat}$. Notice that $\leq_\Kat^\ccap$ is a~preorder.

We write $J'\approx_\Kat J$ if $J'\leq_\Kat J$ and $J\leq_\Kat J'$, similarly for $\approx_{\mKat}$, $\approx^\ccap_\Kat$, etc. 
\end{definition}

It follows from the definitions that $\pfrak_\Kat(I,J)$ is undefined iff $I\leqK J$. In general: 

\begin{lemma}\label{undef:pl}
   Let $a$ be a set, $D\subseteq \baire{a}$, $I\subseteq\pts(a)$ and $J\subseteq\pts(\omega)$. Then: 
   \begin{enumerate}[label = \normalfont (\alph*)]
       \item $\pfrak_D(I,J)$ is undefined iff $I\subseteq f^\to(J)$ for some $f\in D$.
       \item $\lfrak_D(I,J)$ is undefined iff $I\subseteq f^{\to}(J^{\rm dc})$ for some $f\in D$.\qed
   \end{enumerate}
\end{lemma}

We have the following general result for~$\leq_K$.

\begin{lemma}\label{thm:IctdecrK}
Let $I\subseteq\pts(a)$, $I'\subseteq\pts(a')$, $A\subseteq\pts(\omega)$ and let $D\subseteq \baire{a}$ and $D'\subseteq\baire{a'}$.
Further, assu\-me that $I'\leqK I$ is witnessed by an $f\colon a\to a'$ and $f\circ x\in D'$ for every $x\in D$.
Then\/ $\la D, {}^\omega I,{\in^A}\ra\leqT \la D', {}^\omega I',{\in^A}\ra$ and\/ $\la D, \cnt(I),{\in^A}\ra\leqT \la D', \cnt(I'),{\in^A}\ra$.
In particular, $\dfrak_A(D,{}^\omega I)\leq \dfrak_A(D',{}^\omega I')$ and\/ $\dfrak_A(D,\cnt(I))\leq\dfrak_A(D',\cnt(I'))$.
\end{lemma}

\begin{proof}
    This is 
    a~particular case of \Cref{gen:decrK}.
\end{proof}

If $M$ is a~class, $\sqsubseteq$ is a~relation on~$M$, and $\lambda(I)$ is a~cardinal invariant depending on some parameter $I\in M$ then we say that $\lambda(I)$ is $\sqsubseteq$-decreasing ($\sqsubseteq$-increasing) if $\lambda(I_1)\leq\lambda(I_0)$ ($\lambda(I_0)\leq\lambda(I_1)$) whenever $I_0\sqsubseteq I_1$. 
In a~similar way, when $\Rbf(I)$ is a~relational system depending on $I\in M$, we say that $\Rbf(I)$ is $\sqsubseteq$-increasing if $\Rbf(I_0)\leqT \Rbf(I_1)$ whenever $I_0 \sqsubseteq I_1$ (and decreasing is defined naturally). 
For example, the previous lemma indicates that $\la\baire{a},I,{\in^A}\ra$ and $\la\baire{a},\cnt(I),{\in^A}\ra$ are $\leq_\Kat$-decreasing on $I$, and so are $\dfrak_A(\baire{a},I)$ and $\dfrak_A(\baire{a},\cnt(I))$ as a~consequence.

A more concrete example:

\begin{lemma}\label{L2.4}
 Let $J$ be an ideal on $\omega$. Then\footnote{The fact about $\slalome(I,J)$ has been shown in~\cite{SJ}. Moreover, the $\leq_\Kat$-monotonicity only requires $J\subseteq\pts(\omega)$.} 
\begin{enumerate}[label=\rm(\alph*)]
\item\label{sledecrI}
$\pLc_J(I)$ and\/ $\pL_{I,J}$ are\/ $\leq_\Kat$-decreasing on $I\subseteq\pts(\omega)$. Consequently,
$\slalome(I,J)$ and\/ $\pfrak_\Kat(I,J)$ are\/
$\le_\Kat$-decreasing on~$I$.
In particular, $\covst(I)$ is\/ $\le_\Kat$-decreasing on~$I$.

\item\label{L2.4b}
$\Lc_J(I)$ and\/ $\Lbf_{I,J}$ are\/ $\le_\Kat$-decreasing on $I\subseteq\pts(\omega)$, and\/ $\slalomt(I,J)$ and\/ $\lfrak_\Kat(I,J)$ are\/
$\le_\Kat^\ccap$-decreasing on~ideals $I$ on $\omega$.
\end{enumerate}
\end{lemma}

\begin{proof}
\ref{sledecrI} and \ref{L2.4b} for $\leqK$ are immediate consequences of~\Cref{thm:IctdecrK}.
For the remaining part of~\ref{L2.4b}: Let $I'\le_\Kat^\ccap I$ be witnessed by a~set $F\subseteq\Baire$. 
By applying \Cref{Kinterc4} to $D=D'=\Baire$, we obtain $\Lc(I',J)\leqT {}^{F\, } \Lc(I,J)$ and $\Lbf_{I',J}\leqT {}^{F\, }\Lbf_{I,J}$, so the result follows (also because $\slalomt(I,J)$ and $\lfrak_{\Kat}(I,J)$ are infinite, the latter possibly with value $\infty$).
\end{proof}

The previous results are enough to characterize the slalom numbers $\slalome(\star,J)$ and $\slalomt(\star,J)$. We first generalize (part of) \Cref{L2.4} as follows.

\begin{lemma}\label{thm:IdecrK}
    Let $b=\seqn{b(n)}{n<\omega}$ and $b'=\seqn{b'(n)}{n<\omega}$ be sequences of non-empty sets,
    $\bar I=\seqn{I_n}{n<\omega}$ and $\bar I'=\seqn{I'_n}{n<\omega}$  such that $I_n\subseteq\pts(b(n))$ and $I'_n\subseteq\pts(b'(n))$ for all $n<\omega$, let $A\subseteq \pts(\omega)$, and let $D\subseteq \prod b$ and $D'\subseteq \prod b'$. Assume that, for $n<\omega$, $I'_n\leqK I_n$ witnessed by a~function $f_n\colon b(n)\to b'(n)$. Further assume that, for any $x\in D$, $x'\in D'$ where $x'$ is the function defined by $x'(n):=f_n(x(n))$. Then\/ $\la D,\bar I,{\in^A}\ra\leqT \la D',\bar I',{\in^A}\ra$, in particular $\dfrak_A(D,\bar I)\leq\dfrak_A(D',\bar I')$.
\end{lemma}

\begin{proof}
    Immediate from \Cref{gen:decrK}.
\end{proof}

\begin{theorem}\label{cor:min*}
For $A\subseteq\pts(\omega)$,\footnote{This can be generalized to $\dfrak_A(\baire{a},\star) = \min\set{\dfrak_A(\baire{a},I)}{I \text{ is an ideal on }a}$ for any infinite set $a$.}
   \[\begin{split}
       \dfrak_A(\star) & =\min\set{\dfrak_A(\bar I)}{\bar{I} \text{ is a~sequence of ideals on }\omega}\\
        & =\min\set{\dfrak_A(I)}{I \text{ is an ideal on }\omega}.
   \end{split}\]
\end{theorem}

\begin{proof}
    The first equality is clear: on one hand, for any sequence $\bar I$ of ideals on $\omega$, any $S\subseteq\prod\bar I$ satisfies property $\star$, so $\dfrak_A(\star)\leq \dfrak_A(\bar I)$ follows; on the other hand, if $S$ witnesses $\dfrak_A(\star)$, since it satisfies property~$\star$, then $S\subseteq\prod\bar I$ for some sequence $\bar I$ of ideals on $\omega$, so $\dfrak_A(\bar I)\leq \dfrak_A(\star)$.

    For the second equality, $\leq$ is clear. For the converse, if $\bar{I}$ is a~sequence of ideals, then there is some ideal $I$ on $\omega$ such that $I_n\leqK I$ for all $n<\omega$, see~\cite{Mez,BorFar}. Hence, by \Cref{thm:IdecrK}, $\dfrak_A(I)\leq \dfrak_A(\bar I)$.
\end{proof}

As a~consequence:

\begin{theorem}\label{G2}
    Let $J$ be an ideal on $\omega$. Then
    \begin{align*}
      \slalome(\star,J) & =\min\set{\slalome(\bar I,J)}{\bar{I} \text{ is a~sequence of ideals on }\omega}\\
        & =\min\set{\slalome(I,J)}{I \text{ is an ideal on }\omega}
    \intertext{and}
        \slalomt(\star,J) & =\min\set{\slalomt(\bar I,J)}{\bar{I} \text{ is a~sequence of ideals on }\omega}\\
        & =\min\set{\slalomt(I,J)}{I \text{ is an ideal on }\omega}.
    \end{align*}
\end{theorem}

\begin{proof}
    Immediate from \Cref{cor:min*} applied to $A= J^+$ and $A = J^\dual$, respectively.
\end{proof}

Concerning the case $J=\Fin$, it is known that:

\begin{theorem}[{\cite{SoSu,Su22}}]\label{S4_eq}
\

\begin{enumerate}[label=\rm(\alph*)]
    \item\label{S4_eq:01} $\slalome(I,\Fin) = \min\{\covst(I),\bfrak\}$.

    \item\label{S4_eq:02} $\pfrak_\Kat(\star,\Fin) = \min\set{\covst(I)}{\text{$I$ is an ideal on $\omega$}} = \pfrak$.

    \item\label{S4_eq:03} $\slalome(\star,\Fin) = \pfrak.$\qed
\end{enumerate}
\end{theorem}

The equality $\slalomt(\star,\Fin)=\covm$ is proved in~\cite{Su22} by applying topological selection principles. 
Here, we provide a~direct combinatorial proof.

\begin{theorem}\label{sltstarcovM}
$\slalomt(\star,\Fin) = \slalomt(\pts(\omega)\menos\{\omega\},\Fin) =\covm$.
\end{theorem}

\begin{proof}
    Recall that $\covm$ is the smallest size of an eventually different family in $\Baire$ (see \Cref{F3.3}~\ref{F3.3a} for $h=1$), i.e., of a~family $D\subseteq\Baire$ such that 
    \[
    (\forall x\in\Baire)(\exists y\in D) (\exists m<\omega) (\forall n\geq m)\ x(n)\neq y(n).
    \]
    We show that such a~$D$ allows to construct a~localizing set of slaloms with property $\star$. This guarantees $\slalomt(\star,\Fin)\leq\covm$. Instead of $\Baire$, we look at $\prod_{n<\omega}W_n$ where $W_n:=\Fn(w_n,\omega)$, $\set{w_n}{n<\omega}$ is a~partition of $\omega$ into infinite sets and $\Fn(A,B)$ denotes the set of finite partial functions from $A$ into $B$. For $y\in\Baire$, define $s_y\in\prod_{n<\omega}\pts(W_n)$ by
    \[
    s_y(n):=\set{p\in W_n}{(\forall m\in\dom p)\ p(m)\neq y(m)}.
    \]
    We show that $S:=\set{s_y}{y\in D}$ is a~localizing family with property $\star$.
    Indeed, if $z\in\prod_{n<\omega}W_n$, then we can find some $x\in\Baire$ such that $z(n)\subseteq x$ for all $n<\omega$. Hence, there is some $y\in\Baire$ eventually different with $x$, so $z(n)\in s_y(n)$ for all but finitely many $n<\omega$.

    It remains to show that $S$ satisfies property $\star$. Let $\set{y_i}{i<k}\subseteq D$ and $n<\omega$, and we show that $\bigcup_{i<k}s_{y_i}(n)$ is co-infinite. We may assume that $k\neq0$. For any $c:=\set{m_i}{i<k}\subseteq w_n$ (one-to-one enumeration), we can define $p\in W_n$ with domain $c$ such that $p(m_i):=y_i(m_i)$, so $p\notin \bigcup_{i<k}s_{y_i}(n)$. This guarantees that $\bigcup_{i<k}s_{y_i}(n)$ is co-infinite.

    The inequality $\covm\leq\slalomt(\star,\Fin)$ is easier to show. By \Cref{lem:monJ}~\ref{mon:JP}, $\slalomt(\pts(\omega)\menos\{\omega\},\Fin)\leq \slalomt(\star,\Fin)$, so it is enough to show that $\covm\leq\slalomt(\pts(\omega)\menos\{\omega\},\Fin)$. We even have the Tukey connection $\la\Baire,\allowbreak\Mwf,{\in}\ra\leqT \Lc_\Fin(\pts(\omega)\menos\{\omega\})$ because, for $s\in{}^\omega(\pts(\omega)\menos\{\omega\})$, we have that $\set{x\in\Baire}{ x\in^{\dfil{\Fin}} s}$ is~$F_\sigma$~meager. Moreover, we actually have that $\Lc_\Fin(\pts(\omega)\menos\{\omega\}) \eqT \la \Baire,\Baire,{\neq^{\Fin^+}}\ra$.
\end{proof}

Concerning $\lfrak_\Kat(\star,\Fin)$, we have:

\begin{lemma}\label{L2.1}
If $J$ is an ideal on~$\omega$ and there is an infinite partition of $\omega$
into $J$-positive sets, then\/ $\lfrak_\Kat(\star,J)=\lfrak_\Kat(I,J)=\infty$ (i.e., undefined) for every ideal~$I$ on~$\omega$.
\end{lemma}

\begin{proof}
Choose a~partition $\{a_n:n\in\omega\}\subseteq J^+$ of $\omega$ and define $x(k)=n$ for
$k\in a_n$ and $n\in\omega$.
If $a\in I$, then $x^{-1}\llbracket a\rrbracket\notin J^\dual$ because
$x^{-1}\llbracket a\rrbracket\subseteq\omega\smallsetminus a_n\notin J^\dual$ for any
$n\in\omega\smallsetminus a$. This indicates that $x \notin^{J^\dual} s$ for any $s\in \cnt(I)$.
\end{proof}

We show in \Cref{sec:sumI} that the above is the only case when $\lfrak_\Kat(\star,J)$ is undefined.

To look at more monotonicity results, we generalize the orderings of \Cref{def:leqK} as follows.

\begin{definition}\label{def:ordU}
For $I\subseteq \pts(a)$ and $J\subseteq \pts(a')$, define
\begin{align*}
I\le_\Kat^\ccup J &\text{ iff }
I\subseteq \bigset{w\subseteq a}{\bigcup_{f\in F}f^{-1}\llbracket w\rrbracket\in J} \text{ for some non-empty $F\in[{}^{a'} a]^{<\omega}$,}\\
I\le_\mKat^\ccup J &\text{ iff } 
\bigset{w\subseteq a'}{\bigcup_{f\in F}f^{-1}\llbracket w\rrbracket\in I} \subseteq J \text{ for some non-empty $F\in[{}^a a']^{<\omega}$,}\\
I\le_\RK^\ccup J &\text{ iff } 
\bigset{w\subseteq a}{\bigcup_{f\in F}f^{-1}\llbracket w\rrbracket\in J} = I \text{ for some non-empty $F\in[{}^{a'}a]^{<\omega}$.}
\end{align*}
The orderings $\le_\KB^\ccup$, $\le_\mKB^\ccup$ and $\le_\RB^\ccup$ are obtained by demanding that all functions in $F$ are finite-to-one. We have ${\le_\Kat}\subseteq{\le_\Kat^\ccup}$, ${\le_\mKat}\subseteq{\le_\mKat^\ccup}$, ${\le_\KB}\subseteq{\le_\KB^\ccup}$, ${\le_\mKB}\subseteq{\le_\mKB^\ccup}$ and ${\le_\RB}\subseteq{\le_\RB^\ccup}$. We define $\approx^\ccup_{\Kat}$, etc., in the natural way.
\end{definition}

\begin{lemma}\label{L2.6}
 Let $I$ be an ideal on $\omega$. In the results below, $J$ runs on ideals on $\omega$. 
\begin{enumerate}[label=\rm(\alph*)]
\item\label{L2.6a}
$\pL_{I,J}$ is\/ $\le_\Kat$-increasing on $J$ and $\Lbf_{I,J}$ is\/ $\le_\mKat$-decreasing on $J$. Also,
$\pfrak_\Kat(I,J)$ and\/ $\pfrak_\Kat(\star,J)$ are\/ $\le_\Kat$-increasing\footnote{Up to here, it is not required that $I$ is an ideal.} and\/
$\lfrak_\Kat(I,J)$ and\/ $\lfrak_\Kat(\star,J)$ are\/ $\le_\mKat^\ccup$-decreasing on the parameter~$J$.
\item\label{L2.6b}  
$\pLc_J(I)$ is\/ $\le_\KB$-increasing and\/ $\Lc_J(I)$ is\/ $\le_\mKB$-decreasing on $J$. Also,
$\slalome(I,J)$ and\/ $\slalome(\star,J)$ are\/ $\le_\KB$-increasing and\/
$\slalomt(I,J)$ and\/ $\slalomt(\star,J)$ are\/ $\le_\mKB^\ccup$-decreasing on the parameter~$J$.
\end{enumerate}
\end{lemma}

\begin{proof}
    \ref{L2.6a}: We apply \Cref{thm:RB1,thm:RBuni} to $D=D'=\Baire$ and $E=E'=\cnt(I)$. 
    When $J'\leq_\Kat J$,
    the hypotheses of \Cref{thm:RB1}~\ref{it:RBincr} (with $s':=s$) are satisfied, where $f$ is a~function witnessing $J'\leqK J$, so $\pL_{I,J'}\leqT \pL_{I,J}$.

    When $J\leq^\ccup_\mKat J'$, the hypotheses of \Cref{thm:RBuni}~\ref{RBuni-1}, 
    with $F$ witnessing $J\leq^\ccup_\mKat J'$, are satisfied, so we get that $\Lbf_{I,J'}\leqT {}^{F\, }\Lbf_{I,J}$.

    \ref{L2.6b}: We apply \Cref{thm:RB1,thm:RBuni} to $D=D'=\Baire$ and $E=E'={}^\omega\I$. Since they satisfy the hypotheses of \Cref{thm:RB1}, where $f$ is a~function witnessing either $J'\leq_\KB J$ or $J\leq_\mKB J'$, we get that $\pLc_{J'}(I)\leqT \pLc_{J}(I)$  
    in the first case, and $\Lc_{J'}(I)\leqT \Lc_{J}(I)$ 
    in the second.

    The monotonicity result for $\le_\mKB^\ccup$ follows by \Cref{thm:RBuni}~\ref{RBuni-1} in a~similar way.
\end{proof}

In particular, when $I=\Fin$:

\begin{corollary}\label{monBaire}
On the parameter $J$, $\la\Baire,\Baire,{\leq^{J^\dual}}\ra$ is\/ $\leq_{\KB}$-decreasing and\/ $\leq_{\mKB}$-decreasing. Moreover:
\begin{enumerate}[label=\normalfont(\alph*)]
    \item $\bfrak_J$ is\/ $\leq_\KB$-increasing and\/  $\leq_{\mKB}^\ccup$-increasing.
    \item $\dfrak_J$ is\/ $\leq_\KB$-decreasing and\/  $\leq_{\mKB}^\ccup$-decreasing.
    \qed
\end{enumerate}
\end{corollary}

We list our results below in terms of equivalences.

\begin{theorem}\label{C2.7}
 Let $I$, $J$ and $J'$ be ideals on $\omega$. Then 
\begin{enumerate}[label=\rm(\alph*)]
\item\label{C2.7a}
If $J'\approx_\Kat J$, then\/
$\pL_\Kat(I,J)\eqT \pL_\Kat(I,J')$. In particular,
$\pfrak_\Kat(I,J')=\pfrak_\Kat(I,J)$ and\/
$\pfrak_\Kat(\star,J')=\pfrak_\Kat(\star,J)$.

\item\label{C2.7b}
If $J'\approx_\mKat J$, then\/
$\Lbf_\Kat(I,J) \eqT \Lbf_\Kat(I,J')$,\footnote{Up to here (including~\ref{C2.7a}), there is no need that $I$ is an ideal.} and whenever $J'\approx_\mKat^\ccup J$,
$\lfrak_\Kat(I,J')=\lfrak_\Kat(I,J)$ and\/
$\lfrak_\Kat(\star,J')=\lfrak_\Kat(\star,J)$.

\item
If $J'\approx_\KB J$, then\/
$\la\Baire,\Baire, {\leq^{J^\dual}}\ra \eqT \la\Baire,\Baire,{\leq^{J'^\dual}}\ra$ and\/ $\pLc(I,J)\eqT \pLc(I,J')$. In particular, 
$\bfrak_{J'}=\bfrak_J$, $\dfrak_{J'}=\dfrak_J$, $\slalome(I,J')=\slalome(I,J)$,
and\/ $\slalome(\star,J')=\slalome(\star,J)$.

\item\label{C2.7d}
If $J'\approx_\mKB J$, then\/ $\la\Baire,\Baire,{\leq^{J^\dual}}\ra \eqT \la\Baire,\Baire,{\leq^{J'^\dual}}\ra$ and\/ $\Lc(I,J)\eqT \Lc(I,J')$.
Moreover, 
if $J'\approx_\mKB^\ccup J$, then\/
$\bfrak_{J'}=\bfrak_J$, $\dfrak_{J'}=\dfrak_J$, $\slalomt(I,J')=\slalomt(I,J)$,
and\/ $\slalomt(\star,J')=\slalomt(\star,J)$.

\item\label{C2.7e}
If $J$ is an ideal with the Baire property, then\/ $\Lbf_{I,J} \eqT \Lbf_{I,\Fin}$, $\la \Baire,\Baire, {\leq^{J^\dual}}\ra \eqT \la \Baire,\Baire, {\leq^{\Fin^\dual}}\ra$ and\/ $\Lc(I,J) \eqT \Lc(I,\Fin)$. In particular,
$\lfrak_\Kat(I,J)=\lfrak_\Kat(\star,J)=\infty$,
$\bfrak_J=\bfrak$, $\dfrak_J=\dfrak$,\footnote{The equalities $\bfrak_J=\bfrak$, $\dfrak_J=\dfrak$ are well-known, see the introduction for the appropriate references.} $\slalomt(I,J)=\slalomt(I,\Fin)$, and\/
$\slalomt(\star,J)=\slalomt(\star,\Fin)$.
\end{enumerate}
\end{theorem}

\begin{proof}
    \ref{C2.7a}--\ref{C2.7d} follow by \Cref{L2.6} and \Cref{monBaire}.

    \ref{C2.7e}: Clearly, $\lfrak_\Kat(I,\Fin)=\lfrak_\Kat(\star,\Fin)=\infty$ by \Cref{L2.1}.
Since $J$~has the Baire property, $\Fin\le_\RB J$ by Jalali--Naini and Talagrand
theorem and $\Fin\le_\mKB J$ because $\Fin\subseteq J$.
Therefore $J\approx_\mKB\Fin$ and all Tukey equivalences are consequences of
\ref{C2.7b} and~\ref{C2.7d}.
\end{proof}

To compare the assumptions of \Cref{C2.7} \ref{C2.7d}--\ref{C2.7e} note that
\begin{equation*}
\text{``$J$~has the Baire property''}\Leftrightarrow
\Fin\le_\RB J\Leftrightarrow
J\le_\mKB \Fin\Leftrightarrow
\Fin\approx_\mKB J.
\end{equation*}
Consequently, if $J'$ has the Baire property, then
\begin{equation*}
\text{``$J$~has the Baire property''}\Leftrightarrow
J\le_\mKB J'\Leftrightarrow
J'\approx_\mKB J.
\end{equation*}

We now turn to slalom numbers of the form $\slalomt(b,h,J) = \dfrak_{\dfil{J}}(\prod b,\Scal(b,h))$ and $\slalome(b,h,J) = \dfrak_{\dfil{J}}(\prod b,\Scal(b,h))$ for $h\in\Baire$ and $b=\seqn{b(i)}{i<\omega}$ with $b(i)$ non-empty for all $i<\omega$. 
In the case when all $b(i)$ are countable, it is enough to study the case $b(i)\in\omega\cup\{\omega\}$ for all $i<\omega$ thanks to the following result.

\begin{lemma}\label{L3.1}
Let $J$ be an ideal on $\omega$. 
If\/ $\set{i<\omega}{|b'(i)|\leq |b(i)|}\in \dfil{J}$ and $h\le^{J^\dual} h'$, 
then\/
$\Lc_{J}(b',h')\leqT \Lc_{J}(b,h)$ and\/ $\pLc_{J}(b',h')\leqT \pLc_{J}(b,h)$. In particular,
$\slalome(b',h',J)\le\slalome(b,h,J)$ and\/
$\slalomt(b',h',J)\le\slalomt(b,h,J)$.
\end{lemma}

\begin{proof}
Let $w:= \set{i<\omega}{|b'(i)|\leq |b(i)| \text{ and }h(i)\leq h'(i)}\in\dfil{J'}$. For each $i\in w$, pick a~one-to-one function $f_i\colon b'(i)\to b(i)$, and for $i\in \omega\menos w$ pick any $f_i\colon b'(i)\to b(i)$. Apply \Cref{Kinterc} to $\bar f:= \la f_i\colon i<\omega\ra$ and $F:=\{\bar f\}$.
\end{proof}

We first look at monotonicity results.

\begin{lemma}\label{L3.1.2}
    If $J\subseteq J'\subseteq\pts(\omega)$, then\/ $\Lc_{J'}(b,h)\leqT \Lc_J(b,h)$ and\/ $\pLc_J(b,h)\leqT \pLc_{J'}(b,h)$. In particular, $\slalomt(b,h,J')\leq \slalomt(b,h,J)$ and\/ $\slalome(b,h,J)\leq \slalome(b,h,J')$.
\end{lemma}

\begin{proof}
Immediate from~\Cref{lem:monJ} \ref{mon:J}.
\end{proof}

\begin{lemma}\label{monbh}
  Let $f\in\Baire$ be a~finite-to-one function and let $J$ and $J'$ be ideals on $\omega$. For $n<\omega$, denote $a_n^f := f^{-1}\llbracket \{n\}\rrbracket$. 
  Let $h\in\Baire$, and consider functions $h^-_f,h'_f\in\Baire$ such that 
  \[
  \bigset{n\in \omega}{\text{whenever $a^f_n\neq\emptyset$, } h^-_f(n)\leq \min_{k\in a^f_n}h(k) \text{ and } \sum_{k\in a_n^f}h(k)\leq h'_f(n)} \in {J'}^\dual.
  \]
  Let $b=\seqn{b(n)}{n<\omega}$, $b^-_f = \seqn{b^-_f(n)}{n<\omega}$ and $b'_f = \seqn{b'_f(n)}{n<\omega}$ be sequences of non-empty sets such that 
  \[
  \bigset{n<\omega}{\left|\prod_{k\in a^f_n}b(k)\right|\leq |b^-_f(n)|}\in {J'}^\dual \text{ and } \bigset{k\in\omega}{|b(f(k))|\leq |b'_f(k)|}\in J^\dual.
  \]
  Then:
  \begin{enumerate}[label=\normalfont(\alph*)]
      \item\label{monbh.a} If $f$ witnesses that $J'\leq_{\KB} J$ then\/ $\Lc_J(b,h)\leqT \Lc_{J'}(b^-_f,h^-_f)$ and\/ $\pLc_{J'}(b,h'_f)\leqT \pLc_{J}(b'_f,h)$. In particular, $\slalomt(b,h,J)\leq \slalomt(b^-_f,h^-_f,J')$ and\/ $\slalome(b,h'_f,J')\leq \slalome(b'_f,h,J)$.

      \item\label{monbh.b} If $f$ witnesses $J\leq_{\mKB} J'$ then\/ $\pLc_J(b,h)\leqT \pLc_{J'}(b^-_f,h^-_f)$ and\/ $\Lc_{J'}(b,h'_f)\leqT \Lc_J(b'_f,h)$. In particular, $\slalome(b,h,J)\leq \slalome(b^-_f,h^-_f,J')$ and\/ $\slalomt(b,h'_f,J')\leq \slalomt(b'_f,h,J)$.
  \end{enumerate}
  Furthermore, let $F\subseteq \Baire$ be a~finite set of finite-to-one functions and consider $b'_f\colon \omega\to\omega+1\smallsetminus \{0\}$ and $h'_f$ for $f\in F$ as before, and define $h'_F\in \Baire$ by $h'_F(n):= \sum_{f\in F}h'_f(n)$.
  \begin{enumerate}[resume*]
      \item\label{monbh.c} If $F$ witnesses that $J\leq_\mKB^\ccup J'$ then\/ $\Lc_{J'}(b,h'_F)\leqT \bigotimes_{f\in F} \Lc_{J}(b'_f,h)$. In particular, $\slalomt(b,h'_F,J')\leq \prod_{f\in F}\slalomt(b'_f,h,J)$.
  \end{enumerate}
\end{lemma}

\begin{proof}
    Thanks to \Cref{L3.1}, we may assume that, for $n<\omega$, $b(n)$ is a non-zero cardinal number,
    \[
    b^-_f(n)= \prod_{k\in a^f_n}b(k),\ b'_f(n) = b(f(n)),
    \]
    and,  whenever $a^f_n\neq\emptyset$,
    \[
    h^-_f(n) = \min _{k\in a^f_n} h(k) \text{ and }h'_f(n) = \sum_{k\in a^f_n}h(k).
    \] 
    \ref{monbh.a}: For $\Lc$, it follows by \Cref{thm:RB2}~\ref{it:RBdcr} applied to $D=\prod b$, $D'=\prod b^-_f$, $E=\Scal(b,h)$ and $E' = \Scal(b^-_f,h^-_f)$; for $\pLc$, it follows by \Cref{thm:RB1}~\ref{it:RBincr} applied to $D' = \prod b$, $D = \prod b'_f$, $E=\Scal(b'_f,h)$ and $E'=\Scal(b,h'_f)$.

    \ref{monbh.b}: For $\pLc$, it follows by \Cref{thm:RB2}~\ref{it:RBbdcr} applied to $D=\prod b$, $D'=\prod b^-_f$, $E=\Scal(b,h)$ and $E' = \Scal(b^-_f,h^-_f)$; for $\Lc$, it follows by \Cref{thm:RB1}~\ref{it:RBbincr} applied to $D' = \prod b$, $D = \prod b'_f$, $E=\Scal(b'_f,h)$ and $E'=\Scal(b,h'_f)$.

    \ref{monbh.c}: It follows by \Cref{thm:RBuni}~\ref{RBuni-1} applied to $D_f = \prod b'_f$, $D = \prod b$, $E_f = \Scal(b'_f,h)$ and $E = \Scal(b,h'_F)$.
\end{proof}

In the case when $\prod b=\Baire$, the previous result applies to $\prod b^-_f = \prod b'_f =\Baire$. As a~consequence: 

\begin{lemma}\label{L3.2}
    Let $h\in\Baire$ and $J$ and $J'$ be ideals on $\omega$.
    \begin{enumerate}[label=\normalfont(\alph*)]
        \item\label{L3.2a} If $J'\leq_{\KB} J$ then there are $h^-,h'\in\Baire$ such that\/ $\Lc_J(h)\leqT \Lc_{J'}(h^-)$ and\/ $\pLc_{J'}(h')\leqT \pLc_J(h)$, in particular, $\slalomt(h,J)\leq \slalomt(h^-,J')$ and\/ $\slalome(h',J')\leq \slalome(h,J)$.

        \item\label{L3.2b} If $J\leq_{\mKB} J'$ then there are $h^-,h'\in\Baire$ such that\/ $\pLc_J(h)\leqT \pLc_{J'}(h^-)$ and\/ $\Lc_{J'}(h')\leqT \Lc_J(h)$. In particular, $\slalome(h,J)\leq \slalome(h^-,J')$ and\/ $\slalomt(h',J')\leq \slalomt(h,J)$.

        \item\label{L3.2c} If $J'\leq_{\RB} J$ then there are $h^-,h'\in\Baire$ satisfying the statements in~\ref{L3.2a} and~\ref{L3.2b}. 

        \item\label{L3.2d} If $J\leq^\ccup_{\mKB} J'$ then there is some $h'\in\Baire$ such that\/ $\Lc_{J'}(h')\leqT {}^n \Lc_J(h)$ for some $0<n<\omega$, in particular, $\slalomt(h',J')\leq \slalomt(h,J)$.
    \end{enumerate}
    Moreover, if $h$ diverges to infinity, $h^-$ can be found diverging to infinity, and when $h\geq^* 1$, $h^-$ can be found such that $h^-\geq^* 1$. On the other hand, $h'$ can be found diverging to infinity (and as increasing as desired). 
\end{lemma}

\begin{proof}
    If $f\colon\omega\to \omega$ is a~finite-to-one function witnessing the relation between $J$ and $J'$ indicated in~\ref{L3.2a}--\ref{L3.2c}, then $h^- = h^-_f$ and $h' = h'_f$ can be defined by
    \[
    h^-(n):= 
    \begin{cases}
        \min_{k\in a^f_n} h(k),&\text{if $a^f_n \neq \emptyset$,}\\
        n,&\text{otherwise,}
    \end{cases}
    \qquad
    h'(n):= 
    \begin{cases}
        \sum_{k\in a^f_n} h(k),&\text{if $a^f_n \neq \emptyset$,}\\
        n,&\text{otherwise,}
    \end{cases}
    \]
    where $a^f_n := f^{-1}\llbracket \{n\} \rrbracket$. Then~\ref{L3.2a} and~\ref{L3.2b} follows by \Cref{monbh},~\ref{L3.2c} follows by~\ref{L3.2a} and~\ref{L3.2b}, and~\ref{L3.2d} follows by \Cref{monbh}~\ref{monbh.c}.

    It is clear that $h^-$ and $h'$ diverge to $\infty$ when $h$ does, and that $h\geq^* 1$ implies $h^-\geq^*1$ and $h'\geq^* 1$. 
    But thanks to \Cref{L3.1}, $h'$ can be enlarged as desired.
\end{proof}

Although $\slalomt(h,J)$ and $\slalome(h,J)$ are well-defined and uncountable when $h\geq^{\dfil{J}}1$, with the additional parameter~$b$ we get cases when $\slalomt(b,h,J)$ is finite, likewise for $\slalome(b,h,J)$.

\begin{lemma}\label{aboutbh}
    Let $b=\seqn{b(n)}{n<\omega}$ be a~sequence of non-empty sets and $h\in\Baire$.
    \begin{enumerate}[label = \normalfont (\alph*)]
        \item\label{bh0} $\slalome(b,h,J)$ is well-defined iff $\|h\geq 1\|\in J^+$, and $\slalomt(b,h,J)$ is well-defined iff $\|h\geq 1\|\in J^\dual$.
        \item\label{bh1} For $0<k<\omega$, $\slalome(b,h,J) \leq k$ iff $\set{n<\omega}{|b(n)|\leq k\, h(n)} \in J^+$.
        \item\label{bh2} $\slalome(b,E,J)\neq \aleph_0$ for any set of slaloms $E$.
        \item\label{bh3} For $0<k<\omega$, if $\slalomt(b,h,J)\leq k$ then $\set{n<\omega}{|b(n)|> k\, h(n)}\in J$.
        \item\label{bh4} $\slalomt(b,h,J) = 1$ iff  $\set{n<\omega}{|b(n)|> h(n)} \in J$.
        \item\label{bh5} $\slalomt(b,h,J)=2$ iff $\set{n<\omega}{|b(n)|> 2h(n)} \in J$, $a:= \set{n<\omega}{h(n) < |b(n)| \leq 2h(n)} \in J^+$ and $J\cap\pts(a)$ is a~maximal ideal on~$a$.
        \item\label{bh6} For any set of slaloms $E$, either $\slalomt(b,E,\Fin)=1$ or $\slalomt(b,E,\Fin)>\aleph_0$.
    \end{enumerate}
\end{lemma}
\begin{proof}
    Let $c_k:=\set{n<\omega}{|b(n)|\leq k\, h(n)}$ for $0<k<\omega$. For $n\in c_k$, let $\seqn{s^k_\ell(n)}{\ell<k}\subseteq [b(n)]^{\leq h(n)}$ be a~covering of $b(n)$, and let $s^k_\ell(n):=\emptyset$ for $n\in \omega\menos c_k$. This defines $s^k_\ell \in \Scal(b,h)$ for $\ell<k$.

    \ref{bh0}: If $\|h\geq 1\|\in J^+$ then, for any $x\in \prod b$, there is some $s\in \Scal(b,h)$ such that $x(n)\in s(n)$ for all $n\in \|h\geq 1\|$, namely $s(n):=\{x(n)\}$ when $n\in\|h\geq 1\|$, and $s(n):= \emptyset$ otherwise, so $x\in^{J^+} s$; and if $\|h\geq 1\|\in J$ then $x\notin^{J^+}s$ for all $x\in\prod b$ and $s\in \Scal(b,h)$, since $n\notin \|h\geq 1\|$ implies that $s(n) = \emptyset$. A similar argument works for $\slalomt(b,h,J)$.

    \ref{bh1}: If $c_k\in J^+$ then, for any $x\in\prod b$, there is some $\ell<k$ such that $x\in^{J^+} s^k_\ell$, so $\slalome(b,h,J)\leq k$. Conversely, if $c_k\in J$ then, for any $S\subseteq\Scal(b,h)$ of size ${\leq}k$ and $n\in \omega\smallsetminus c_k$, $\bigcup_{s\in S}s(n) \subsetneq b(n)$, so there is some $x(n)\in b(n)$ outside this union. Then $x(n)\notin s(n)$ for all $n\in\omega\smallsetminus c_k$ and $s\in S$, i.e., $x\notin^{J^+} s$. 

    \ref{bh2}:   
    Let $S=\set{s^\ell}{\ell<\omega}\subseteq E$ and assume that $\slalome(b,E,J)$ is infinite. For each $m<\omega$, we can find some $x^m\in\prod b$ and a~set $a_m\in J^\dual$ such that $x^m(n)\notin s^\ell(n)$ for all $n\in a_m$ and $\ell\leq m$. By taking intersections if necessary, we may assume that $a_{m+1}\subseteq a_m$. Define $x\in\prod b$ such that, for any $n\in a_0$, $x(n):= x^{m_n}(n)$ where $m_n:=\max\set{m\leq n}{n\in a_m}$. Fix $\ell<\omega$. For $n\in a_{\ell}\smallsetminus \ell$, $m_n\geq \ell$, so $x(n) = x^{m_n}(n) \notin s^\ell(n)$, hence $x\notin^{J^+} s^\ell$.

    \ref{bh3}: Assume that $\omega\menos c_k\in J^+$. Then, like in the second part of the proof of~\ref{bh1}, for any $S\subseteq \Scal(b,h)$ of size~${\leq} k$ we can find some $x\in\prod b$ such that $x\notin^{J^\dual} s$ for all $s\in S$.

    \ref{bh4}: One implication follows from~\ref{bh3}. For the converse, if $\omega\menos c_1\in J$, i.e., $c_1\in J^d$, then there is some $s\in \Scal(b,h)$ such that $s(n) = b(n)$ for all $n\in c_1$. Then $x\in^{J^\dual} s$ for all $x\in\prod b$.

    \ref{bh5}: First note that $a = c_2\menos c_1$. 
    Assume that $\slalomt(b,h,J)=2$ witnessed by $\{t^0,t^1\}$. Then, by~\ref{bh3} and~\ref{bh4}, $\omega\menos c_1\in J^+$, $c_2\in J^\dual$ and $a\in J^+$. Without loss of generality we may assume that, for $n\in c_1$, $t^0(n) = t^1(n) = b(n)$ and, for $n\in c_2$, $t^0(n)\cup t^1(n) = b(n)$ (because $\set{n\in\omega}{t^0(n)\cup t^1(n) \neq b(n)}\in J$). Now let $a'\subseteq a$. Define $x\in\prod b$ such that, for $n\in a$, $x(n)\in t^1(n)\menos t^0(n)$ when $n\in a'$, and $x(n)\in t^0(n)\menos t^1(n)$ when $n\in a\menos a'$. Then, $\|x\in t^e\|\in J^\dual$ for some $e\in\{0,1\}$. If $e=0$ then $\|x\in t^0\|\cap a' = \emptyset$, so $a'\in J$; and if $e=1$, $\|x\in  t^1\|\cap a\subseteq a'$, so $a\menos a' \in J$. This shows that $J\cap \pts(a)$ is a~maximal ideal on~$a$.

    For the converse, assume that $c_2\in J^\dual$, $a\in J^+$ and $J\cap \pts(a)$ is a~maximal ideal on $a$. Then, by~\ref{bh4}, $\slalomt(b,h,J) >1$.
    So it remains to show that $\{s^2_0,s^2_1\}$ witnesses that $\slalomt(b,h,J)= 2$ by further assuming that $s^2_0(n) = s^2_1(n) = b(n)$ for $n\in c_1$. Let $x\in \prod b$ and $a':=\set{n\in a}{x(n)\in s^2_0(n)}$. Since $J\cap \pts(a)$ is a~maximal ideal on $a$, either $a'\in J$ or $a\menos a'\in J$. The first case implies $x\in^{J^\dual} s^2_1$, while the second implies $x\in^{J^\dual} s^2_0$.

    \ref{bh6}: Let $S=\set{s^\ell}{\ell<\omega}\subseteq E$ and assume that $\slalomt(b,E,\Fin)>1$. For each $\ell<\omega$, we can find some $x^\ell\in\prod b$ and a~set $a_\ell\in [\omega]^{\aleph_0}$ such that $x^\ell(n)\notin s^\ell(n)$ for all $n\in a_\ell$. We can easily construct a~$x\in\prod b$ such that, for any $\ell<\omega$, $x(n)= x^{\ell}(n)$ for infinitely many $n\in a_\ell$, so $x\notin^{\Fin^\dual} s^\ell$.
\end{proof}

It looks harder to characterize $\slalomt(b,h,J) = k$ when $3\leq k <\omega$.

The slalom numbers of the form $\slalomt(h,\Fin)$ are $\slalome(h,\Fin)$ are very well-known as they characterize other classical cardinal characteristics of the continuum.

\begin{theorem}\label{F3.3}
Let $g,h\in{}^\omega\omega$. Then 
\begin{enumerate}[label=\rm(\alph*)]
\item {\rm(Bartoszy\'nski~\cite{Ba1987} and Miller~\cite{Mi1982}, see also \cite[Thm.~5.1 and~3.17]{CM23})}\label{F3.3a}
$\nonm=\slalome(h,\Fin)$ and\/ $\covm = \slalome^\perp(h,\Fin)$ 
when $h\ge^{\Fin^+} 1$, moreover, $\pLc_\Fin(h)\eqT \pLc_\Fin(1)$.
\item{\rm(Bartoszy\'nski~\cite{Ba1984},  see also \cite[Thm.~4.2]{CM23})}\label{F3.3b}
$\Lc_\Fin(g) \eqT \Nwf$ when\/
$\lim_{n\to \infty}g(n)=\infty$. As a~consequence,
$\slalomt(g,\Fin)=\cofn$ and\/ $\slalomt^\perp(g,\Fin)=\addn$. 
\qed
\end{enumerate}
\end{theorem}

As a~consequence of this theorem and \Cref{L3.1}, we get that $\nonm\leq \slalome(h,J)$ and $\slalomt(g,J)\leq \cofn$ for any ideal $J$ on $\omega$ when $h\geq^{J^+}1$ and $g$ diverges to $\infty$. On the other hand, $\slalome(1,J_*)=\slalomt(1,J_*)=\cfrak$ when $J_*$ is a maximal ideal on $\omega$ (because of the size of an ultrapower of $\omega$, see~\cite[Prop.~5.44]{foreman}). Also recall that $\slalomt(h,\Fin)=\cfrak$ and $\slalomt^\perp(h,\Fin)$ is finite when $h\geq^*1$ and $h$ does not diverge to infinity (see~\cite[Thm.~3.12]{CM23}).

The ideal $\Fin$ in \Cref{F3.3} can be replaced by any ideal with the
Baire property:

\begin{theorem}\label{L3.4}
Let $J$ be an ideal on~$\omega$ with the Baire property and let
$h,g\in{}^\omega\omega$ be such that $h\ge^*1$ and\/
$\lim_{n\in\omega}g(n)=\infty$.
Then\/ $\pLc_J(h) \eqT \pLc_\Fin(1)$ and\/ $\Lc_J(g) \eqT \Nwf$, in particular\/
$\slalome(h,J)=\nonm$, $\slalomt(g,J)=\cofn$, $\slalomt^\perp(g,J) =\addn$ and\/ $\slalome^\perp(h,J)= \covm$.\footnote{The first equality is shown in~\cite{SoDiz}.}
\end{theorem}

\begin{proof}
Note that $\Fin\le_\RB J$ (by Jalali--Naini and Talagrand, see e.g.~\cite{farah}). Hence, by \Cref{L3.2}~\ref{L3.2c}, there are $g',g^-,h',h^-\in\Baire$ such that $\Lc_\Fin(g') \leqT \Lc_J(g) \leqT \Lc_{\Fin}(g^-)$ and $\pLc_\Fin(h') \leqT \pLc_J(h) \leqT \pLc_\Fin(h^-)$, even more, $g'$ and $g^-$ can be found diverging to infinity and $h',h^-\geq^*1$. Therefore, by \Cref{F3.3}, $\Lc_J(g) \eqT \Nwf$ and $\pLc_J(h) \eqT \pLc_\Fin(1)$.
\end{proof}

\begin{figure}[ht]
\begin{center}
\begin{tikzpicture}[scale=0.8]
\node (ale) at (-7, -3.5) {$\aleph_1$};
\node (a) at (-5, -3.5) {$\pp$};
\node (as) at (-5, -1) {$\sla{\I,\fin}$};
\node (b) at (-5, 1.5) {$\bb$};
\node (ba) at (-5, 4) {$\nonm$};
\node (aa) at (-2, -3.5) {$\slamg{\J}$};
\node (aas) at (-2, -1) {$\sla{\I,\J}$};
\node (bb) at (-2, 1.5) {$\bb_\J$};
\node (bba) at (-2, 4) {$\sla{h,\J}$};
\node (c) at (1, -3.5) {$\dslaml{\J}$};
\node (cs) at (1, -1) {$\dsla{\I,\J}$};
\node (f) at (1, 1.5) {$\dd_\J$};
\node (csa) at (1, 4) {$\dsla{h,\J}$};
\node (xpf) at (4, 1.5) {$\dd$};
\node (xpc) at (4, -3.5) {$\covm$};
\node (xpcs) at (4, -1) {$\dslago{\I}$};
\node (xpfa) at (4, 4) {$\cofn$};
\node (cont) at (6, 4) {$\cc$};
\foreach \from/\to in {aas/bb,aas/cs,cs/f,aas/cs} \draw [line width=.15cm,
white] (\from) -- (\to);
\foreach \from/\to in {ale/a,a/as, aa/aas, c/cs, b/bb, a/aa, aa/c, bb/f, as/b, aas/bb, cs/f, as/aas, aas/cs,f/xpf,cs/xpcs,c/xpc,xpcs/xpf,xpc/xpcs, b/ba,bb/bba,f/csa,xpf/xpfa,ba/bba,bba/csa,csa/xpfa,xpfa/cont} \draw [->] (\from) -- (\to);

\end{tikzpicture}
\end{center}
\caption{Inequalities among particular cases of slalom numbers for ideals $I$ and $J$ on $\omega$ and $h\to\infty$. An arrow denotes $\leq$. When $h\geq^*1$ and $h$ does not diverge to $\infty$, $\cofn$ should be replaced by $\slalomt(h,\Fin)=\cfrak$.}
\label{BasicDia}
\end{figure}

\begin{figure}[ht]
\begin{center}
\begin{tikzpicture}[scale=0.85]
\node (as) at (-5, -1) {$\pLc_\Fin(I)$};
\node (b) at (-5, 1.5) {$\pLc_\Fin(\Fin)$};
\node (ba) at (-5, 4) {$\pLc_\Fin(h)$};
\node (aas) at (-2, -1) {$\pLc_J(I)$};
\node (bb) at (-2, 1.5) {$\pLc_J(\Fin)$};
\node (bba) at (-2, 4) {$\pLc_J(h)$};
\node (cs) at (1, -1) {$\Lc_J(I)$};
\node (f) at (1, 1.5) {$\Lc_J(\Fin)$};
\node (csa) at (1, 4) {$\Lc_J(h)$};
\node (xpf) at (4, 1.5) {$\Lc_{\Fin}(\Fin)$};
\node (xpcs) at (4, -1) {$\Lc_\Fin(I)$};
\node (xpfa) at (4, 4) {$\Lc_\Fin(h)$};
\foreach \from/\to in {aas/bb,aas/cs,cs/f,aas/cs} \draw [line width=.15cm,
white] (\from) -- (\to);
\foreach \from/\to in {b/bb, bb/f, as/b, aas/bb, cs/f, as/aas, aas/cs,f/xpf,cs/xpcs, xpcs/xpf, b/ba,bb/bba,f/csa,xpf/xpfa,ba/bba,bba/csa,csa/xpfa} \draw [->] (\from) -- (\to);

\end{tikzpicture}
\end{center}
\caption{Diagram of Tukey connections for ideals $I$ and $J$ on $\omega$ and $h\geq^*1$. An arrow denotes $\leqT$.}
\label{TukeyDia}
\end{figure}

\Cref{BasicDia} summarizes the inequalities between slalom numbers of the form $\slalomt(I,J)$ and $\slalome(I,J)$ with other cardinals from Cicho\'n's diagram. These are consequences of our results so far (including monotonicity). The upper part of the diagram can be obtained via Tukey connections, see \Cref{TukeyDia}, which implies \Cref{DiaDual}.

\begin{figure}[ht]
\begin{center}
\begin{tikzpicture}[scale=0.8]
\node (a) at (-5, -1) {$\addn$};
\node (b) at (-5, 1.5) {$\bb$};
\node (ba) at (-5, 4) {$\nonm$};
\node (aa) at (-2, -1) {$\slalomt^\perp(h,J)$};
\node (bb) at (-2, 1.5) {$\bb_\J$};
\node (bba) at (-2, 4) {$\sla{h,\J}$};
\node (c) at (1, -1) {$\slalome^\perp(h,J)$};
\node (f) at (1, 1.5) {$\dd_\J$};
\node (csa) at (1, 4) {$\dsla{h,\J}$};
\node (xpf) at (4, 1.5) {$\dd$};
\node (xpc) at (4, -1) {$\covm$};
\node (xpfa) at (4, 4) {$\cofn$};
\node (cont) at (6, 4) {$\cc$};
\foreach \from/\to in {aas/bb,aas/cs,cs/f,aas/cs} \draw [line width=.15cm,
white] (\from) -- (\to);
\foreach \from/\to in {a/b, aa/bb, c/f, b/bb, a/aa, aa/c, bb/f, f/xpf, c/xpc, xpc/xpf, b/ba,bb/bba,f/csa,xpf/xpfa,ba/bba,bba/csa,csa/xpfa,xpfa/cont} \draw [->] (\from) -- (\to);

\end{tikzpicture}
\end{center}
\caption{Inequalities among particular cases of slalom numbers for ideals $I$ and $J$ on $\omega$ and $h\to\infty$. An arrow denotes $\leq$. When $h\geq^*1$ and $h$ does not diverge to $\infty$, $\cofn$ should be replaced by $\slalomt(h,\Fin)=\cfrak$ and $\addn$ by $\slalomt^\perp(h,\Fin)$, which is finite.}
\label{DiaDual}
\end{figure}

\section{Disjoint sum of ideals}\label{sec:sumI}

In this section, we look at the slalom numbers associated with the disjoint sum of ideals. They have a~nice behavior and are very useful to prove consistency results as in \Cref{sec:forcing}. Applications are also available in~\cite{GaMe}.

\begin{definition}\label{def:oplus}
\ 
\begin{enumerate}[label=\rm(\arabic*)]
    \item\label{oplus1} For sets $a_0$ and $a_1$,
denote $a_0\oplus a_1=(a_0\times\{0\})\cup(a_1\times\{1\})$. When $A_0$ and $A_1$ are families of sets, define \[A_0\oplus A_1=\{w_0\oplus w_1 : w_0\in A_0\text{ and }w_1\in A_1\}.\]
When $I_0$ and $I_1$ are ideals on $a_0$ and $a_1$, respectively, we refer to $I_0\oplus I_1$ as a~\emph{disjoint sum of ideals}, which is an ideal on $a_0\oplus a_1$.
    \item For arbitrary two functions $f_0$ and~$f_1$, define
the function $f_0\oplus f_1$ with domain $\dom f_0 \oplus\dom f_1$ by
$(f_0\oplus f_1)(n,i)=f_i(n)$.
Conversely, to every function $f$ with domain $a_0\oplus a_1$
assign two functions $(f)_0$ and~$(f)_1$ with domain~$a_0$ and $a_1$, respectively, such that
$f=(f)_0\oplus(f)_1$.

    \item If $D_0$ and $D_1$ are sets of functions, we define\footnote{There will not be confusion with~\ref{oplus1} from the context.}
\[
D_0\oplus D_1 := \set{f_0\oplus f_1}{f_0\in D_0,\ f_1\in D_1}.
\]
\end{enumerate}
Note that, When $I_0\subseteq\pts(a_0)$ and $I_1\subseteq\pts(a_1)$,  
$(I_0\oplus I_1)^\dual = I^\dual_0 \oplus I^\dual_1$ and $(I_0\oplus I_1)^{\rm c} = (I^{\rm c}_0\oplus \pts(a_1))\cup (\pts(a_0)\oplus I^{\rm c}_1)$.
\end{definition}

\begin{lemma} \label{L_op_ord}
Let $I_0$ and $I_1$ be ideals on a set $a$. Then 
\begin{enumerate}[label=\rm(\alph*)]
\item\label{L_op_ord-b} $I_0\cap I_1\le_\RB I_0\oplus I_1$, $I_0\le_\RB I_0\oplus I_0$ and $I_0\approx_\RB I_0\oplus\Pcal(a) \approx_\RB \Pcal(a)\oplus I_0$.
\item\label{L_op_ord-d} $I_0\oplus I_1\le_\KB I_0, I_1$.
\item\label{L_op_ord-c} $I_0\oplus I_1\le_\KB^\ccap I_0\cap I_1$ and $I_0\oplus I_0\le_\RB^\ccup I_0$.
\end{enumerate}
\end{lemma}

\begin{proof}
The mappings $(n,i)\to n$, $n\to (n,0)$ and $n\to (n,1)$ can be used to prove the relations in~\ref{L_op_ord-b}--\ref{L_op_ord-c}. For example,~\ref{L_op_ord-c} uses the set of the last two maps.
%
%
%
%
\end{proof}

From now on, in our slalom relational systems we may consider that the domain of an ideal and of the functions and slaloms in consideration may not be $\omega$ but some other countable set like $\omega\oplus\omega$, so that we can discuss slalom numbers for ideals of the form, e.g.\ $J_0\oplus J_1$ (see also \Cref{rem:gengen}). The same conventions fixed in \Cref{def:Lc} apply, for example, omitting $D$ when $D = {}^{\omega\oplus \omega} \omega$.

For sums of ideals, we have the following general result.

\begin{lemma}\label{lem:prod}
 Let $D_0, D_1, E_0, E_1$ be sets of functions with domain $\omega$, and let $J_0,J_1\subseteq\pts(\omega)$. Then
\begin{enumerate}[label=\rm(\alph*)]
    \item $\Lc_{J_0\oplus J_1}(D_0\oplus D_1, E_0\oplus E_1) \eqT \Lc_{J_0}(D_0, E_0) \otimes \Lc_{J_1}(D_1, E_1)$.
    In particular,
    \[
    \max\{\slalomt(D_0,E_0,J_0), \slalomt(D_1,E_1,J_1)\}\leq \slalomt(D_0\oplus D_1, E_0\oplus E_1,J_0\oplus J_1)\leq \slalomt(D_0,E_0,J_0)\cdot \slalomt(D_1,E_1,J_1).
    \]

    \item $\pLc_{J_0\oplus J_1}(D_0\oplus D_1, E_0\oplus E_1) \eqT \pLc_{J_0}(D_0, E_0) \boxtimes \pLc_{J_1}(D_1, E_1)$.
    In particular,
    \[
    \slalome(D_0\oplus D_1, E_0\oplus E_1,J_0\oplus J_1) = \min\{\slalome(D_0,E_0,J_0), \slalome(D_1,E_1,J_1)\}.
    \]
\end{enumerate}
\end{lemma}
\begin{proof}
The Tukey connections are constructed using the canonical bijections $F\colon D_0\times D_1 \to D_0\oplus D_1$ and $G\colon E_0\times E_1 \to E_0\oplus E_1$. Note that 
\begin{align*}
F(x_0,x_1) \in^{(J_0\oplus J_1)^\dual} G(y_0,y_1) & \text{ iff }x_0 \in^{J_0^\dual} y_0 \text{ and }x_1 \in^{J_1^\dual} y_1,\\
F(x_0,x_1) \in^{(J_0\oplus J_1)^{\rm c}} G(y_0,y_1) & \text{ iff } x_0 \in^{J_0^{\rm c}} y_0 \text{ or }x_1 \in^{J_1^{\rm c}} y_1.
\end{align*}
The rest follows by \Cref{products}.
\end{proof}

We start looking at more particular cases. As a~direct consequence of the previous result:

\begin{fact}\label{ftc:oplus}
Let $J_0,J_1\subseteq\pts(\omega)$, $b_0$ and $b_1$ functions with domain $\omega$ and $h_0, h_1\in\Baire$. Then:
\begin{align*}
    \Lc_{ J_0 \oplus J_1}(b_0\oplus b_1, h_0\oplus h_1) &  = \Lc_{J_0 \oplus J_1}\left(\prod b_0\oplus \prod b_1,\Scal(b_0,h_0)\oplus\Scal(b_1,h_1)\right)\\
     & \eqT \Lc_{J_0}(b_0,h_0) \otimes \Lc_{J_1}(b_1,h_1),\\
    \pLc_{ J_0 \oplus J_1}(b_0\oplus b_1, h_0\oplus h_1) &  = \pLc_{ J_0 \oplus J_1} \left(\prod b_0\oplus \prod b_1,\Scal(b_0,h_0)\oplus\Scal(b_1,h_1)\right)\\
     & \eqT \pLc_{J_0}(b_0,h_0) \boxtimes \pLc_{J_1}(b_1,h_1).
\end{align*}
In particular,
\begin{align*}
    \max\{\slalomt(b_0,h_0,J_0),\slalomt(b_1,h_1,J_1)\} \leq \slalomt(b_0\oplus b_1,h_0\oplus h_1,J_0\oplus J_1) & \leq \slalomt(b_0,h_0,J_0)\cdot \slalomt(b_1,h_1,J_1),\\
    \slalome(b_0\oplus b_1,h_0\oplus h_1,J_0\oplus J_1) & =
\min\{\slalome(b_0,h_0,J_0),\slalome(b_1,h_1,J_1)\}.\qed
\end{align*}
\end{fact}

\begin{lemma}\label{bL3.5}
    Let $J_0$ and $J_1$ be ideals on $\omega$, $b_0$ and $b_1$ functions with domain $\omega$ into the non-zero cardinal numbers, and $h_0, h_1\in\Baire$. Then:
    \begin{enumerate}[label = \normalfont (\alph*)]
        \item\label{bL4.11b2} $\Lc_{J_0\cap J_1}(\inf\{b_0,b_1\},h_0+h_1) \leqT \Lc_{ J_0 \oplus J_1}(b_0\oplus b_1, h_0\oplus h_1) \leqT \Lc_{J_0\cap J_1}(b_0\cdot b_1,\inf\{h_0,h_1\})$ and likewise for $\pLc$. In particular,
        \begin{align*}
        \slalomt(\inf\{b_0,b_1\},h_0+h_1,J_0\cap J_1) & \le\slalomt(b_0\oplus b_1,h_0\oplus h_1,J_0\oplus J_1)\le
        \slalomt(b_0\cdot b_1,\inf\{h_0,h_1\},J_0\cap J_1)\\
        \slalome(\inf\{b_0,b_1\},h_0+h_1,J_0\cap J_1) & \le\slalome(b_0\oplus b_1,h_0\oplus h_1,J_0\oplus J_1)\le
        \slalome(b_0\cdot b_1,\inf\{h_0,h_1\},J_0\cap J_1).
        \end{align*}

        \item\label{bL4.11c2} $\Lc_{J_0\cap J_1}(b_0,2h_0) \leqT \Lc_{ J_0 \oplus J_1}(b_0\oplus b_0, h_0\oplus h_0) \leqT \Lc_{J_0\cap J_1}(b_0^2,h_0)$ and likewise for $\pLc$. In particular,
        \begin{align*}
          \slalomt(b_0,2h_0,J_0\cap J_1) & \le\slalomt(b_0\oplus b_0,h_0\oplus h_0,J_0\oplus J_1)\le
          \slalomt(b_0^2,h_0,J_0\cap J_1)\\
          \slalome(b_0,2h_0,J_0\cap J_1) & \le\slalome(b_0\oplus b_0,h_0\oplus h_0,J_0\oplus J_1) \leq \slalome(b^2_0,h_0,J_0\cap J_1).
        \end{align*}
    \end{enumerate}
\end{lemma}

\begin{proof}
\ref{bL4.11b2}: 
Consider the function $\pi\colon\omega\oplus\omega\to\omega$ defined by $\pi(n,i):=n$ and note that $J_0\cap J_1=\pi^\rightarrow(J_0\oplus J_1)$, so $J_0\cap J_1\le_\RB J_0\oplus J_1$ (see \Cref{L_op_ord}~\ref{L_op_ord-b}). Also notice that $\pi^{-1}\llbracket \{n\}\rrbracket = \{(n,0),(n,1)\}$.
\Cref{monbh} can be applied: we can use
    \begin{align*}
        b^-_\pi(n) & := \prod_{u\in \pi^{-1}\llbracket\{n\}\rrbracket} (b_0\oplus b_1)(u) = b_0(n)\cdot b_1(n),\\
        h^-_\pi(n) & :=\min\set{(h_0\oplus h_1)(u)}{u\in\pi^{-1}\llbracket\{n\}\rrbracket} = \min\{h_0(n),h_1(n)\},\\
        h'_\pi (n) & := \sum_{u\in\pi^{-1}\llbracket\{n\}\rrbracket}(h_0\oplus h_1)(u) = h_0(n)+h_1(n).
    \end{align*}
    We also require a~function $b$ with domain $\omega$ that allows $b'_\pi = b_0\oplus b_1$ when applying \Cref{monbh}, i.e., satisfying $|b(\pi(n,i))| \leq (b_0\oplus b_1)(n,i)$ for all $(n,i)\in \omega\oplus \omega$. This is equivalent to $|b(n)|\leq b_e(n)$ for all $n<\omega$ and $e\in\{0,1\}$, so $b(n):=\min\{b_0(n),b_1(n)\}$ works.

\ref{bL4.11c2}: 
Immediate from~\ref{bL4.11b2} applied to $b_1:=b_0$ and $h_1:= h_0$.
\end{proof}

As a~direct consequence, we get:

\begin{corollary}\label{L3.5}
Let $J_0$ and $J_1$ be ideals on $\omega$ and $h_0, h_1\in\Baire$. Then
\begin{enumerate}[label=\rm(\alph*)]
\item\label{L3.5a} $\Lc_{J_0\oplus J_1}(h_0\oplus h_1)\eqT \Lc_{J_0}(h_0) \otimes \Lc_{J_1}(h_1)$ and $\pLc_{J_0\oplus J_1}(h_0\oplus h_1)\eqT \pLc_{J_0}(h_0) \boxtimes \pLc_{J_1}(h_1)$. In particular,
\begin{align*}
\slalomt(h_0\oplus h_1,J_0\oplus J_1) & =
\max\{\slalomt(h_0,J_0),\slalomt(h_1,J_1)\},\\ \slalomt^\perp(h_0\oplus h_1,J_0\oplus J_1) & =
\min\{\slalomt^\perp(h_0,J_0),\slalomt^\perp(h_1,J_1)\},\\
\slalome(h_0\oplus h_1,J_0\oplus J_1) & =
\min\{\slalome(h_0,J_0),\slalome(h_1,J_1)\}, \\
\slalome^\perp(h_0\oplus h_1,J_0\oplus J_1) & =
\max\{\slalome^\perp(h_0,J_0),\slalome^\perp(h_1,J_1)\}.
\end{align*}

\item
$\Lc_{J_0\cap J_1}(h_0+h_1) \leqT \Lc_{ J_0 \oplus J_1}(h_0\oplus h_1) \leqT \Lc_{J_0\cap J_1}(\inf\{h_0,h_1\})$ and likewise for $\pLc$. In particular,
        \begin{align*}
        \slalomt(h_0+h_1,J_0\cap J_1) & \le\slalomt(h_0\oplus h_1,J_0\oplus J_1)\le
        \slalomt(\inf\{h_0,h_1\},J_0\cap J_1)\\
        \slalome(h_0+h_1,J_0\cap J_1) & \le\slalome(h_0\oplus h_1,J_0\oplus J_1)\le
        \slalome(\inf\{h_0,h_1\},J_0\cap J_1).
        \end{align*}

\item\label{L3.5c}
$\Lc_{J_0\cap J_1}(2h_0) \leqT \Lc_{ J_0 \oplus J_1}(h_0\oplus h_0) \leqT \Lc_{J_0\cap J_1}(h_0)$ and likewise for $\pLc$. In particular, 
        \begin{align*}
          \slalomt(2h_0,J_0\cap J_1) & \le\slalomt(h_0\oplus h_0,J_0\oplus J_1)\le
          \slalomt(h_0,J_0\cap J_1)\\
          \slalome(2h_0,J_0\cap J_1) & \le\slalome(h_0\oplus h_0,J_0\oplus J_1)= \slalome(h_0,J_0\cap J_1).
        \end{align*}
\end{enumerate}
\end{corollary}
\begin{proof}
    Immediate from \Cref{bL3.5} except $\slalome(h_0,J_0\cap J_1) \leq \slalome(h_0\oplus h_0,J_0\oplus J_1)$ in~\ref{L3.5c}. This follows because, by \ref{L3.5a}, $\slalome(h_0\oplus h_0,J_0\oplus J_1) = \min\{\slalome(h_0,J_0),\slalome(h_0,J_1)\}$, and both $\slalome(h_0,J_0)$ and $\slalome(h_0,J_1)$ are above $\slalome(h_0,J_0\cap J_1)$ by \Cref{L3.1.2}.
\end{proof}

We now look at the slalom numbers coming from $\Lc_J(I)$, $\pLc_J(I)$, $\Lbf_{I,J}$ and $\pL_{I,J}$.

\begin{theorem}\label{L6.2}
Let $I$, $I_0$, $I_1$, and $J$ be ideals on $\omega$. Then 
\begin{enumerate}[label=\rm(\alph*)]
\item\label{L4.14a} $\Lc_J(I\oplus\pts(\omega))\eqT \Lc_J(I\oplus I) \eqT \Lc_J(I)$, likewise for\/ $\pLc$, $\Lbf$ and\/ $\pL$. In particular,
\begin{multicols}{2}
\begin{enumerate}[label=\rm(\alph{enumi}\arabic*)]
\item
$\slalomt(I\oplus\Pcal(\omega),J)=\slalomt(I\oplus I,J)=\slalomt(I,J)$;
\item
$\slalome(I\oplus\Pcal(\omega),J)=\slalome(I\oplus I,J)=\slalome(I,J)$;
\item $\lfrak_\Kat(I\oplus\Pcal(\omega),J) = \lfrak_\Kat(I\oplus I,J)=\lfrak_\Kat(I,J)$;
\item $\pfrak_\Kat(I\oplus\Pcal(\omega),J) = \pfrak_\Kat(I\oplus I,J)=\pfrak_\Kat(I,J)$.
\end{enumerate}
\end{multicols}

\item\label{L4.14b} For $e\in\{0,1\}$, $\Lc_J(I_e) \leqT \Lc_J(I_0\oplus I_1) \leqT \Lc_J(I_0\cap I_1) \leqT \Lc_J(I_0)\otimes \Lc_J(I_1)$ and\/ $\pLc_J(I_e) \leqT \pLc_J(I_0\oplus I_1) \leqT \pLc_J(I_0\cap I_1) \leqT \Lc_J(I_e)\otimes \pLc_J(I_{1-e})$, likewise for\/ $\Lbf$ and\/~$\pL$ replacing\/ $\Lc$ and\/~$\pLc$, respectively. In particular,
\begin{enumerate}[label=\rm(\alph{enumi}\arabic*)]
\item\label{L4.14b1}
$\max\{\slalomt(I_0,J),\slalomt(I_1,J)\}=\slalomt(I_0\oplus I_1,J)=
\slalomt(I_0\cap I_1,J)$;
\item\label{L4.14b2}
$\max\{\slalome(I_0,J),\slalome(I_1,J)\}\le\slalome(I_0\oplus I_1,J)\le
\slalome(I_0\cap I_1,J)\leq
\min\{\max\{\slalomt(I_0,J),\slalome(I_1,J)\},\allowbreak\max\{\slalome(I_0,J),\slalomt(I_1,J)\}\}$;
\item $\max\{\lfrak_\Kat(I_0,J),\lfrak_\Kat(I_1,J)\}=\lfrak_\Kat(I_0\oplus I_1,J)=
\lfrak_\Kat(I_0\cap I_1,J)$;
\item\label{L4.14b4}
$\max\{\pfrak_\Kat(I_0,J),\pfrak_\Kat(I_1,J)\}\le\pfrak_\Kat(I_0\oplus I_1,J)\le
\pfrak_\Kat(I_0\cap I_1,J)\leq
\min\{\max\{\lfrak_\Kat(I_0,J),\pfrak_\Kat(I_1,J)\},\allowbreak\max\{\pfrak_\Kat(I_0,J),\lfrak_\Kat(I_1,J)\}\}$.
\end{enumerate}

\item\label{L4.14c}
$\slalome(I_0\oplus I_1,\Fin)=\slalome(I_0\cap I_1,\Fin) = \max\{\slalome(I_0,\Fin),\slalome(I_1,\Fin)\} = \min\{\bfrak,\max\{\covst(I_0),\covst(I_1)\}\}$.

\item \label{L4.14d}
$\covst(I_0\oplus I_1)=\covst(I_0\cap I_1)=
\max\{\covst(I_0),\covst(I_1)\}$.
\end{enumerate}
\end{theorem}

\begin{proof}
 \ref{L4.14a} and~\ref{L4.14b} follow by \Cref{L2.4} and \Cref{slint}~\ref{slint1}--\ref{slint2} since, by \Cref{L_op_ord} we have $I\oplus\Pcal(\omega)\approx_\Kat I\oplus I\approx_\Kat I$ and $I_0\cap I_1\leq_\Kat I_0\oplus I_1\leq_\Kat I_0,I_1$.

 \ref{L4.14c}: Immediate by~\ref{L4.14b2}, \Cref{slint}~\ref{slint3}, and \Cref{S4_eq}~\ref{S4_eq:01}.

 \ref{L4.14d}: Immediate by~\ref{L4.14b4} and \Cref{slint}~\ref{slint3}.
\end{proof}

\begin{theorem}\label{L6.3}
Let $I, I_0, J_0$, and $J_1$ be ideals on $\omega$. Then 
\begin{enumerate}[label=\rm(\alph*)]
\item\label{L4.15a}
$\Lc_{J\oplus \pts(\omega)}(I) \eqT \Lc_{J\oplus J}(I) \eqT \Lc_J(I)$, similarly for\/ $\pLc$, $\Lbf$ add $\pL$. In particular,
\begin{align*}
    \slalomt(I,J\oplus\Pcal(\omega))= \slalomt(I,J\oplus J)= \slalomt(I,J), &&
    \slalome(I,J\oplus\Pcal(\omega))= \slalome(I,J\oplus J)= \slalome(I,J),\\
    \lfrak_\Kat(I,J\oplus\Pcal(\omega))=\lfrak_\Kat(I,J\oplus J)=\lfrak_\Kat(I,J), &&
    \pfrak_\Kat(I,J\oplus\Pcal(\omega))=\pfrak_\Kat(I,J\oplus J)=\pfrak_\Kat(I,J).
\end{align*}

\item\label{L4.15b}
$ \Lc_{J_e}(I)\leqT \Lc_{J_0\cap J_1}(I) \leqT \Lc_{J_0\oplus J_1}(I) \eqT \Lc_{J_0}(I)\otimes \Lc_{J_1}(I)$ for $e\in\{0,1\}$, likewise for\/~$\Lbf$. In particular,
\begin{align*}
\slalomt(I,J_0\cap J_1) & =
\slalomt(I,J_0\oplus J_1)=
\max\{\slalomt(I,J_0),\slalomt(I,J_1)\},\\   
\lfrak_\Kat(I,J_0\cap J_1) & =
\lfrak_\Kat(I,J_0\oplus J_1)=
\max\{\lfrak_\Kat(I,J_0),\lfrak_\Kat(I,J_1)\}.
\end{align*}

\item\label{L4.15c} $\aLc_{J_0\cap J_1}(I)\leqT \aLc_{J_0\oplus J_1}(I) \eqT \aLc_{J_0}(I)\boxtimes \aLc_{J_1}(I)$, likewise for\/~$\pL$. In particular,
\begin{align*}
\slalome(I,J_0\cap J_1) & \le
\slalome(I,J_0\oplus J_1)=
\min\{\slalome(I,J_0),\slalome(I,J_1)\},\\
\pfrak_\Kat(I,J_0\cap J_1) & \le
\pfrak_\Kat(I,J_0\oplus J_1)=
\min\{\pfrak_\Kat(I,J_0),\pfrak_\Kat(I,J_1)\}.
\end{align*}

\item\label{L4.15d}
$\dsla{\star,J_0\oplus J_1}=\max\{\dsla{\star,J_0},\dsla{\star,J_1}\}$ and 
$\sla{\star,J_0\oplus J_1}=
\min\{\sla{\star,J_0},\sla{\star,J_1}\}$.

\item\label{L4.15e} $\pfrak_\Kat(\star,J_0\oplus J_1) = \min\{\pfrak_\Kat(\star,J_0),\pfrak_\Kat(\star,J_1)\}$ and $\lfrak_\Kat(\star,J_0\oplus J_1) = \max\{\lfrak_\Kat(\star,J_0),\lfrak_\Kat(\star,J_1)\}$.

\item\label{L4.15f} $\dsla{\star,J_0\cap J_1}=\max\{\dsla{\star,J_0},\dsla{\star,J_1}\}$ and $\lfrak_\Kat(\star,J_0\cap J_1) = \max\{\lfrak_\Kat(\star,J_0),\lfrak_\Kat(\star,J_1)\}$.
\end{enumerate}
\end{theorem}

\begin{proof}
\ref{L4.15a}:
It follows by \Cref{C2.7} because, by \Cref{L_op_ord}, we have $J\oplus\Pcal(\omega)\approx_\RB J\oplus J\approx_\RB J$.

\ref{L4.15b}: The Tukey equivalence follows by \Cref{lem:prod} because $\baire{I}\oplus \baire{I} = {}^{\omega\oplus\omega}I$ and $\cnt(I)\oplus\cnt(I)$ is the set of constant functions from $\omega\oplus\omega$ into $I$; and $\leqT$ follows by \Cref{L2.6} because $J_0\oplus J_1 \leq_{\mKB} J_0\cap J_1 \subseteq J_e$ by \Cref{L_op_ord}.

\ref{L4.15c}: The Tukey equivalence follows by \Cref{lem:prod} and $\leqT$ follows by \Cref{L2.6} because $J_0\cap J_1 \leq_{\RB} J_0\oplus J_1$.

\ref{L4.15d} By \Cref{G2}, $\sla{\star,J_0\oplus J_1}=
\min\{\sla{\star,J_0},\sla{\star,J_1}\}$ is immediate from~\ref{L4.15c}, and
\begin{align*}
    \dsla{\star,J_0\oplus J_1} & = \min_I \dsla{I,J_0\oplus J_1} = \min_I \max\{\dsla{I,J_0},\dsla{I,J_1}\} \geq \max\{\dsla{\star,J_0},\dsla{\star,J_1}\}
\end{align*}
follows by~\ref{L4.15b}. On the other hand, by also using \Cref{lem:prod},
\begin{align*}
    \max\{\dsla{\star,J_0},\dsla{\star,J_1}\} & = \max\left\{\min_{I_0}\dsla{I_0,J_0},\min_{I_1}\dsla{I_1,J_1}\right\} = \min_{I_0,I_1}\max\{\dsla{I_0,J_0},\dsla{I_1,J_1}\}\\
    & = \min_{I_0,I_1} \dsla{{}^\omega I_0\oplus {}^\omega I_1,J_0\oplus J_1} \geq \min_{\bar I} \dsla{\bar I,J_0\oplus J_1} = \dsla{\star,J_0\oplus J_1},
\end{align*}
which concludes the proof.

\ref{L4.15e}: The equation for $\pfrak_\Kat(\star,J_0\oplus J_1)$ is immediate from~\ref{L4.15c} and \Cref{minpklk}. For $\lfrak_\Kat(\star,J_0\oplus J_1)$, and in the proof above for $\slalomt(\star,J_0\oplus J_1)$ we can show that $\lfrak_\Kat(\star,J_0\oplus J_1)\geq \max\{\lfrak_\Kat(\star,J_0),\lfrak_\Kat(\star,J_1)\}$. For the converse, by \Cref{minpklk}, there are two ideals $I_0$ and $I_1$ on $\omega$ such that $\lfrak_\Kat(\star,J_e) = \lfrak_\Kat(I_e,J_e)$ for $e\in\{0,1\}$. As in the proof of \Cref{cor:min*}, find an ideal $I$ on $\omega$ such that $I_e\leqK I$ for $e\in\{0,1\}$. Hence, by \Cref{L2.4}~\ref{L2.4b},
\[\max_{e\in\{0,1\}} \lfrak_\Kat(\star,J_e) = \max_{e\in\{0,1\}} \lfrak_\Kat(I_e,J_e) \geq \max_{e\in\{0,1\}} \lfrak_\Kat(I,J_e) = \lfrak_\Kat(I,J_0\oplus J_1) \geq \lfrak_\Kat(\star,J_0\oplus J_1),\]
where the last equality holds by~\ref{L4.15b}.

\ref{L4.15f}: The proof is similar to~\ref{L4.15d} and~\ref{L4.15e} using~\ref{L4.15b} (for $J_0\cap J_1$).
\end{proof}

A direct application of~\ref{L4.15b} to $I=\Fin$ (while using $\bfrak_J=\slalomt^\perp(\Fin,J)$) yields:

\begin{corollary}\label{L6.4}
Let $J, J_0$, and $J_1$ be ideals on $\omega$. Then
\begin{enumerate}[label=\rm(\alph*)]
\item\label{L6.4a} $\bfrak\le \bfrak_{J\oplus\Pcal(\omega)}=\bfrak_{J\oplus J}=\bfrak_{J} \leq \dfrak_{J\oplus\Pcal(\omega)}=\dfrak_{J\oplus J}=\dfrak_{J} \leq \dfrak$.

\item
$\bfrak_{J_0\cap J_1}=\bfrak_{J_0\oplus J_1}=
\min\{\bfrak_{J_0},\bfrak_{J_1}\}$ and $\dfrak_{J_0\cap J_1}=\dfrak_{J_0\oplus J_1}=
\max\{\dfrak_{J_0},\dfrak_{J_1}\}$.\qed
\end{enumerate}
\end{corollary}


Thanks to \Cref{L6.3}, we can characterize the ideals $J$ on $\omega$ for which $\lfrak_\Kat(\star,J)$ is well defined.

\begin{theorem}\label{lK*}
    Let $J$ be an ideal on $\omega$. Then, the following statements are equivalent.
    \begin{enumerate}[label = \normalfont (\roman*)]
        \item\label{lK*1} $\lfrak_\Kat(\star,J)$ is well-defined.
        \item\label{lK*2} $\omega$ cannot be partitioned into infinitely many $J$-positive sets.
        \item\label{lK*3} $J$ is the disjoint sum of finitely many maximal ideals on $\omega$.
    \end{enumerate}
\end{theorem}
\begin{proof}
    \ref{lK*1}${}\Rightarrow{}$\ref{lK*2} follows by \Cref{L2.1}. Now assume~\ref{lK*2}. This means that the poset $\Por:=\pts(\omega)\menos J$, ordered by $\subseteq$, does not contain maximal antichains, i.e., it is $\omega$-cc. This implies that the set of atoms of this poset is dense. Hence, it contains a maximal antichain formed by atoms, which should be finite by the $\omega$-cc property. This implies that we can partition $\omega$ into finitely many atoms $\set{a_k}{k<n}$ in $\Por$, i.e., $a_k\in J^+$ and $J_k:=\pts(a_k)\cap J$ is a maximal ideal on $a_k$ for all $k<n$. Hence, $J$ is the disjoint sum of $\set{J_k}{k<n}$. This shows \ref{lK*2}${}\Rightarrow{}$\ref{lK*3}.

    Now assume~\ref{lK*3}. By \Cref{L6.3}~\ref{L4.15b} and \Cref{undef:pl}, it is enough to show that, whenever $\set{J_k}{k<n}$ is a finite set of maximal ideals on $\omega$, there is some maximal ideal $I$ on $\omega$ such that $I\nleq_\Kat J_k$ for all $k<n$. This is clear because there are at most $\cfrak$ many maximal ideals Kat\v{e}tov-below one ideal, but there are a total of $2^\cfrak$ many maximal ideals on $\omega$. This shows \ref{lK*3}${}\Rightarrow{}$\ref{lK*1}.
\end{proof}

As a consequence of \Cref{L6.3}~\ref{L4.15e}, it is enough to study $\lfrak_\Kat(\star,J_*)$ for all maximal ideals $J_*$ on $\omega$ to understand $\lfrak_\Kat(\star,J)$ for any ideal $J$ on $\omega$.

By counting maximal ideals as in the proof of \ref{lK*3}${}\Rightarrow{}$\ref{lK*1}, $\pfrak_\Kat(\star,J_*)$ is well-defined for any maximal ideal $J_*$ and hence, by \Cref{lem:monJ}, $\pfrak_D(\star,J)$ is well-defined for any $D\subseteq\Baire$ and any ideal $J$ on $\omega$.

\section{Selection principles}\label{sec:selection}

The current section is devoted to the study of selection principles and their connection to slalom numbers. First, we develop a~framework where our generalized slalom number is the uniformity number of a~topological property. Afterward, we show that topological spaces possessing most of the studied selection principles are singular.\footnote{For the meaning of \emph{singular} in this context, see~\cite[\S40]{KuraTop},~\cite{MiSp}, and~\cite[Ch.~8]{BStr}.} Throughout the section, we assume that all topological spaces are Hausdorff. The letter $X$ is reserved to denote a~(Hausdorff) topological space.

\begin{definition}\label{def:covers}
Fix a~Hausdorff space $X$, a~non-empty set~$a$, and $H\subseteq \PP(a)$. A~sequence $\seqn{V_m}{m\in a}$ of subsets of $X$ is \emph{non-trivial (in $X$)} if $V_m\neq X$ for all $m\in a$. Otherwise, we say that it is trivial. We use this terminology in connection with (open) covers of $X$. Trivial sequences cover $X$, so we usually refer to them as \emph{trivial covers}.

We say that $\seqn{V_m}{m\in a}$ is an~\emph{$H$-$\gamma$-cover of~$X$} if it is a~sequence in $X$ such that $\set{m\in a}{x\not\in V_m} \in H$ for each $x\in X$. We sometimes distinguish between trivial and non-trivial $H$-$\gamma$-covers in the sense of the previous paragraph. 
The family of all open $H$-$\gamma$-covers of a~topological space~$X$ is denoted by $H$-$\Gamma(X)$, or shortly~\mbox{$H$-$\Gamma$}. 
\end{definition}

\begin{remark}\label{rem:ainH}
When $a\notin H$, it is clear that any $H$-$\gamma$-cover of $X$ actually covers $X$. However, when $a\in H$, there may be $H$-$\gamma$-covers that are not necessarily covers of~$X$, e.g., when $H=\pts(a)$, any sequence of subsets of $X$ indexed by $a$ is an $H$-$\gamma$-cover. We allow this pathology for practicality as in \Cref{conv-w}, but in practice $a\notin H$ (and also $\emptyset\in H$ and $H$ is $\subseteq$-downwards closed), excluding pathological ``covers".
\end{remark}

We shall concentrate on the following situation.
\begin{enumerate}[label=\rm(H\arabic*)]
    \item\label{So:cover} For a~constant $q\in\omega$, $H$ is the family of all finite subsets of~$\omega$ with cardinality at most~$q$, i.e.,\ $H=[\omega]^{\leq q}$. The family $[\omega]^{\leq q}$-$\Gamma$, denoted by $\Gammac{q}$ shortly, is then the family of all (countable) open $\gamma_q$-covers of~$X$, studied in~\cite{SoDiz}, i.e.,\  a~sequence $\seqn{U_m}{m\in\omega}$ is called a~\emph{$\gamma_q$-cover of~$X$} if it is a~cover of~$X$ satisfying 
    $|\set{m\in\omega}{x\notin U_m}|\leq q$  for each $x\in X$. 
    
    Furthermore, let us emphasize that no monotone increasing cover $\seqn{U_m}{m<\omega}$ can be a non-trivial~$\gamma_q$-cover, otherwise, it will be a trivial $\gamma_q$-cover where $U_m=X$ for all $m\geq q$.


    \item\label{So:cover2} We expand~\ref{So:cover} as follows. For a~constant $q\in\omega$ and a~set~$a$ (maybe finite), $H$ is the family of all finite subsets of~$a$ with cardinality at most~$q$, i.e., $H=[a]^{\leq q}$. The family $[a]^{\leq q}$-$\Gamma$, denoted shortly~$\Gammac{a,q}$, is then the family of all open $\gamma_{a,q}$-covers of~$X$, where a~sequence $\seqn{U_m}{n\in a}$ is called a~\emph{$\gamma_{a,q}$-cover of~$X$} if 
    $|\set{m\in a}{x\notin U_m}|\leq q$ for each $x\in X$.

    When $1\leq q<|a|$, a~non-trivial open $\gamma_{a,q}$-cover exists iff $q|X|\geq|a|$. Indeed, if $|a|\leq q|X|$, there is some injection $f\colon a\to q\times X$, so $\seqn{U_m}{m\in a}$ is an open $\gamma_{a,q}$-cover where
    $U_m:= X\menos\{f_1(m)\}$ and $f(m)=(f_0(m),f_1(m))$. For the converse, if $\seqn{V_m}{m\in a}$ is a~non-trivial open $\gamma_{a,q}$-cover then, for any map $g\colon a\to X$ such that $g(m)\notin V_m$, $|g^{-1}[\{x\}]|\leq q$ for all $x\in X$, so $|a|=\left|\bigcup_{x\in X}g^{-1}[\{x\}]\right|\leq |X|q$.

    When $|a|\leq q$, $a\in[a]^{\leq q}$, so any sequence $\seqn{V_m}{m\in a}$ is a~$\gamma_{a,q}$-cover. On the other hand, 
    the only $\gamma_{a,0}$-cover is the trivial cover composed by $X$ alone.

    Clearly, $q\leq q'<\omega$ implies $\Gamma^{a,q}\subseteq\Gamma^{a,q'}$. 

    \item $H$ is the family of all finite subsets of~$\omega$, i.e., $H=\fin=[\omega]^{<\omega}$. The family $\fin$-$\Gamma$, denoted shortly~$\Gamma$, is then the family of all (countable) open \emph{$\gamma$-covers} of~$X$.

    \item $H$ is an~ideal~$\I$ on~$\omega$, i.e., $H=\I$. The family $\I$-$\Gamma$ is then the family of all open $\I$-$\gamma$-covers of~$X$. As in~\ref{So:cover2}, there exists a non-trivial open $I$-$\gamma$-cover of $X$ iff $X$ is infinite.

    \item $H$ is the family $\I^{\rm dc}\subseteq \pts(\omega)$ of sets not in the filter $\I^\dual$ when $I$ is an ideal on $\omega$. The family $H$-$\Gamma$, denoted by $\I$-$\Lambda$, is then the family of all open \emph{$\I$-large covers of~$X$}, see~\cite{DaKoCh}. When $|X|\geq 2$ and $I$~is not a~maximal ideal, there is a non-trivial open $I$-large cover of~$X$. 

    \item $H$ is the family $\PP(\omega)\smallsetminus\{\omega\}$. The family $H$-$\Gamma$, denoted by $\OO$, is the family of all countable open covers of~$X$, see~\cite{Comb1}.

    \item $H$ is the family $\pts(a)\menos\{a\}$. Here $H\rr\Gamma$ is the family of all open covers of~$X$ indexed by~$a$, which we denote by~$\OO^a$, or by just~$\OO$ when clear from the context.
\end{enumerate}
The notions above are usually defined in the literature for non-trivial covers, but we allow trivial covers in this paper for reasons we discuss later in \Cref{easyS1},~\ref{convS1} and~\ref{selct}. 

We also consider the notion of $\omega$-cover: recall that a~non-trivial sequence $\seqn{V_m}{m\in a}$ is an \emph{$\omega$-cover of~$X$} if, for any finite $F\subseteq X$, $F\subseteq V_m$ for some $m\in a$. Note that a~non-trivial sequence $\seqn{V_m}{m\in a}$ is an $\omega$-cover if it is an $I$-$\gamma$-cover for some ideal $I\subseteq\pts(a)$, which is equivalent to the fact that $\set{m\in a}{F\subseteq V_m}$ is infinite for all finite $F\subseteq X$. We use this to extend the notion of $\omega$-cover to trivial covers: regardless of whether $\seqn{V_m}{m\in a}$ is trivial, we say that it is an \emph{$\omega$-cover} if $\set{m\in a}{F\subseteq V_m}$ is infinite for all finite $F\subseteq X$, which is equivalent to being an $I$-$\gamma$-cover for some ideal $I\subseteq\pts(a)$.

Denote by $\Omega^a$ the collection of open $\omega$-covers indexed by $a$, and $\Omega:=\Omega^\omega$. Notice that there are no $\omega$-covers indexed by a~finite set and that finite spaces cannot have non-trivial $\omega$-covers. 

It is clear that $H\subseteq H'$ implies $H\text{-}\Gamma \subseteq H'\text{-}\Gamma$. In particular, when $I$ is an~ideal on~$\omega$,
\begin{align}
 \Gammac{q} \subseteq \Gamma \subseteq \GammaB{I}\subseteq \LambdaB{I}\subseteq \mathcal{O}  \text{ and }I\rr\Gamma\subseteq\Omega.
\end{align}
We shall also deal with sequences of covers.

\begin{definition}\label{def:seqcovers}
Let us consider a~set~$E\subseteq\prod_{i<\omega}\PP(b(i))$ for a~sequence $b=\seqn{b(n)}{n\in\omega}$ of non-empty sets. The family $E$-$\Gamma(X)$, or shortly $E$-$\Gamma$, is the family of sequences $\seqn{\seqn{V_{n,m}}{m\in b(n)}}{n\in\omega}$ such that each $V_{n,m}$ is open in $X$ and $\seqn{\set{m\in b(n)}{x\not\in V_{n,m}}}{n\in\omega}\in E$ for each $x\in X$. 
In particular, when $E$~is the product $\prod_{n\in \omega}H_n$, we denote
\[
E\text{-}\Gamma=\prod\limits_{n\in\omega}(H_n\text{-}\Gamma).
\]
We shall concentrate on the following situation.
\begin{enumerate}[label=\rm(P\arabic*), series=Plist]
    \item\label{HGamma} $E$ is the family $\baire H$ for some $H\subseteq\pts(a)$ as in \Cref{def:covers} (so $b(n)=a$ for all $n$). In such a~case, we keep the notation $H$-$\Gamma$ instead of using $(\baire H)$-$\Gamma$. 
    When $H=\pts(a)\menos\{a\}$, we just use $\OO^a$, or $\OO$; 
    when $a=\omega$ and $H=\Fin$, we just use $\Gamma$.

    \item $E$ is the family $\cnt(H)$ of constant sequences in ${}^\omega H$, where $H$ is as before. Notice that $\seqn{\seqn{V_{n,m}}{m\in a}}{n\in\omega}\in \cnt(H)\rr\Gamma$ iff for $m\in a$ and $n,n'<\omega$, $V_{n,m} = V_{n',m}$ (is open in $X$), i.e.,\ a~sequence is in $\cnt(H)\rr\Gamma$ iff it is the constant sequence of an~open $H$-$\gamma$-cover of $X$.

    \item $E$ is the family $\calS(b,h)$ for $h\in{}^\omega\omega$. The family $\calS(b,h)$-$\Gamma$ will be denoted by $\Gamma_{b,h}$, shortly.
    \item $E$ is the family $\calS(\omega,h)$. The family $\calS(\omega,h)$-$\Gamma$ is denoted by $\Gamma_h$ in~\cite{SoDiz}.  
\end{enumerate}
\end{definition}

In particular, we have $\Gammah{h}\subseteq {}^\omega{}\Fin\text{-}\Gamma\subseteq {}^\omega I\rr\Gamma \subseteq {}^\omega(I\rr\Lambda)$. 

In the more general framework of \Cref{sec:LpL}, the invariant~$\dfrak_A(D,\Pwf)$ is defined for $\Pwf$ being a~property of families of functions with domain~$\omega$, or equivalently, a~collection of families of functions with domain~$\omega$. As in the previous notion, we use $\Pcal$ to define a~collection of sequences as follows.

\begin{definition}\label{supersel}
In the following, we are given a~sequence $b=\seqn{b(n)}{n\in\omega}$ of non-empty sets and $\Pwf\subseteq\PP(\prod_{i<\omega}\PP(b(i)))$. We shall introduce the family $\Pwf$-$\Gamma$ for sequences of sequences of open sets associated with~$\Pwf$ as well. Indeed, $\Pwf$-$\Gamma(X)$, or shortly $\Pwf$-$\Gamma$, is the family of $\seqn{\seqn{V_{n,m}}{m\in b(n)}}{n\in\omega}$ such that each $V_{n,m}$ is open in $X$ and 
\[\bigset{\seqn{\set{m\in b(n)}{x\not\in V_{n,m}}}{n\in\omega}}{x\in X}\in \Pwf \text{ (has property~$\Pwf$).}\] 
We shall concentrate on the following situation.
\begin{enumerate}[resume*=Plist]
    \item\label{seqcov4} $\Pwf$ is the family $\PP(E)$ with $E$ being a~set containing functions with domain~$\omega$. Here, $\Pcal$-$\Gamma$ is $E$-$\Gamma$, in accordance with the previously adopted notation.

    \item\label{seqcov5} $\Pwf$ is the family of all collections of functions in $\prod b$ with the coordinate-wise finite union property (i.e., property $\star$ from \Cref{def:Lc}), which we also abbreviate by \emph{FUPC}.  The family $\Pwf$-$\Gamma$ corresponds to the family of all sequences of $\omega$-covers. Hence $\Pwf$-$\Gamma = \prod_{n<\omega} \Omega^{b(n)}$, so we shall keep the notation $\Omega$ for sequences of $\omega$-covers as well.

    \item\label{seqcov6} $\Pwf$ is property $\Pwf_{\cnt}$ (see \Cref{def:Lc}). Then, $\Pwf_{\cnt}$-$\Gamma$ is the collection of constant sequences of $\omega$-covers.
\end{enumerate}

For $D\subseteq\prod b$ and $R\subseteq\Pwf(\omega)$, we say that a~Hausdorff topological space~$X$ is an~\emph{$\isonefgg{D}{\Pwf}{R}$-space} if, for each $\seqn{\seqn{V_{n,m}}{m\in b(n)}}{n\in\omega}\in \Pwf\text{-}\Gamma$, there is a~$d\in D$ such that $\seqn{V_{n,d(n)}}{n\in\omega}$ is an~$R$-$\gamma$-cover of~$X$ (which could be trivial). We focus on the following situation:
\begin{enumerate}[label=\rm(S\arabic*)]
    \item For the family $D$:
    \begin{enumerate}[label=\rm(\alph*)]
        \item\label{def_sone_Dproduct} If $D=\prod b$ then we just write $b$ in the notation, i.e., $\isonefgg{b}{\Pwf}{R}$.
        \item If $b$ is the constant sequence of a~set~$a$, then we just keep $a$, i.e., $\isonefgg{a}{\Pwf}{R}$.
        \item If $D=\baire\omega$ ($a$ is $\omega$ in previous item), then we write $\isonegg{\Pwf}{R}$.
    \end{enumerate}
    \item For the family $R$:
    \begin{enumerate}[label=\rm(\alph*)]
        \item If $R=\J$ then we write $\isonefgg{D}{\Pwf}{\J}$.
        \item If $R=\J^{\rm dc}$ then we write $\isonefgl{D}{\Pwf}{\J}$.
    \end{enumerate} 
    \item For the family $\Pwf$:
    \begin{enumerate}[label=\rm(\alph*)]
        \item If $\Pwf= \pts(\prod\bar{H}) = \pts\left(\prod_{n<\omega}H_n\right)$ when $\bar{H}=\seqn{H_n}{n<\omega}$ is a~sequence of families $H_n\subseteq\pts(b(n))$, then we use $\isonefgg{D}{\bar{H}}{R}$.
        \item If $H\subseteq\pts(a)$ (here $b(n)=a$ for all $n$) and $\Pwf=\pts(\baire H)$ then we use $\isonefgg{D}{H}{R}$.

        \item If $\Pwf=\Pwf(\cnt(H))$ for some $H\subseteq\pts(a)$ then we use $\Srm_1^D(\cnt(H)\rr\Gamma,R\rr\Gamma)$ or $[H\rr\Gamma, R\rr\Gamma]_D$, the latter as in \cite{SoSu,Su22,BaSuZd}. Here, $D\subseteq{}^\omega a$. 
        Notice that $X$ is an $[H\rr\Gamma, R\rr\Gamma]_D$-space iff, for any open $H$-$\gamma$-cover $\seqn{V_m}{m\in a}$, there is some $d\in D$ such that $\seqn{V_{d(n)}}{n<\omega}$ is an $R$-$\gamma$-cover. In the cited literature all covers are non-trivial, but equivalence with our proposed version holds in the interesting cases, see \Cref{selct}.

        \item If $\Pwf$ is property $\star$, then we use $\isonefmg{D}{R}$.

        \item If $\Pwf$ is property $\Pwf_{\cnt}$, we also use $[\Omega,R\rr\Gamma]_D$. Notice that $X$ is an $[\Omega,R\rr\Gamma]_D$-space if, for any open $\omega$-cover $\seqn{V_m}{m<\omega}$ (possibly trivial), there is some $d\in D$ such that  $\seqn{V_{d(n)}}{n<\omega}$ is an $R$-$\gamma$-cover (possibly trivial), see \cite{Su22,BaSuZd}. Traditionally, trivial covers are excluded, but we have equivalence with our proposed version in the interesting cases, see \Cref{selct}.

        \item  If $\Pwf= \pts(\calS(\omega,h))$ then we use $\isonefgsg{D}{h}{R}$. Accordingly, we use $\sonef{D}{\Gamma_q}{R\text{-}\Gamma}$ for the family $\calS(\omega,q)$, where~$q$ is the function with constant value equal to~$q\in\omega$.

        \item  If $\Pwf= \pts(\calS(b,h))$ then we use $\isonefgsg{D}{b,h}{R}$.
    \end{enumerate}
    We can define $\Srm_1^D(\Pcal\rr\Gamma,\Omega)$ similarly.

    Define the cardinal number $\non{\isonefgg{D}{\Pwf}{R}}$ as the smallest size of a~Hausdorff space that is not an $\isonefgg{D}{\Pwf}{R}$-space, which is known as the \emph{uniformity} or \emph{critical cardinality of $\isonefgg{D}{\Pwf}{R}$-spaces}. We define $\non{\Srm_1^D(\Pcal\rr\Gamma,\Omega)}$ similarly.

    We say that $X$ is a~\emph{traditional $\isonefgg{D}{\Pwf}{R}$-space} if it satisfies the principle $\isonefgg{D}{\Pwf}{R}$ but forbidding trivial covers.
\end{enumerate}
\end{definition}

Unless otherwise indicated, for the rest of this section, $a$ is a non-empty set and $b$, $\Pwf$, $D$, and $R$ are as in \Cref{supersel}.

The reason we allow trivial coverings in \Cref{def:seqcovers} is that this does not affect the well-known instances of the traditional principle $\isonefgg{D}{\Pwf}{R}$, i.e., forbidding trivial covers (as it has been traditionally studied). We justify this with a~series of results.

\begin{lemma}\label{easyS1}
    If a~Hausdorff space $X$ is an $\isonefgg{D}{\Pwf}{R}$-space then it is a~traditional $\isonefgg{D}{\Pwf}{R}$-space.\qed
\end{lemma}

The converse holds in certain situations.

\begin{lemma}\label{convS1}
    The converse of \Cref{easyS1} holds when $D=\prod b$, $R$~is downwards $\subseteq$-closed, $\Pcal$~is as in~\ref{seqcov4} with $E=\prod_{n<\omega} H_n$ for some $H_n\subseteq \pts(b(n))$ and, for any $n<\omega$, $X$ has a~non-trivial $H_n$-$\gamma$-cover. The latter requirement holds when $X$ is infinite and $[b(n)]^{\leq 1}\subseteq H_n$ for all $n<\omega$.

    The above is also valid when considering $\Omega^{b(n)}$ in the place of $H_n$-$\gamma$-covers for some desired $n<\omega$, and also~$\Omega$ in the place of $R$-$\Gamma$. In particular, $\Pcal$ can be as in~\ref{seqcov5}.
\end{lemma}
\begin{proof}
    Assume that, for $n<\omega$, $\seqn{V_{n,m}}{m\in b(n)}$ is an $H_n$-$\gamma$ cover. Let $a:=\set{n<\omega}{\forall m\in b(n) \allowbreak\ (V_{n,m}\neq X)}$. 
    For $n\in\omega\menos a$, choose some non-trivial $H_n$-$\gamma$-cover $\seqn{V'_{n,m}}{m\in b(n)}$ of $X$; 
    for $n\in a$ and $m\in b(n)$ let $V'_{n,n}:= V_{n,m}$, so we can apply the traditional $\isonefgg{b}{E}{R}$ (for non-trivial covers) to get some $d'\in\prod b$ such that $\seqn{V'_{n,d'(n)}}{n<\omega}$ is an $R$-$\gamma$-cover. Define $d\in\prod b$ by $d(n):=d'(n)$ when $n\in a$, otherwise choose $d(n)\in b(n)$ such that $V_{n,d(n)}=X$. Then, for $x\in X$, $\set{n<\omega}{x\notin V_{n,d(n)}}= \set{n\in a}{x\notin V_{n,d(n)}}\subseteq \set{n<\omega}{x\notin V'_{n,d(n)}}\in R$, so $\set{n<\omega}{x\notin V_{n,d(n)}}\in R$.
\end{proof}

We analyze the situation when $X$ may not have a~non-trivial $H_n$-$\gamma$-cover for some $n<\omega$ as follows.

\begin{lemma}\label{conv-w}
    Let $D$, $R$, $\Pcal$, $E$ and $\seqn{H_n}{n<\omega}$ be as in \Cref{convS1}. Let $X$ be a~Hausdorff space and let $w$ be the set of all $n<\omega$ such that $X$ has a~non-trivial $H_n$-$\gamma$-cover, and assume that\/ $\emptyset\in R$. Then:
    \begin{enumerate}[label =\rm (\alph*)]
        \item\label{conv-wa} If $w\neq\omega$ then $X$ is (vacuously) a~traditional\/ $\isonefgg{b}{E}{R}$-space.

        \item\label{conv-wb} $X$ is an\/ $\isonefgg{b}{E}{R}$-space iff it is an\/ $\isonefgg{b\frestr w}{E\frestr w}{\pts(w)\cap R}$-space\footnote{The version of $\isonefgg{b}{E}{R}$ by indexing the sequences of covers with~$w$ instead of~$\omega$.}, which is in turn equivalent to being a~traditional\/ $\isonefgg{b\frestr w}{E\frestr w}{\pts(w)\cap R}$-space. However, $\Rightarrow$ needs\/ $H_n\neq\emptyset$ for $n\in\omega\menos w$.

        \item\label{conv-wc} If $w\in R$ then $X$ is an\/ $\isonefgg{b\frestr w}{E\frestr w}{\pts(w)\cap R}$-space.
    \end{enumerate}
    The above is also valid when considering $\Omega^{b(n)}$ in the place of $H_n$-$\gamma$-covers for some desired $n<\omega$, as long as $b(n)$ is infinite.
\end{lemma}

\begin{proof}
    \ref{conv-wa}: Since there is no sequence in $E$-$\Gamma$ of non-trivial sequences (i.e., any member of $E$-$\Gamma$ contains a~trivial cover at any $n\in\omega\menos w$), the traditional principle holds vacuously.

    \ref{conv-wb}: By \Cref{convS1}, $X$ is an $\isonefgg{b\frestr w}{E\frestr w}{\pts(w)\cap R}$-space iff it is a~traditional $\isonefgg{b\frestr w}{E\frestr w}{\pts(w)\cap R}$-space. On the other hand, if $X$ is an $\isonefgg{b}{E}{R}$-space, then we can use a suitable trivial cover at any $n\in\omega\menos w$ to check that $X$ is an $\isonefgg{b\frestr w}{E\frestr w}{\pts(w)\cap R}$-space. The converse is easy to check.

    \ref{conv-wc}: Trivial because $w\in\pts(w)\cap R=\pts(w)$, so any sequence $\seqn{U_n}{n\in w}$ is a~$\pts(w)$-$\gamma$-cover.    
\end{proof}

Similar proofs yield 
the case when $R\rr\Gamma$ is replaced by~$\Omega$. 

\begin{lemma}\label{conv-wo}
    Under the same assumptions of \Cref{conv-w} (excluding $R$):
    \begin{enumerate}[label =\rm (\alph*)]
        \item\label{conv-wo0}  If $w\neq\omega$ then $X$ is (vacuously) a~traditional\/ $\Srm_1^b(E\rr\gamma,\Omega)$-space.     
        \item\label{conv-woa} If\/ $|\omega\menos w|=\aleph_0$ then $X$ is an\/ $\Srm_1^b(E\rr\gamma,\Omega)$-space.
        \item\label{conv-woc} If $X$ is an\/ $\Srm_1^{b\frestr w}(E\frestr w\rr\Gamma, \Omega^w)$-space then it is an\/  $\Srm_1^b(E\rr\gamma,\Omega)$-space.
        \item\label{conv-wocc} The converse of\/~\ref{conv-wocc} holds when $\omega\menos w$ is finite and\/ $H_n\neq \emptyset$ for $n\in\omega\menos w$.
        \item\label{conv-wod} If $w$ is finite then $X$ is not an\/ $\Srm_1^{b\frestr w}(E\frestr w\rr\Gamma, \Omega^w)$-space.
    \end{enumerate}
    The above is also valid when considering $\Omega^{b(n)}$ in the place of $H_n$-$\gamma$-covers for some desired $n<\omega$, as long as $b(n)$ is infinite.
\end{lemma}
\begin{proof}
    \ref{conv-woa}: Any $d\in\prod b$ picking $X$ from a~trivial cover at each $n\in\omega\menos w$ produces a~trivial $\omega$-cover.

    \ref{conv-wod}: Clear because there are no $\omega$-covers indexed by finite sets.
\end{proof}

The non-traditional selection principle has the following interesting effect on the slalom numbers.

\begin{lemma}\label{S1d1}
If there is some\/ $\isonefgg{D}{\Pwf}{R}$-space $X$ then\/ $\dfrak_{R^{\rm c}}(D,\Pwf)>1$.
In addition, if either this $X$ is finite, or $S\cup\{\overline{\emptyset}\}\in \Pcal$ for any $S\in \Pcal$ (where\/ $\overline{\emptyset}$ is the infinite constant sequence formed by the empty set), then\/ $\dfrak_{R^{\rm c}}(D,\Pwf)\geq \min\{\aleph_0,|X|+1\}$.
\end{lemma}
\begin{proof}
    First assume that $S\cup\{\overline{\emptyset}\}\in \Pcal$ for any $S\in \Pcal$. 
    To check $\dfrak_{R^{\rm c}(D,\Pwf)}\geq \min\{\aleph_0,|X|+1\}$, assume that $F\subseteq X$ is finite and $\set{\varphi_x}{x\in F}\in\Pcal$.     
    For $n<\omega$ and $m\in b(n)$, define $V_{n,m}:= X\menos \set{x\in F}{m\in \varphi_x(n)}$, which is clearly open in $X$ (but $V_{n,m}=X$ when $m\notin \varphi_x(n)$ for all $x\in F$). For $n<\omega$ and $x\in X$, $\set{m\in b(n)}{x\notin V_{n,m}}$ is $\varphi_x(n)$ when $x\in F$, and empty otherwise. Hence, $\seqn{\seqn{V_{n,m}}{m\in b(n)}}{n<\omega}$ satisfies $\Pcal$ (because $\set{\varphi_x}{x\in F}\cup\{\overline{\emptyset}\}\in \Pcal\}$). Since $X$ is an $\isonefgg{D}{\Pwf}{R}$-space, there is some $d\in D$ such that $\set{V_{n,d(n)}}{n<\omega}$ is an~$R$-$\gamma$-cover. Then, for $x\in F$, $\|d\in \varphi_x\| = \set{n<\omega}{x\notin V_{n,d(n)}}\in R$, i.e.,\ $d\notin^{R^{\rm c}} \varphi_x$.

    When $X$ is finite and $0<k\leq|X|$, we can partition $X$ into closed sets $\seqn{C_i}{i<k}$. For $n<\omega$ and $m\in b(n)$ define $V_{n,m}:= X\menos \bigcup_{i<k}\set{x\in C_i}{m\in \varphi_i(n)}$ and proceed like above. Since the previous partition is possible for $k=1$, we can conclude $\dfrak_{R^{\rm c}}(D,\Pwf)>1$ without the additional assumption.
\end{proof}

It is also worth to compare the traditional versions of $[H\rr\Gamma,R\rr\Gamma]_D$ with the current one.

\begin{lemma}\label{selct}
    Let $a$ be a~non-empty set, $D\subseteq {}^\omega a$, $H\subseteq \pts(a)$, $R\subseteq\pts(\omega)$ downwards\/ $\subseteq$-closed, and let $X$ be a~Hausdorff space. 
    In each of the following cases,  $X$ is a~traditional\/ $[H\rr\Gamma,R\rr\Gamma]_D$-space implies that its is an\/ $[H\rr\Gamma,R\rr\Gamma]_D$-space.
    \begin{enumerate}[label = \normalfont (\arabic*)]
        \item\label{selct1} $|X|\geq |a|$ and $z\cup c\in H$ for all $z\in H$ and $c\in[a]^{<\aleph_0}$. When $|X|<|a|$ and $X$ is finite, the implication holds when $H=I$ or $H=I^\dc$ for some ideal $I$ on~$a$, and\/ $\pfrak_D(H,R)>1$.\footnote{Note that $\dfrak_{R^{\rm c}}(D,\cnt(H)) = \pfrak_D(H,R)$, always.} 

        \item\label{selct2} $H=\pts(a)\menos\{a\}$ and\/ $\pfrak_D(H,R)>1$. The latter is equivalent to\/ $(\forall m\in a)(\exists d\in D)\ \|d\neq m\|\in R$.

        \item\label{selct3} In the case of\/ $[\Omega^a,R\rr\Gamma]_D$, it requires that either $a$~is finite or\/ $\pfrak_D(([a]^{<\aleph_0})^\dc,R)>1$.
    \end{enumerate}
    The same holds for when replacing $R\rr\Gamma$ by~$\Omega$, and $R$ by\/ $\Fin$ in\/ $\pfrak_D(\,\cdot\,,R)$.
\end{lemma}

\begin{proof}
To proceed with the proof of all the items, we fix an~open $H$-$\gamma$ cover (or $\omega$-cover) $\seqn{V_m}{m\in a}$, and let $w:=\set{m\in a}{V_m\neq X}$. When $w=a$ we can apply the traditional principle, so the problem is when $w\neq a$. In the case $w\in H$, it is enough to deal with the case when $V_m=\emptyset$ for all $m\in w$, in which we can appeal to $\pfrak_D(H,R)>1$ for finding some $d\in D$ such that, for $x\in X$, $\set{n<\omega}{x\notin V_{n,d(n)}} = \set{n<\omega}{d(n)\in w}\in R$. We analyze each case below.

    \ref{selct1}: When $|X|\geq|a|$, we can define an open $H$-$\gamma$-cover $\set{V'_m}{m\in a}$ such that $V'_m:=V_m$ for $m\in w$, otherwise $V'_m:= X\menos\{x_m\}$, where $\set{x_m}{m\in a\menos w}$ is a~chosen one-to-one sequence in $X$. Therefore, the traditional principle can be applied.

    When $X$ is finite and $H=I$ or $H=I^\dc$, in the case $w\in H$ we appeal to $\pfrak_D(H,R)>1$. So consider the case when $w\notin H$. If $H=I$ then $X$ must be infinite, indicating that $w\in I$ when $X$ is finite; if $H=I^\dc$ then $w\in I^\dual$, so we can define an $I^\dc$-$\gamma$-cover $\seqn{V''_m}{m\in a}$ where  $V''_m:=V_m$ when $m\in w$, and $V''_m:=\emptyset$ otherwise. Then, the traditional principle can be applied.

    \ref{selct2}: If $w\neq a$ then already $w\in H=\pts(a)\menos a$, so we can appeal to $\pfrak_D(H,R)>1$.

    \ref{selct3}: If $a$ is finite then there are no $\omega$-covers, so both $[\Omega^a,R\rr\Gamma]_D$ and its traditional version hold vaccuously. So assume that $a$ is infinite. If $a\menos w$ is finite then the sequence $\seqn{V''_m}{m\in a}$ defined in the proof of~\ref{selct1} is an $\omega$-cover and the traditional principle can be applied. Otherwise $w\in ([a]^{<\aleph_0})^\dc$, so we appeal to $\pfrak_D(([a]^{<\aleph_0})^\dc,R)>1$.

    The same arguments can be used when replacing $R\rr\Gamma$ by $\Omega$.
\end{proof}

The flexibility to allow trivial covers allows us to look at finite spaces $X$, which is reasonable in the context of this section because there will be cases when $\non{\isonefgg{D}{\Pwf}{R}}$ is finite.

The slalom numbers of the form $\dfrak_A(D,\Pwf)$ characterize $\non{\isonefgg{D}{\Pwf}{R}}$ as follows. This characterization is one of the main results in this section.

\begin{theorem}\label{G}
If $\Pwf$ is a~family of sets of functions with domain~$\omega$ and $A\subseteq\pts(\omega)$, then 
\[
\dfrak_A(D,\Pwf)=\non{\isonefgg{D}{\Pwf}{A^c}}.
\]
In particular, if $E$ is a~family of functions with domain~$\omega$ then\/ $\dfrak_A(D,E)=\non{\isonefgg{D}{E}{A^c}}$.
\end{theorem}

This result is immediate from the following two lemmata.

\begin{lemma}
Let $\Pwf$ be a~family of families of functions with domain~$\omega$, and let $X$ be a~topological space. If\/ $|X|<\dfrak_A(D,\Pwf)$ then $X$ is an\/ $\isonefgg{D}{\Pwf}{A^{\rm c}}$-space.
\end{lemma}

\begin{proof}
Let $|X|<\dfrak_A(D,\Pwf)$ and $\V=\seqn{\seqn{V_{n,m}}{m\in b(n)}}{n\in\omega}\in\Pwf\text{-}\Gamma$. For $x\in X$, we define the sequence~$x_\V$ by $x_\V(n)=\set{m\in b(n)}{x\not\in V_{n,m}}$ (the set of exceptions of the $n$-th cover), and set $E_0:=\set{x_\V}{x\in X}$. Then $E_0\in\Pwf$ since $\V\in\Pwf\text{-}\Gamma$. By the assumption, there is a~$d\in D$ such that, for all $x\in X$, $\| d \in x_\V\|\in A^{\rm c}$. However, $\set{n\in\omega}{x\not\in V_{n,d(n)}}=\| d \in x_\V\|$, so $\seqn{V_{n,d(n)}}{n\in\omega}$ is an $A^{\rm c}$-$\gamma$-cover.
\end{proof}

In the next lemma, we shall assume that $E_0\in \Pcal$ is equipped with a~topology such that $\set{p\in E_0}{m\not\in p(n)}$ is open for each $m$,~$n$. For instance, this is the case of the discrete topology on~$E_0$, or the topology on~$E_0$ inherited from the Tychonoff product topology of~$\prod_{n<\omega}\pts(b(n))$ when $E_0\subseteq\prod_{n<\omega}\pts(b(n))$ and $\pts(b(n))$ is considered with the product topology of ${}^{b(n)}2$ (with $2=\{0,1\}$ discrete). In the latter case, when each $b(n)$ is countable, $E_0$ is homeomorphic to a~set of reals.

\begin{lemma}
Let $\Pwf$ be a~family of sets of functions with domain~$\omega$, $A\subseteq\pts(\omega)$, and let $E_0\in\Pwf$. If $E_0$ is an\/ $\isonefgg{D}{\Pwf}{A^c}$-space then there is a~$d\in D$ such that $d\not\in^Ap$ for each $p\in E_0$.
\end{lemma}

\begin{proof}
We consider the sequence $\V_{E_0}=\seqn{\seqn{V_{n,m}}{m\in b(n)}}{n\in\omega}$ defined by $V_{n,m}:=\set{p\in E_0}{m\not\in p(n)}$. 
Note that $\seqn{\set{m\in b(n)}{p\not\in V_{n,m}}}{n\in\omega}=p$, hence, we have $\seqn{\seqn{V_{n,m}}{m\in b(n)}}{n\in\omega}\in \Pwf\text{-}\Gamma$. Observe that, for any $d\in D$ and any $p\in E_0$, we have $\set{n\in \omega}{d(n)\in p(n)}=\set{n\in \omega}{p\not\in V_{n,d(n)}}.$ 
Hence, $d\not\in^A p$ for all $p\in E_0$ if and only if $\seqn{V_{n,d(n)}}{n\in\omega}$ is an~$A^c$-$\gamma$-cover of~$E_0$. 
\end{proof}

The case $A=\Fin^{\dual}$ is quite special.

\begin{lemma}\label{LOO}
    \ 
    \begin{enumerate}[label = \normalfont (\alph*)]
        \item\label{LOOa} If $H\subseteq \pts(a)$ then the principles\/ $\Srm_1(H\rr\Gamma,\Fin\rr\Lambda)$ and\/ $\Srm_1(H\rr\Gamma,\OO)$ are equivalent.

        \item\label{LOOb} Assume $H_n\subseteq \pts(b(n))$ and\/ $\bigcup H_n=b(n)$ for $n<\omega$. 
        If $D\subseteq \prod b$ and\/ $\dfrak_{\Fin^\dual}(D,\bar H)\geq\kappa:=\sum_{n<\omega}|b(n)|$, then\/ $\non{\Srm^D_1(\bar H\rr\Gamma,\Fin\rr\Lambda)} = \non{\Srm^D_1(\bar H\rr\Gamma,\OO)}$.
    \end{enumerate}
\end{lemma}
\begin{proof}
    Any $\Srm_1(\Pcal\rr\Gamma,\Fin\rr\Lambda)$-space is $\Srm_1(\Pcal\rr\Gamma,\OO)$, hence $\non{\Srm_1(\Pcal\rr\Gamma,\Fin\rr\Lambda)} \leq \non{\Srm_1(\Pcal\rr\Gamma,\OO)}$. We show the converse in the situation above.

    \ref{LOOa}: Assume that $X$ is an $\Srm_1(H\rr\Gamma,\OO)$-space. For $n<\omega$, let $\V_n = \seqn{V_{n,m}}{m\in a}\in H\rr\Gamma$. Partition $\omega$ into infinite sets $\seqn{W_k}{k<\omega}$. By applying the principle $\Srm_1(H\rr\Gamma,\OO)$ to $\seqn{\V_n}{n\in W_k}$, we can find some $d_k\in {}^{W_n}a$ such that $\seqn{V_{n,d_k(n)}}{n\in W_k}$ covers $X$. Set $d:=\bigcup_{k<\omega}d_k$. Hence, $\seqn{V_{n,d(n)}}{n<\omega}\in \Fin\rr\Lambda$.

    \ref{LOOb}: By \Cref{G} it is enough to show that $\dfrak_{\Fin^\dual}(D,\bar H) = \dfrak_{\{\omega\}}(D,\bar H)$. The inequality $\leq$ is clear; for the converse, notice that, for any $\Lc_\Fin(D,\bar H)$-dominating $Y\subseteq \prod \bar H$, the set of finite modifications of members of $Y$ within $\bar H$ is $\la D, \prod\bar H, \in^{\{\omega\}} \ra$-dominating (because $b(n) = \bigcup H_n$ for all $n<\omega$), thus $\dfrak_{\Fin^\dual}(D,\bar H) \leq \dfrak_{\{\omega\}}(D,\bar H)\leq \max\left\{\kappa,\dfrak_{\Fin^\dual}(D,\bar H)\right\}$. But $\dfrak_{\Fin^\dual}(D,\bar H)\geq \kappa$, so $\dfrak_{\{\omega\}}(D,\bar H) =\dfrak_{\Fin^\dual}(D,\bar H)$.
\end{proof}

As a~consequence of \Cref{G},~\ref{LOO}, and the results of \Cref{sec:partcases}, we obtain:

\begin{corollary}\label{critical}
Let $\I,\J$ be ideals on $\omega$ and let $g,h\in{}^\omega\omega$ be such that\/ $\lim_{n\in\omega}g(n)=\infty$ and $h\geq^*1$.
\begin{alignat*}{5}
&\non{\mathrm{S}_1(\Gamma_h,\Gamma)}&&=\nonm,&\qquad& \non{\mathrm{S}_1(\Gamma_g,\Fin\rr\Lambda)} && =\cofn,\\
&\non{\mathrm{S}_1(\Gamma_h,\GammaB{{J}})}&&=\cLambda{h}{{J}}, && \non{\mathrm{S}_1(\Gamma_h,\LambdaB{J})} && =\tsl{h}{{J}},\\     
&\non{\mathrm{S}_1(\Gamma_{b,h},\GammaB{{J}})}&&=\cLambda{{b,h}}{{J}}, &&\non{\mathrm{S}_1(\Gamma_{b,h},\LambdaB{J})}&&=\tsl{{b,h}}{{J}},\\     
&\non{\isoneggfi{\J}}&&=\bfrak_\J, &&  \non{\isoneglfi{\J}} && =\dfrak_\J,\\
&\non{\isonegg{\I}{\J}}&&=\sla{\I,\J}, && \non{\Srm_1(I\rr\Gamma,\LambdaB{J})} &&=\dsla{\I,\J},\\
&\non{\isonemg{\J}}&&=\slamg{\J}, && \non{\Srm_1(\Omega,\LambdaB{J})} &&=\dslaml{\J}.
\end{alignat*}
Even more, when $J=\Fin$, $J\rr\Lambda$ can be replaced by $\OO$ in the right side column, as long as $g$ and $h$ are non-zero everywhere and $h(n)<|b(n)|$ for infinitely many $n<\omega$.\qed
\end{corollary}


The equalities in the bottom three lines of \Cref{critical} have been proven in~\cite{SoSu,Su22}. The equalities for $\sla{h,\J}$ and $\dsla{h,\J}$ were obtained in~\cite{SoDiz}. We visualize \Cref{critical} in~\Cref{SjednaABcard} and~\ref{critical_function}. In fact, the rows using a~function~$h$ in the latter diagram were not considered in~\cite{SoSu,Su22} and are new to this work.

\begin{figure}[ht]
\centering
\begin{tikzpicture}[scale=0.8]
\small{
\node (a1) at (-4.5, -3) {$\schema{\Omega}{\Gamma}$};
\node (a2) at (0, -3) {$\schema{\Omega}{\GammaB{{J}}}$};
\node (a3) at (4.5, -3) {$\schema{\Omega}{\LambdaB{{J}}}$};
\node (a4) at (9, -3) {$\schema{\mathcal{O}}{\mathcal{O}}$};
\node (aa1) at (-4.5, -3.7) {\footnotesize$\mathfrak{p}$};
\node (aa2) at (0, -3.7) {\footnotesize$\cLambda{\star}{{J}}$};
\node (aa3) at (4.5, -3.7) {\footnotesize$\tsl{\star}{{J}}$};
\node (aa4) at (9, -3.7) {\footnotesize$\covm$};

\node (b1) at (-4.5, 0) {$\schema{\GammaB{{I}}}{\Gamma}$};
\node (b2) at (0, 0) {$\schema{\GammaB{{I}}}{\GammaB{{J}}}$};
\node (b3) at (4.5, 0) {$\schema{\GammaB{{I}}}{\LambdaB{{J}}}$};
\node (b4) at (9, 0) {$\schema{\GammaB{{I}}}{\mathcal{O}}$};
\node (bb1) at (-4.5, -0.7) {\footnotesize$\min \{\covh{{I}}, \mathfrak{b}\}$};
\node (bb2) at (0, -0.7) {\footnotesize$\cLambda{{I}}{{J}}$};
\node (bb3) at (4.5, -0.7) {\footnotesize$\tsl{{I}}{{J}}$};
\node (bb4) at (9, -0.7) {\footnotesize$\tsl{{I}}{\fin}$};

\node (c1) at (-4.5, 3) {$\schema{\Gamma}{\Gamma}$};
\node (c2) at (0, 3) {$\schema{\Gamma}{\GammaB{{J}}}$};
\node (c3) at (4.5, 3) {$\schema{\Gamma}{\LambdaB{{J}}}$};
\node (c4) at (9, 3) {$\schema{\Gamma}{\mathcal{O}}$};
\node (cc1) at (-4.5, 2.3) {\footnotesize$\mathfrak{b}$};
\node (cc2) at (0, 2.3) {\footnotesize$\be_{{J}}$};
\node (cc3) at (4.5, 2.3) {\footnotesize$\de_{{J}}$};
\node (cc4) at (9, 2.3) {\footnotesize$\mathfrak{d}$};

\node (d1) at (-4.5, 6) {$\schema{\Gammah{h}}{\Gamma}$};
\node (d2) at (0, 6) {$\schema{\Gammah{h}}{\GammaB{{J}}}$};
\node (d3) at (4.5, 6) {$\schema{\Gammah{h}}{\LambdaB{{J}}}$};
\node (d4) at (9, 6) {$\schema{\Gammah{h}}{\mathcal{O}}$};
\node (dd1) at (-4.5, 5.3) {\footnotesize$\nonm$};
\node (dd2) at (0, 5.3) {\footnotesize$\cLambda{h}{{J}}$};
\node (dd3) at (4.5, 5.3) {\footnotesize$\tsl{h}{{J}}$};
\node (dd4) at (9, 5.3) {\footnotesize$\cofn$};
}
\foreach \from/\to in {a1/a2, a2/a3, a3/a4, b1/b2, b2/b3, b3/b4, c1/c2, c2/c3, c3/c4, d1/d2,d2/d3, d3/d4, a1/bb1,b1/cc1, c1/dd1, a2/bb2, b2/cc2, c2/dd2, a3/bb3, b3/cc3, c3/dd3, a4/bb4, b4/cc4, c4/dd4}
\draw [->] (\from) -- (\to);
\end{tikzpicture}
\caption{Critical cardinality of selection principles when $h\to\infty$. If $h\geq^* 1$ and $h\nrightarrow\infty$, $\cofn$ is replaced by $\cfrak$.}
\label{SjednaABcard}
\end{figure}

\begin{figure}[h!]
\begin{center}
\begin{tikzpicture}[]
\node (a) at (-5, -3.5) {?};
\node (na) at (-5, -3.9) {$\addn$};
\node (as) at (-5, -1) {$\mathrm{S}_1(\Gamma_{b,h},\Gamma)$};
\node (nas) at (-5, -1.4) {$\sla{b,h,\fin}$};
\node (b) at (-5, 1.5) {$\mathrm{S}_1(\Gamma_h,\Gamma)$};
\node (nb) at (-5, 1.1) {$\nonm$};
\node (aas) at (-2, -1) {$\mathrm{S}_1(\Gamma_{b,h},\GammaB{{J}})$};
\node (naas) at (-2, -1.4) {$\sla{b,h,\J}$};
\node (bb) at (-2, 1.5) {$\mathrm{S}_1(\Gamma_h,\GammaB{{J}})$};
\node (nbb) at (-2, 1.1) {$\sla{h,\J}$};
\node (cs) at (1, -1) {$\mathrm{S}_1(\Gamma_{b,h},\LambdaB{J})$};
\node (ncs) at (1, -1.4) {$\dsla{b,h,\J}$};
\node (d) at (1, 1.5) {$\mathrm{S}_1(\Gamma_h,\LambdaB{J})$};
\node (nd) at (1, 1.1) {$\dsla{h,\J}$};
\node (e) at (4, -3.5) {$\soneoo$};
\node (ne) at (4, -3.9) {$\covm$};
\node (es) at (4, -1) {$\mathrm{S}_1(\Gamma_{b,h},\OO)$};
\node (nes) at (4, -1.4) {$\dsla{b,h,\fin}$};
\node (f) at (4, 1.5) {$\mathrm{S}_1(\Gamma_h,\OO)$};
\node (nf) at (4, 1.1) {$\cofn$};
\foreach \from/\to in {a/nas, e/nes, d/f, b/bb, bb/d, as/nb, aas/nbb, cs/nd, es/nf, as/aas, aas/cs, cs/es, a/e} \draw [->] (\from) -- (\to);
\end{tikzpicture}
\end{center}
\caption{Selection principles for slaloms bounded by a~function $b$. The question mark indicates that it is not known which selection principle has $\addn$ as its critical cardinality. When $h\nrightarrow \infty$, replace $\cofn$ by $\cfrak$.}
\label{critical_function}
\end{figure}

Recall from~\cite{Comb1} that $\Srm_1(\Omega,\OO)$ and $\Srm_1(\OO,\OO)$ are equivalent principles.

Strict inequalities between two cardinal characteristics reflect the existence of spaces satisfying one selection principle but not the other. 

\begin{corollary}\label{nonsgamah}
If\/ $\non{\Srm_1^{D'}(\Pcal'\rr\Gamma,R'\rr\Gamma)} <\non{\Srm_1^D(\Pcal\rr\Gamma,R\rr\Gamma)}$ then there is an\/ $\Srm_1^D(\Pcal\rr\Gamma,R\rr\Gamma)$-space which is not an\/ $\Srm_1^{D'}(\Pcal'\rr\Gamma,R'\rr\Gamma)$-space.\qed








\end{corollary}



Another application of \Cref{critical} is the following consequence for cardinal invariants. 
\begin{corollary}
    Let $h\geq^*1$. Then, $\min\{\sla{h,J},\covh{J}\}\leq\nonm$.
\end{corollary}

\begin{proof}
If a topological space $X$ is both an $\schema{\Gammah{h}}{\GammaB{{J}}}$-space and a $[J\rr\Gamma,\Gamma]$-space, then $X$ is an $\schema{\Gammah{h}}{\Gamma}$-space. This implies that the minimum of the critical cardinalities of $\schema{\Gammah{h}}{\GammaB{{J}}}$ and $[J\rr\Gamma,\Gamma]$ is below $\non{\schema{\Gammah{h}}{\Gamma}}$. 
On the other hand, by \Cref{critical} $\non{\schema{\Gammah{h}}{\GammaB{{J}}}}=\sla{h,J}$, $\non{\schema{\Gammah{h}}{\Gamma}}=\nonm$, and $\non{[J\rr\Gamma,\Gamma]}= \pfrak_\Kat(J,\Fin) =\covh{J}$ (the latter was directly proved in~\cite{SoSu}). 

We also present a combinatorial proof. Since $\sla{h,J}\leq\sla{1,J}$, it is enough to work with $h=1$. 
We use that $\nonm$ and $\covm$ are the $\bfrak$ and $\dfrak$-numbers, respectively, of the relational system $\Ed^*:=\la \Baire,[\omega]^{\aleph_0}\times\Baire,\eqcirc\ra$, where $x\eqcirc (w,y)$ means that $x(i)\neq y(i)$ for all but finitely many $i\in w$ (see \cite{BJ} and \cite[Thm.~5.3]{CM23}). So let $F\subseteq\Baire$ of size ${<}\min\{\sla{1,J},\covh{J}\}$. Then, we can find some $y\in\Baire$ such that $a_x:=\|x = y\|\in J$ for all $x\in F$. Since $|F|<\covh{J}$, we can find some $w\in[\omega]^{\aleph_0}$ such that $w\cap a_x$ is finite for all $x\in F$, which implies $x \eqcirc (w,y)$.

The latter argument can be easily modified to show that $\covm\leq \max\{\slalome^\perp(1,J),\nonst(J)\}$.
\end{proof}

The rest of this section is devoted to studying topological properties of $\Srm_1^D(\Pcal\rr\Gamma,R\rr\Gamma)$-spaces, in particular, we show that in most cases these spaces are totally imperfect, i.e., they do not contain a subspace homeomorphic with the Cantor space. These results are generalizations of results obtained in~\cite{SoDiz}.

For $H\subseteq\pts(a)$, we say that a sequence $\seqn{V_m}{m\in a}$ of sets is \emph{$H$-wise disjoint} if $\bigcap_{m\in w}V_m=\emptyset$ for any $w\in\pts(a)\menos H$. In the case $H=[a]^{< q}$ for some $0<q<\omega$, we say that an $[a]^{< q}$-wise disjoint sequence is \emph{$q$-wise disjoint}, i.e.,\ $\bigcap_{m\in w}V_m = \emptyset$ for any $w\subseteq a$ of size $q$.


\begin{prop}\label{chargammac} 
Let $H\subseteq \pts(a)$ and let $X$ be a topological space. 
A~sequence\/ $\seqn{U_m}{m\in a}$ of open subsets of $X$ is an~$H$-$\gamma$-cover of~$X$ if and only if\/ $\set{X\smallsetminus U_m}{m\in a}$ is an\/~$H$-wise disjoint sequence of closed subsets of $X$. In particular, for $q<\omega$, $\seqn{U_m}{m\in a}$ is a~$\gamma_q$-cover if and only if\/ $\set{X\smallsetminus U_m}{m\in a}$ is a~${q+1}$-wise disjoint sequence of closed subsets of~$X$.
\end{prop}

\begin{proof} Assume that $\seqn{U_m}{m\in a}$ is an~$H$-$\gamma$-cover of~$X$. For $x\in X$,
$\set{m\in a}{x\in X\smallsetminus U_m} = \set{m\in a}{x\notin U_m}\in H$, so $\bigcap_{m\in w}X\menos U_n =\emptyset$ for any $w\subseteq a$ not in $H$.

Conversely, assume that $\set{X\smallsetminus U_m}{m\in a}$ is an~$H$-wise disjoint sequence of closed sets. Let $x\in X$ and $v:=\set{m\in a}{x\notin U_m}$. Then $x\in \bigcap_{m\in v} (X\smallsetminus U_m)$, so $v\in H$. Thus, the sequence $\seqn{U_m}{m\in a}$ is an~$H$-$\gamma$-cover of $X$.
\end{proof}

Continuous mappings preserve $H$-$\gamma$-covers.

\begin{lemma}\label{continuous} 
Let $X$ and $Y$ be two topological spaces, $H\subseteq\pts(a)$, and let
$f\colon X\rightarrow Y$ be a~continuous mapping. If\/ $\seqn{U_{m}}{m\in a}$ is an open~$H$-$\gamma$-cover of~$Y$, then\/ $\seqn{f^{-1}\llbracket U_m\rrbracket}{m\in a}$ is an~$H$-$\gamma$-cover of~$X$.
\end{lemma}

\begin{proof}
Let $\seqn{U_m}{m\in a}$ be an open $H$-$\gamma$-cover of $Y$. Fix $x\in X$. Since $f$ is continuous, $f^{-1}\llbracket U_m\rrbracket$ is open for each $m\in a$  and $\set{m \in a}{x\notin f^{-1}\llbracket U_m\rrbracket} = \set{m\in a}{f(x)\notin U_m}\in H$. Thus $\seqn{f^{-1}\llbracket U_m\rrbracket}{m\in a}$ is an~$H$-$\gamma$-cover of $X$.
\end{proof}

Since being an $\omega$-cover is equivalent to being an $I$-$\gamma$-cover for some ideal $I$, it follows that:

\begin{corollary}
    If $f\colon X\rightarrow Y$ is continuous and\/ $\seqn{U_{m}}{m\in a}$ is an open~$\omega$-cover of~$Y$, then\/ $\seqn{f^{-1}\llbracket U_m\rrbracket}{m\in a}$ is an~$\omega$-cover of~$X$.\qed
\end{corollary}


The selection principle $\Srm_1^D(\Pcal\rr\Gamma,R\rr\Gamma)$ is preserved under continuous images and closed subsets.

\begin{lemma}\label{spojitostSGHJ}
Assume that $X$ is an\/ $\Srm_1^D(\Pcal\rr\Gamma,R\rr\Gamma)$-space.
\begin{enumerate}[label = \normalfont (\alph*)]
    \item\label{sp-a} If $f\colon X\to Y$ is a continuous surjection then $Y$ is an\/ $\Srm_1^D(\Pcal\rr\Gamma,R\rr\Gamma)$-space.
    \item\label{sp-b} If $S\cup\{\overline{\emptyset}\}\in\Pcal$ for all $S\in\Pcal$ and $Z\subseteq X$ is closed, then $Z$ is an\/ $\Srm_1^D(\Pcal\rr\Gamma,R\rr\Gamma)$-space.
\end{enumerate}
The same is valid when $R\rr\Gamma$ is replaced by $\Omega$.
\end{lemma}

\begin{proof}
\ref{sp-a}: Let $\seqn{\seqn{U_{n,m}}{m \in b(n)}}{n\in\omega}\in\Pcal\rr\Gamma(Y)$. Notice that
\begin{align*}
\bigset{\seqn{\set{m\in b(n)}{x\notin f^{-1}\llbracket U_{n,m}\rrbracket}}{n<\omega}}{x\in X} & = \bigset{\seqn{\set{m\in b(n)}{f(x)\notin  U_{n,m}}}{n<\omega}}{x\in X} \\ 
& = \bigset{\seqn{\set{m\in b(n)}{y\notin U_{n,m}}}{n<\omega}}{y\in Y}\in \Pcal,
\end{align*}
so $\seqn{\seqn{f^{-1}\llbracket U_{n,m}\rrbracket}{m \in b(n)}}{n\in\omega}\in\Gamma\rr\Pcal(X)$.
Consequently, since $X$ is an $\Srm_1^D(\Pcal\rr\Gamma,R\rr\Gamma)$-space, 
there is some $d\in D$ such that $\seqn{f^{-1}\llbracket U_{n,d(n)}\rrbracket}{n\in\omega} \in R\rr\Gamma(X)$. For $y\in Y$, if $x\in X$ and $y=f(x)$, then
\[
\set{n\in\omega}{y\notin U_{n,d(n)}} = 
\set{n\in\omega}{x\notin f^{-1}\llbracket U_{n,\varphi(n)}\rrbracket}\in R.
\]
Thus $\seqn{U_{n,d(n)}}{n\in\omega}$ is an~$R$-$\gamma$-cover of $Y$. 

\ref{sp-b}: Let $\seqn{\seqn{U_{n,m}}{m \in b(n)}}{n\in\omega}\in\Gamma\rr\Pcal(Z)$. Set $V_{n,m}:= U_{n,m}\cup (X\menos Z)$, which is open in $X$. Then, 
\begin{align*}
\bigset{\seqn{\set{m\in b(n)}{x\notin V_{n,m}}}{n<\omega}}{x\in X} = \bigset{\seqn{\set{m\in b(n)}{x\notin  U_{n,m}}}{n<\omega}}{x\in Z}\cup\{\overline{\emptyset}\}\in \Pcal.
\end{align*}
Since $X$ is an $\Srm_1^D(\Pcal,R\rr\Gamma)$-space, 
there is some $d\in D$ such that $\seqn{V_{n,d(n)}}{n\in\omega} \in R\rr\Gamma(X)$. Then, for $x\in Z$,
\[\set{n<\omega}{x\notin V_{n,d(n)}} = \set{n<\omega}{x\notin U_{n,d(n)}}\in R.\]
Thus $\seqn{U_{n,d(n)}}{n\in\omega}$ is an~$R$-$\gamma$-cover of $Z$. 
\end{proof}

As a particular case, we emphasize the weakest versions of our selection principles, given by $\schema{\Gammah{h}}{\Gamma}$, $\schema{\Gammah{h}}{\OO}$ and with $h=q$ constant. 
These principles are related as in \Cref{schema-gamah-gamma}.

\begin{figure}[ht]
\centering
\begin{tikzpicture}[scale=0.8]
\node (b) at (-4.5, 0) {$\schema{\Gamma}{\Gamma}$};
\node (c) at (0, 0) {$\schema{\Gammah{h}}{\Gamma}$};
\node (d) at (-4.5, 1.5) {$\schema{\Gammah{q}}{\Gamma}$};
\node (e) at (0, 1.5) {$\schema{\Gammah{1}}{\Gamma}$};
\node (f) at (4.5, 0) {$\schema{\Gammah{h}}{\OO}$};
\node (g) at (4.5, 1.5) {$\schema{\Gammah{1}}{\OO}$};

\foreach \from/\to in {b/c, b/d, d/e, c/e,c/f,e/g,f/g}
\draw [->] (\from) -- (\to);
\end{tikzpicture}
\caption{Relations with respect to the well-known $\schema{\Gamma}{\Gamma}$-space.}
\label{schema-gamah-gamma}
\end{figure}

J.~Gerlits and Zs.~Nagy~\cite{GerNag} have introduced the~notion of a~$\gamma$-set, i.e., a~topological space with every $\omega$-cover having a~$\gamma$-subcover. They have shown that a~topological space $X$ is a~$\gamma$-set if and only if $X$ is an $\Srm_1(\Omega,\Gamma)$-space. Hence, all $\gamma$-sets are examples of~$\schema{\Gammah{h}}{\Gamma}$-spaces. On the other hand, we show that topological spaces satisfying the selection principle are totally imperfect. Notice that $\Srm_1(\Gamma_1,\OO)$ is the weakest among all the interesting selection principles. In fact, if $\Pcal$ is $\subseteq$-downwards closed, $\Scal(b,1)\in\Pcal$ and $\omega\notin R$, then any $\Srm^D_1(\Pcal\rr\Gamma,R\rr \Gamma)$-space is $\Srm_1(\Gamma_1,\OO)$ (likewise if replacing $R\rr\Gamma$ by $\Omega$).

\begin{theorem}\label{CantorShfin}
\startlist 
\begin{enumerate}[label = \normalfont (\alph*)]
    \item\label{CSf1} The Cantor space is not\/ $\Srm_1(\Gamma_1,\OO)$.
    \item\label{CSf2} Any Hausdorff\/ $\Srm_1(\Gamma_1,\OO)$-space is totally imperfect.
    \item\label{CSf3} No uncountable Polish space is\/ $\Srm_1(\Gamma_1,\OO)$.
\end{enumerate}
\end{theorem}

\begin{proof}
\ref{CSf1}: Fix a bijection $f\colon{\omega\times\omega}\to\omega$. For $n,m<\omega$, define closed set 
$$F_{n,m}=\set{x\in{}^\omega 2}{x(f(n,m))=1,\ x(f(n,i))=0 \text{ for all $i\neq m$}}.$$  
The sequence $\seqn{{}^\omega2\smallsetminus F_{n,m}}{m\in\omega}$ is a~$\gamma_1$-cover for any $n\in\omega$. Assume that $d\in{}^\omega\omega$ and define $x\in{}^\omega 2$ by
\begin{equation*}
x(i) = \begin{cases}
1& \text{if }\exists n\in\omega\ (f(n,\varphi(n))=i),\\
0&\text{otherwise.}
\end{cases}
\end{equation*}
It is clear that $x\in F_{n,d(n)}$ for each $n\in\omega$, thus $\seqn{{}^\omega 2\smallsetminus F_{n,d(n)}}{n\in\omega}$ does not cover ${}^\omega 2$.

\ref{CSf2}: It follows from the fact that no $\Srm_1(\Gamma_1,\OO)$ contains a subspace isomorphic with the Cantor space by \Cref{spojitostSGHJ} and~\ref{CSf1}.

\ref{CSf3}: By~\ref{CSf2} because no uncountable Polish space is totally imperfect.
\end{proof}

\section{Consistency results}\label{sec:forcing}

This section aims to show the behavior of our slalom numbers in forcing models. We focus on models constructed via finite support iteration and pay special attention to the effect of adding Cohen reals.

As usual in forcing arguments, we work in a~ground model~$V$ unless otherwise indicated. For two posets $\Por$ and $\Qor$, $\Por\subsetdot \Qor$ means that the inclusion map is a complete embedding from $\Por$ into $\Qor$. 
When $\seqn{\Por_\alpha}{\alpha\leq\beta}$ is a~$\subsetdot$-increasing sequence of posets (like an iteration) and $G$ is $\Por_\beta$-generic over $V$, we denote, for $\alpha\leq\beta$,
$G_\alpha:= \Por_\alpha\cap G$ and $V_\alpha := V[G_\alpha]$. If $\Por_{\alpha+1}$ is obtained by a~two-step iteration $\Por_\alpha\ast\Qnm_\alpha$, $G(\alpha)$ denotes the $\Qnm[G_\alpha]$-generic
set over $V_\alpha$ such that $V_{\alpha+1}= V_\alpha[G(\alpha)]$ (i.e., $G_{\alpha+1} = G_\alpha\ast G(\alpha)$). We use $\Vdash_\alpha$ to denote the forcing relation on $\Por_\alpha$, and $\leq_\alpha$ to denote its order relation (although we use $\leq$ when clear from the context). 

\subsection{Effect of Cohen reals}\label{sec:cohen}

Recall the following well-known result from Canjar.

\begin{lemma}[Canjar~{\cite{Canjar2}}]\label{Jcohen}
Let $\J\subseteq\pts(\omega)$ be a~family with the FUP. If $c\in \Baire$ is Cohen over~$V$, then
\[
\J\cup \set{\set{i<\omega}{c(i)<x(i)}}{x\in\Baire \cap V} \text{ has the FUP.}
\]
As a consequence, $x\leq^{{J'}^{\dual}}c$ for all $x\in\Baire\cap V$, where $J'$ is the ideal on $\omega$ generated by the family above. Moreover, any $J$-positive set in $V$ is $J'$-positive, i.e.,\ $J'\cap V = J$.\qed
\end{lemma}

We extend this result in connection to slalom numbers. First, fix some notation.

\begin{notation}
Let $J$ be an ideal on $\omega$. We say that a~function $h\in\Baire$ is \emph{$J$-unbounded} if the set $\set{n<\omega}{h(n)\geq k}$ is not in $J$ for all $k<\omega$, and that $\lim^J h = \infty$ if $\set{n<\omega}{h(n)< k}\in J$ for all $k<\omega$.
\end{notation} 

\begin{lemma}\label{cohen-slt_h}
Let $J$ be an ideal on $\omega$ (or just a family with the FUP) and let $h\in\Baire$ such that $h$ is $J$-unbounded. If $c\in\prod_{n\in\omega}[\omega]^{\leq h(n)}$ is Cohen over $V$, then $V[c]\models J\cup \set{\set{n<\omega}{x(n)\not\in c(n)}}{x\in\Baire \cap V}$ has the FUP. In particular, this set generates an ideal $J'$ such that $x\in^{{J'}^\dual} c$ for all $x\in\Baire\cap V$. Moreover, $\lim^{J'} h =\infty$ and, whenever\/ $\lim^J h=\infty$, $J'\cap V = J$.
\end{lemma}

\begin{proof}
In this proof, we consider Cohen forcing $\Cor$ as the set of conditions $p\in\prod_{i\in u}[\omega]^{\leq h(i)}$ for some $u\in\Fin$, ordered by $\supseteq$. We denote the name of its generic real by $\dot c$.

Working in $V$,
suppose that $F\subseteq\baire\omega$ is finite, $a\in J$, $k<\omega$ and $p\in\Cor$. It is enough to prove that there is a~$q\leq p$ and an $m\in \omega\menos a$ such that $h(m)\geq k$ and $q\Vdash (\forall x\in F)\ x(m)\in \dot c(m)$. Since $h$ is $J$-unbounded, there is an $m\in \omega\menos (a\cup \dom p)$ such that $\max\{k,|F|\}\leq h(m)$. Next, define a~function $q\supseteq p$ such that $\dom q:=\dom p\cup\{m\}$ and $q(m):=\set{x(m)}{x\in F}$. Note that $|q(m)|\leq|F|\leq h(m)$, so $q\in\Cor$ and $q\leq p$. It is clear that $q$ forces what we want.

By using the constant functions in $\Baire$, since $c$ dominates $\Baire\cap V$, we can conclude that $\lim^{J'} h=\infty$. Concretely, for each $n\in\omega$ and $k\leq n$, $\|k\in  c\|\in \dfil{J'}$, so $\|n\subseteq c\|\in \dfil{J'}$, which implies that $\|n\leq h\|\in\dfil{J'}$ because $c\in\Scal(\omega,h)$.

In the case when $\lim^J h=\infty$, the proof above can be modified to find $m\in a'\menos a$ for any given $a'\in J^+$. Since $J\subseteq J'$ in $V[c]$, it is clear that $\lim^{J'} h=\infty$.
\end{proof}

\begin{lemma}\label{clm:cohenev}
Let $\calS\in V$ be a~set of slaloms with the FUPC and let $\J$ be an ideal on $\omega$ (or just a~family with the FUP). If $c\in\Baire$ is Cohen over $V$ then, in $V[c]$, $\J\cup\set{\set{i<\omega}{c(i)\in S(i)}}{S\in\calS}$ has the FUP. In particular, this generates an ideal $\J'$ such that $c \notin^{{J'}^+} S$ for all $S\in\calS$. Moreover, $J'\cap V = J$.
\end{lemma}

\begin{proof}
Consider Cohen forcing $\Cor$ as the set of finite partial functions $\omega\to\omega$, ordered by $\supseteq$. 

Let $\F\subseteq\calS$ be finite, $a\in\J$, $a'\in J^+$ and $p\in\Cor$. It is enough to show that there are $i\in a'\menos a$ and $q\leq p$ such that $q\Vdash (\forall S\in\F)\ c(i)\notin S(i)$. Pick any $i\in a'\menos(a\cup\dom p)$. By the FUPC of $\calS$, $\bigcup_{S\in\F}S(i)\neq\omega$, so choose $k\in\omega\smallsetminus \bigcup_{S\in\F}S(i)$. Any $q\leq p$ such that $q(i)=k$ is as required.
\end{proof}


As a consequence of these results, adding Cohen reals strongly affects slalom numbers with ideals. Recall from \Cref{BasicDia} and \Cref{DiaDual} that many slalom numbers are between $\slalome(\star,J)$ and $\slalomt(h,J)$, and between $\slalomt^\perp(h,J)$ and $\slalomt(h,J)$. 

\begin{theorem}\label{seq:sn}
    Let $\pi$ be an ordinal with uncountable cofinality, $J_0$ an ideal on $\omega$ and let\/ $\seqn{\Por_{\alpha}}{\alpha\leq\pi}$ be an\/ 
    $\subsetdot$-increasing sequence of posets such that\/ $\Por_\pi = \bigcup_{\alpha<\pi}\Por_\alpha$. Assume that\/ $\Por_\pi$ has\/ $\cf(\pi)$-cc and that\/ $\Por_{\alpha+1}$ adds a~Cohen real over $V_\alpha$ for all $\alpha<\pi$. Let $\lambda:=|\pi|$. Then:
    \begin{enumerate}[label=\rm(\alph*)]
        \item\label{seq:sn-a} $\Por_\pi$ forces that, for any $J_0$-unbounded $h\colon\omega\to\omega$, there is a~(maximal) ideal $\J\supseteq J_0$ such that\/ $\slalomt^\perp(h,J)= \dsla{h,\J}=\cf(\pi)$ and $h$ is $J$-unbounded. 
    \end{enumerate}
    For the following items, further assume that\/ $\Vdash \cfrak=\lambda$ and that $\lambda$ divides $\pi$, i.e., $\pi=\lambda\delta$ for some ordinal~$\delta$.\footnote{We must have $\delta<\lambda^+$, otherwise $\Por_\pi$ would add too many Cohen reals an force $\lambda>\cfrak$.}
    \begin{enumerate}[resume*]
        \item\label{seq:sn-a2} 
        $\Por_\pi$ forces that there is an ideal $\J\supseteq J_0$ such that\/ $\slalomt^\perp(h,J)=\dsla{h,\J}=\cf(\pi)$
        for any $J$-unbounded $h\colon\omega\to\omega$. This implies that $J$ is maximal.

        \item\label{seq:sn-b} Let $\theta\leq \cf(\pi)$ be a~cardinal. If $\lambda^{<\theta} = \lambda$ then\/ $\Por_\pi$ forces  that there is a~maximal ideal $\J\supseteq J_0$ such that $\theta\leq \sla{*,\J}\leq \slalomt^\perp(h,J)=\dsla{h,\J}=\cf(\pi)$ for any $J$-unbounded $h\colon\omega\to\omega$.

        \item\label{seq:sn-c} If  $\lambda^{<\cf(\pi)} = \lambda$ then\/ $\Por_\pi$ forces that there is a~maximal ideal $\J\supseteq J_0$ such that\/  $\sla{*,\J}= \slalomt^\perp(h,J)=\dsla{h,\J}=\cf(\pi)\leq\ppk{*,\J}$ for any $J$-unbounded $h\colon\omega\to\omega$.
    \end{enumerate}    
\end{theorem}

\begin{proof}
\ref{seq:sn-a}: 
Fix a $J_0$-unbounded $h\in\Baire\cap V_\pi$, so $h\in V_\alpha$ for some $\alpha<\pi$. In the following argument, it does not hurt to consider $\alpha=0$.

For any $\zeta<\pi$,
denote by $c_\zeta$ a~Cohen real in $\Scal(\omega,h)$ that $\Por_{\zeta+1}$ adds over $V_\zeta$. 
In $V_{\zeta+1}$, define $A_\zeta := \bigset{\set{n<\omega}{x(n)\notin c_{\zeta}(n)}}{x\in\Baire\cap V_\zeta}$. 
By employing~\Cref{cohen-slt_h}, we can prove by recursion on $\zeta\leq\pi$ that, in $V_{\zeta}$, the family $\J_\zeta= J_0\cup \bigcup_{\xi<\zeta}\A_\xi$  has the FUP. Lastly, in $V_\pi$, let $\J$ be the ideal generated by~$\J_{\pi}$. Therefore, $\set{c_\zeta}{ \zeta\in K}$ is a~witness for $\dsla{h,\J}$ for any cofinal $K\subseteq\pi$, so $\dsla{h,\J}\leq\cf(\pi)$. On the other hand, any $F\subseteq\Baire$ of size ${<}\cf(\pi)$ is $\in^{J^\dual}$-bounded by some $c_\zeta$, hence $\cf(\pi)\leq\slalomt^\perp(h,J)$. Moreover, $\lim^{J}h = \infty$ (because $\lim^{J'_1}h=\infty$ after the first application of \Cref{cohen-slt_h}, where $J'_1$ is the ideal generated by $J_1$).  
Notice that we can extend $J$ to a~maximal ideal without affecting the result.

\ref{seq:sn-a2}: This proof is similar to~\ref{seq:sn-a}, but we need a~book-keeping to find one $\J$ that works for all $h$. Let $j\colon\lambda\to\lambda\times\lambda$ be a~bijection such that $j(\varepsilon) = (\xi,\xi')$ implies that $\xi\leq \varepsilon$. On the other hand, for $\alpha<\delta$, let $I_\alpha:=[\lambda\alpha,\lambda(\alpha+1))$, which is an interval of order type $\lambda$. Note that
$\seqn{I_\alpha}{\alpha<\theta}$ is an interval partition of~$\pi$.

We define a~$\Por_{\zeta}$-name $\dot\J_\zeta$ of an ideal on $\omega$ by recursion on $\zeta\leq\pi$ as follows. We start with $\dot\J_0:=J_0$ and, for limit $\zeta$, $\dot\J_\zeta$ is a~$\Por_\zeta$-name of the ideal generated by $\bigcup_{\eta<\zeta}\dot\J_\eta$, so we are left with the induction step $\zeta=\eta+1$. Pick $\alpha<\theta$ such that $\eta\in I_\alpha$ (which is unique), so $\eta=\lambda\alpha+\rho$ for some unique $\rho<\lambda$. Enumerate all the nice $\Por_\eta$-names of $\dot J_\eta$-unbounded functions in $\Baire$ by $\set{\dot h^\eta_{\xi'}}{\xi'<\lambda}$ (which is possible because $\Vdash_\pi \cfrak=\lambda$). For convenience, we also denote $\dot h^{\alpha}_{\rho,\xi'} := \dot h^\eta_{\xi'}$.

Let us define $\dot\J_\zeta$ by cases: 
we let $\dot J_\zeta$ be a~$\Por_\zeta$-name of the ideal generated by $\dot J_\eta\cup\set{\set{n<\omega}{x(n)\in \dot c_{\eta}(n)}}{x\in\Baire\cap V_{\eta}}$ where $\dot c_\eta$ is the $\Por_{\eta+1}$-name of a~Cohen real in $\Scal(\omega,\dot h^\alpha_{j(\varepsilon)})$ over $V_\eta$ in the case when $\Por_\eta$ forces that $\dot h^\alpha_{j(\varepsilon)}$ is $\dot J_\eta$-unbounded (if $j(\varepsilon) = (\xi,\xi')$ then $\xi\leq\varepsilon$, so $\dot h^\alpha_{j(\varepsilon)} = \dot h^\alpha_{\xi,\xi'}$ was already defined at step $\lambda\alpha+\xi\leq \eta$), otherwise let $\dot J_\zeta$ be a $\Por_\zeta$-name of the ideal generated by $\dot J_\eta$. Thanks to \Cref{cohen-slt_h}, each $\dot J_\zeta$ is forced to have the FUP.

Let $\dot J:=\dot J_\pi$. We prove that $\dot J$ is forced as required, i.e., 
$\Vdash_\pi \slalomt^\perp(\dot h,\dot J)= \slalomt(\dot h,\dot J)=\cf(\pi)$ for any (nice) $\Por_\pi$-name $\dot h$ of a~$\dot J$-unbounded function in~$\Baire$. Since $\cf(\pi)>\omega$ and $\Por_\pi$ is $\cf(\pi)$-cc, there is some $\zeta_0<\pi$ such that $\dot h$ is a~$\Por_{\zeta_0}$-name. Then, for any $\zeta_0\leq \zeta<\pi$, $\dot h$ appears in the enumeration $\set{\dot h^\zeta_{\xi'}}{\xi'<\lambda}$, meaning that there is some cofinal subset $K\subseteq \pi$ where the Cohen real $\dot c_\eta$ described in the successor step of the construction of $\dot J$ is in $\Scal(\omega,\dot h)$. Therefore, as in~\ref{seq:sn-a}, $\Por_\pi$ forces $\slalomt^\perp(\dot h,\dot J)=\slalomt(\dot h,\dot J)= \cf(\pi)$.

We now prove that $\Por_\pi$ forces that $\dot J$ is maximal. In $V_\pi$, let $J:=\dot J[G_\pi]$, so $\slalomt(h,J) = \cf(\pi)$ for any $J$-unbounded $h\in\Baire$. If $J$ is not maximal, we can find some $J$-unbounded $h\in\Baire$ such that $\|h=0\|\in J^+$, but this implies that $\slalomt(h,J)$ is undefined, a contradiction.

\ref{seq:sn-b}:
The construction of the ideal is similar to~\ref{seq:sn-a2}, so we keep the same notation from there, e.g.\ the book-keeping function $j$ and the interval $I_\alpha$.

For each $\zeta<\pi$, enumerate $\set{\dot\calS^\zeta_{\xi'}}{\xi'<\lambda}$ the nice $\Por_\zeta$-names of all sets of slaloms in $V_\zeta$ with the FUPC of size~${<}\theta$ (this is possible by the assumption $\lambda^{<\theta}=\lambda$ and $\Vdash_\pi \cfrak=\lambda$). Note that $\set{\dot\calS^\zeta_\xi}{\zeta<\pi,\ \xi'<\lambda}$ enumerates all the nice $\Por_{\pi}$-names of all sets of slaloms in $V_\pi$ with the FUPC of size~${<}\theta$.

For each $\zeta\leq\pi$ we define a~$\Por_{\zeta}$-name $\dot\J_\zeta$ of a~family with the FUP as follows: $\dot\J_0:= J_0$ and, for limit~$\zeta$, $\dot\J_\zeta:=\bigcup_{\xi<\zeta}\dot\J_\xi$, so we are left with the induction step $\zeta=\eta+1$. Pick $\alpha<\delta$ and $\rho<\lambda$ such that  $\eta=\lambda\alpha+\rho$. As before, enumerate all the nice $\Por_\eta$-names of $\dot J_\eta$-unbounded functions in $\Baire$ by $\set{\dot h^\eta_{\xi'}}{\xi'<\lambda}$. For convenience, we also denote $\dot h^\alpha_{\rho,\xi'} := \dot h^\eta_{\xi'}$ and $\dot \Scal^\alpha_{\rho,\xi'} := \Scal^\eta_{\xi'}$. 

Let us define $\dot\J_\zeta$ as a~$\Por_{\zeta}$-name of the ideal generated by each family as in the following by cases:
\begin{align*}
 \dot J_\eta\cup\bigset{\set{n<\omega}{x(n)\in \dot c_{\eta}(n)}}{x\in\Baire\cap V_\eta} & \text{ when $\rho=2\varepsilon$ and $\Por_\eta\Vdash \dot h^\alpha_j(\varepsilon)$ is $\dot J_\eta$-unbounded,}\\
 \dot J_\eta & \text{ when $\rho=2\varepsilon$ and $\Por_\eta\nVdash \dot h^\alpha_j(\varepsilon)$ is $\dot J_\eta$-unbounded,}\\
 \dot\J_\eta \cup\bigset{\set{i<\omega}{\dot c_\eta(i)\in S(i)}}{S\in\dot \calS^\alpha_{j(\varepsilon)}} & \text{ when $\rho=2\varepsilon+1$.}
\end{align*}
In the last case, $\dot c_\eta$~is a~Cohen real in $\Baire$ over $V_\eta$ (added by $\Por_{\eta+1}$), while in the first case $\dot c_\eta$~is a~Cohen real in $\Scal(\omega,\dot h^\alpha_{j(\varepsilon)})$.  
Thanks to \Cref{cohen-slt_h} and~\ref{clm:cohenev}, we obtain that the families in the cases above have the FUP.

Let $\dot\J:= \dot\J_{\pi}$. It can be proved as in~\ref{seq:sn-a2} that $\Vdash_\pi\slalomt^\perp(\dot h,\dot J)=\dsla{\dot h,\dot \J}=\cf(\pi)$ for any $\Por_\pi$-name $\dot h$ of a~$\dot J$\nobreakdash-{\hskip0pt}unbounded~function. On the other hand, in $V_\pi$, every $\calS$ with the FUPC of size~${<}\theta$ is $\dot\J[G_\pi]$-evaded by some Cohen real, hence $\theta\leq\sla{*,\dot\J[G_\pi]}$. It is clear that $\dot J[G_\pi]$ is a~maximal ideal.

\ref{seq:sn-c}: Apply~\ref{seq:sn-b} to $\theta:=\cf(\pi)$.
\end{proof}

As a~consequence of the foregoing result, we derive:

\begin{corollary}\label{FS:sn}
Let $\pi$ be a~limit ordinal of uncountable cofinality, and let\/ $\Por_\pi=\seqn{\Por_\xi,\Qnm_\xi}{\xi<\pi}$ be a~FS iteration of non-trivial\/ $\cf(\pi)$-cc posets. Let $\lambda:=|\pi|$. Then\/ $\Por_\pi$ satisfies\/ \ref{seq:sn-a} of \Cref{seq:sn}, and also\/ \ref{seq:sn-a2}--\ref{seq:sn-c} when $\lambda$ divides $\pi$ and\/ $\Vdash_\xi |\Qnm_\xi|\leq\lambda$ for all $\xi<\pi$.   
\end{corollary}

\begin{proof}
    The sequence $\seqn{\Por_{\omega\alpha}}{\alpha\leq\pi}$ is as required since FS iterations of non-trivial posets add Cohen reals at limit steps.
\end{proof}

Let $I$ be a~set. Denote by $\Cor_I$ be the poset that adds Cohen reals indexed by $I$. Recall that $\bb_\J=\bb$ and $\dd_\J=\dd$ when $J$ is a meager ideal on $\omega$. On the other hand, M.~Canjar~\cite{Canjar2} has shown that in~$\cohen{\lambda}$, if $\mu\leq\lambda$ is regular then there is a maximal ideal~$I_\mu$ such that $\bb_{I_\mu}=\mu$. This result is extended as follows.


\begin{theorem}\label{aplcohen}
Let $\kappa\geq \aleph_1$ be regular and $\lambda$ an infinite cardinal such that $\lambda^{<\kappa} = \lambda$. Then\/ $\Cor_\lambda$ forces\/ $\nonm = \aleph_1$, $\covm = \cfrak =\lambda$ and that, for any regular\/ $\aleph_1\leq\kappa_1\leq\kappa_2\leq \kappa$:
\begin{enumerate}[label=\rm(\alph*)]
    \item\label{aplcohen:a} There is a~maximal ideal~$\J$ on $\omega$ such that\/ $\sla{\star,\J}=\slalomt^\perp(h,J) =\dsla{h,\J}=\kappa_1$ for all $J$-unbounded $h\in\Baire$ (see~\Cref{Cohen_var_val:b}).

    \item\label{aplcohen:b} There is an~ideal~$\J$ on $\omega$ such that\/ $\sla{\star,\J}=\slalomt^\perp(h,J)=\sla{h,\J}=\kappa_1\leq\dsla{\star,\J}=\slalome^\perp(h,J)=\dsla{h,\J}=\kappa_2$ for any $h\in \Baire$ such that\/ $\lim^J h =\infty$  (see~\Cref{Cohen_var_val:b}).
\end{enumerate}
By weakening the assumption $\lambda^{<\kappa}=\lambda$ to $\kappa\leq\lambda$, we can force the above by removing\/ $\slalome(\star,J)$ and\/ $\slalomt(\star,J)$.
In particular, it is consistent that\/ $\bb<\bb_\J<\dd_\J<\dd$ for some ideal $J$ on $\omega$. 
\end{theorem}

\begin{proof}
It is well-known that $\Cor_\lambda$  forces $\bb=\nonm=\aleph_1$, and $\dd=\covm=\cfrak=\lambda$ (see e.g.~\cite{CM22}).

\ref{aplcohen:a}: 
For any regular $\aleph_1\leq \kappa_1\leq\kappa$, since $\Cor_\lambda \cong \Cor_{\lambda\kappa_1}$ and $\Cor_{\lambda\kappa_1}$ can be obtained by a~FS iteration of~$\Cor$ of length $\lambda\kappa_1$, by \Cref{FS:sn} we obtain that $\Cor_\lambda$ forces that there is a~maximal ideal $\J$ such that $\sla{\star,\J}=\slalomt^\perp(h,J)=\dsla{h,\J}=\kappa_1$ for all $J$-unbounded $h\in\Baire$.

\ref{aplcohen:b}:
By using~\ref{aplcohen:a}, in $V^{\Cor_\lambda}$, there are maximal ideals $\J_1, \J_2$ such that $\sla{\star,\J_1}=\slalomt^\perp(h_1,J_1)=\dsla{h_1,\J_1}=\kappa_1$ and $\dsla{\star,\J_2}=\slalomt^\perp(h_2,J_2)=\dsla{h_2,\J_2}=\kappa_2$ for any $J_e$-unbounded $h_e\in\Baire$ and $e\in\{1,2\}$. By letting $\J:=\J_1\oplus\J_2$ (with suitable modifications of $\J_1$ and $\J_2$), 
any $h\in\Baire$ can be written as $h=h_1\oplus h_2$, and $\lim^J h=\infty$ iff $\lim^{J_1} h_1 = \lim^{J_2} h_2 =\infty$. We obtain, 
by \Cref{L3.5} and \Cref{L6.3}, that 
\begin{align*}
    \slalome(\star,J) = &\min\{\slalome(\star,J_1),\slalome(\star,J_2)\} = \kappa_1, & \slalomt(\star,J) = &\max\{\slalomt(\star,J_1),\slalomt(\star,J_2)\}= \kappa_2,\\
    \slalome(h,J)  = & \min\{\slalome(h_1,J_1),\slalome(h_2,J_2)\} = \kappa_1, & \slalome^\perp(h,J)  = & \max\{\slalome^\perp(h_1,J_1),\slalome^\perp(h_2,J_2)\} = \kappa_2,\\
    \slalomt^\perp(h,J)  = &\min\{\slalomt^\perp(h_1,J_1),\slalomt^\perp(h_2,J_2)\} = \kappa_1, & \slalomt(h,J)  = & \max\{\slalomt(h_1,J_1),\slalomt(h_2,J_2)\} = \kappa_2.\qedhere
\end{align*} 
\end{proof}

\subsection{Applications of ccc models}\label{sec:ccc}

We present several ccc forcing constructions to force constellations of~\Cref{BasicDia} by application of~\Cref{seq:sn} and~\Cref{FS:sn}. We skip details in the proofs when they can be found in the cited references. 

\begin{theorem}[cf.~{\cite[Sec.~6]{GKMSsplit}}]\label{gkms}
Let $\lambda_0\leq \lambda_1\leq \lambda_2\leq\lambda_3$ be uncountable regular cardinals and let $\lambda_4$~be a~cardinal such that $\lambda_3\leq\lambda_4=\lambda_4^{{<}\lambda_3}$. Also assume that either $\lambda_1 = \lambda_2$, or $\lambda_2$ is $\aleph_1$-inaccessible\footnote{This means that $\mu^{\aleph_0}<\lambda_2$ for any cardinal $\mu<\lambda_2$.} and\/ $2^{<\lambda_1}<\lambda_2$.
Then there is some poset forcing that there is a~maximal ideal $\J$ satisfying the constellation in~\Cref{gkms:spl} and\/ $\slalomt^\perp(h,J)=\lambda_3$ for any $J$-unbounded $h\in\Baire$.
\begin{figure}[ht]
\begin{center}
\begin{tikzpicture}[scale=0.8,every node/.style={scale=0.8}]
\node (ale) at (-7, -3.5) {$\aleph_1$};
\node (a) at (-5, -3.5) {$\pp$};
\node (as) at (-5, -1) {$\sla{\I,\fin}$};
\node (b) at (-5, 1.5) {$\bb$};
\node (ba) at (-5, 4) {$\nonm$};
\node (aa) at (-2, -3.5) {$\slamg{\J}$};
\node (aas) at (-2, -1) {$\sla{\I,\J}$};
\node (bb) at (-2, 1.5) {$\bb_\J$};
\node (bba) at (-2, 4) {$\sla{h,\J}$};
\node (c) at (1, -3.5) {$\dslaml{\J}$};
\node (cs) at (1, -1) {$\dsla{\I,\J}$};
\node (f) at (1, 1.5) {$\dd_\J$};
\node (csa) at (1, 4) {$\dsla{h,\J}$};
\node (xpf) at (4, 1.5) {$\dd$};
\node (xpc) at (4, -3.5) {$\covm$};
\node (xpcs) at (4, -1) {$\dslago{\I}$};
\node (xpfa) at (4, 4) {$\cofn$};
\node (cont) at (6, 4) {$\cc$};
\foreach \from/\to in {aas/bb,aas/cs,cs/f,aas/cs} \draw [line width=.15cm,
white] (\from) -- (\to);
\foreach \from/\to in {ale/a,a/as, aa/aas, c/cs, b/bb, a/aa, aa/c, bb/f, as/b, aas/bb, cs/f, as/aas, aas/cs,f/xpf,cs/xpcs,c/xpc,xpcs/xpf,xpc/xpcs, b/ba,bb/bba,f/csa,xpf/xpfa,ba/bba,bba/csa,csa/xpfa,xpfa/cont} \draw [->] (\from) -- (\to);

\draw[color=blue,line width=.05cm] (-6,-4)--(-6,5);
\draw[color=blue,line width=.05cm] (-6,2.7)--(-3.7,2.7);
\draw[color=blue,dashed,line width=.05cm] (-6,0.25)--(-3.7,0.25);
\draw[color=blue,dashed,line width=.05cm] (-6,-1.75)--(-3.7,-1.75);
\draw[color=blue,line width=.05cm] (-3.7,-4)--(-3.7,5);
\draw[color=blue,line width=.05cm] (2.7,-4)--(2.7,5);

\draw[circle, fill=cadmiumorange,color=cadmiumorange] (-0.4,0.2) circle (0.4);
\draw[circle, fill=cadmiumorange,color=cadmiumorange] (-4.3,3.3) circle (0.4);
\draw[circle, fill=cadmiumorange,color=cadmiumorange] (-4.3,0.9) circle (0.4);
\draw[circle, fill=cadmiumorange,color=cadmiumorange] (5,2.7) circle (0.4);
\draw[circle,fill=cadmiumorange,color=cadmiumorange] (-4.3,-2.4) circle (0.4);
\draw[circle,fill=cadmiumorange,color=yellow] (-4.3,-0.3) circle (0.4);
\node at (-4.3,-2.4) {$\lambda_0$};
\node at (-0.4,0.2) {$\lambda_3$};
\node at (-4.3,3.3) {$\lambda_2$};
\node at (-4.3,0.9) {$\lambda_1$};
\node at (-4.3,-0.3) {$\textbf{?}$};
\node at (5,2.7) {$\lambda_4$};
\end{tikzpicture}
\end{center}
    \caption{Constellation forced in~\Cref{gkms} for any $J$-unbounded $h$. 
    As indicated in \autoref{BasicDia}, an arrow represents that $\thzfc$ proves ``$\leq$"; the thicker lines divide the diagram in regions where the cardinal characteristics take a value $\lambda_i$ (for example, the cardinal characteristics in the central region take the value $\lambda_3$); 
    The dotted lines indicate regions where the forced value is unclear, in this case that $\lambda_0\leq\slalome(I,\Fin)\leq\lambda_1$ and that the exact value is unclear, although $\slalome(I,\Fin) = \lambda_0$ when $I$ is a maximal ideal.}
\label{gkms:spl}
\end{figure}
\end{theorem}

\begin{proof}
Construct a~finite support iteration $\Por_\pi=\la\Por_\alpha,\Qnm_\alpha\colon\alpha<\lambda_4\ra$ of length $\pi:=\lambda_4\lambda_3$ as in~\cite[Subsec.~6B]{GKMSsplit}, with book-keeping arguments, of the following ccc posets:
\begin{itemize}
    \item  restrictions of $\Eor$ (the standard $\sigma$-centered poset adding an eventually different real) of size ${<}\lambda_2$;
    \item all $\sigma$-centered posets of size~${<}\lambda_0$; and 
    \item all $\sigma$-centered subposets of Hechler forcing of size~${<}\lambda_1$.
\end{itemize}
Then, $\Por_\pi$ forces $\pfrak=\sfrak=\lambda_0$, $\bfrak=\lambda_1$, $\nonm=\lambda_2$ and 
$\covm=\cfrak=\lambda_4$. On the other hand, since the cofinality of $\pi$ is $\lambda_3$, by using~\Cref{FS:sn},  $\Por_\pi$ forces that there is a~maximal ideal $\J$ such that $\sla{\star,\J}=\slalomt^\perp(h,J) =\dsla{h_1,\J}=\lambda_3$ for any $J$-unbounded $h\in\Baire$. On the other hand, by \Cref{S4_eq}~\ref{S4_eq:01}, we know that $\slalome(I,\Fin) = \min\{\covst(I),\bfrak\}$ for any ideal $I$ on $\omega$. In particular, when $I$ is maximal, $\covst(I)\leq \sfrak$ by~\cite{BreShe99}, so $\Por_\pi$ forces that $\slalome(I,\Fin)=\lambda_0$ for any maximal ideal $I$ on $\omega$.
 \end{proof}

\begin{figure}[ht]
\begin{center}
\begin{tikzpicture}[scale=0.8,every node/.style={scale=0.8}]
\node (ale) at (-7, -3.5) {$\aleph_1$};
\node (a) at (-5, -3.5) {$\pp$};
\node (as) at (-5, -1) {$\sla{\I,\fin}$};
\node (b) at (-5, 1.5) {$\bb$};
\node (ba) at (-5, 4) {$\nonm$};
\node (aa) at (-2, -3.5) {$\slamg{\J}$};
\node (aas) at (-2, -1) {$\sla{\I,\J}$};
\node (bb) at (-2, 1.5) {$\bb_\J$};
\node (bba) at (-2, 4) {$\sla{h,\J}$};
\node (c) at (1, -3.5) {$\dslaml{\J}$};
\node (cs) at (1, -1) {$\dsla{\I,\J}$};
\node (f) at (1, 1.5) {$\dd_\J$};
\node (csa) at (1, 4) {$\dsla{h,\J}$};
\node (xpf) at (4, 1.5) {$\dd$};
\node (xpc) at (4, -3.5) {$\covm$};
\node (xpcs) at (4, -1) {$\dslago{\I}$};
\node (xpfa) at (4, 4) {$\cofn$};
\node (cont) at (6, 4) {$\cc$};
\foreach \from/\to in {aas/bb,aas/cs,cs/f,aas/cs} \draw [line width=.15cm,
white] (\from) -- (\to);
\foreach \from/\to in {ale/a,a/as, aa/aas, c/cs, b/bb, a/aa, aa/c, bb/f, as/b, aas/bb, cs/f, as/aas, aas/cs,f/xpf,cs/xpcs,c/xpc,xpcs/xpf,xpc/xpcs, b/ba,bb/bba,f/csa,xpf/xpfa,ba/bba,bba/csa,csa/xpfa,xpfa/cont} \draw [->] (\from) -- (\to);

\draw[color=blue,line width=.05cm] (-7.4,2.7)--(6,2.7);
\draw[color=blue,dashed,line width=.05cm] (-7.4,0.25)--(6,0.25);
\draw[color=blue,dashed,line width=.05cm] (-7.4,-2.25)--(6,-2.25);

\draw[circle, fill=cadmiumorange,color=cadmiumorange] (-0.5,2.1) circle (0.45);
\draw[circle, fill=cadmiumorange,color=cadmiumorange] (-0.5,3.4) circle (0.4);
\draw[circle, fill=cadmiumorange,color=cadmiumorange] (-0.5,-2.9) circle (0.4);
\draw[circle, fill=cadmiumorange,color=yellow] (-0.5,-0.4) circle (0.4);
\node at (-0.5,2.1) {$\cf(\pi)$};
\node at (-0.5,3.4) {$|\pi|$};
\node at (-0.5,-2.9) {$\aleph_1$};
\node at (-0.5,-0.4) {$\textbf{?}$};
\end{tikzpicture}
\end{center}
\caption{Constellation forced in~\Cref{Thm:mixBD}. 
The cardinals between the dotted lines lie between $\aleph_1$ and $\cf(\pi)$, but their exact values are unclear. However, $\slalome(I,\Fin) = \aleph_1$ for any maximal ideal $I$ on $\omega$.
}
\label{HefollowsRa}
\end{figure}

\begin{theorem}\label{Thm:mixBD}
Let $\pi$ be an ordinal of uncountable cofinality, $\lambda:=|\pi|$ and assume that\/ $\lambda^{\aleph_0}=\lambda$. Then, the FS iteration of Hechler forcing of length $\pi$ followed by the random algebra adding\/ $\lambda$-many random reals forces the constellation of~\Cref{HefollowsRa} and\/ $\slalome^\perp(h,J)=\aleph_1$ for any $h\geq^*1$ and any ideal $J$ on $\omega$.
\end{theorem}

\begin{proof}
It is well-known that the first iteration of the Hechler poset forces $\pfrak=\addn=\covn=\aleph_1$, $\addm=\cofm=\cf(\pi)$ and $\nonn=\cfrak=|\pi|$ (see e.g.~\cite[Thm.~5]{mejiamatrix}). 
After further adding $\lambda$-many random reals by using the random algebra, the generic extension satisfies $\non\Nwf=\aleph_1$, $\bfrak=\dfrak=\cf(\pi)$, and $\cov\Nwf=\cfrak=|\pi|$ (details can be found in~\cite[Sec.~5]{GKMScont}). As a consequence, $\slalome^\perp(h,J)=\slalomt(\star, J)=\aleph_1$ and $\slalome(h, J)=\cfrak$ for any ideal $J$ on $\omega$ and
$h\in\Baire$ with $h\geq^* 1$. On the other hand, when $I$ is a maximal ideal, $\covst{I}\leq\sfrak\leq\non{\Nwf}=\aleph_1$ by~\cite{BreShe99}, so $\slalome(I,\Fin) = \min\{\covst{I},\bfrak\} =\aleph_1$.
\end{proof}

\subsection{Several values}\label{subsec:seval}

Using the method of coherent systems from~\cite{mejvert}, we force constellations of \Cref{BasicDia} with many different values of cardinal invariants parametrized with ideals. The method is reviewed as follows.

\begin{definition}[cf.~{\cite[Def.~3.2]{FFMM}}]\label{Defcohsys}
   A \emph{simple coherent system (of FS iterations)} $\sbf$ is composed of the following objects:
   \begin{enumerate}[label=\rm(\Roman*)]
     \item a~partially ordered set $I^\sbf$ with a~maximum $i^*$, an ordinal $\pi^\sbf$,
     \item a~function $\Delta^\sbf\colon\pi^\sbf \menos\{0\}\to I^\sbf$,
     \item\label{csIII} for each $i\in I^\sbf$, a~FS iteration $\Por^\sbf_{i,\pi^\sbf}=\la\Por^\sbf_{i,\xi},\Qnm^\sbf_{i,\xi}:\xi<\pi^\sbf\ra$ such that,
     \begin{enumerate}[label=\rm(\roman*)]
         \item $\Por^\sbf_{i,1}\subsetdot \Por^\sbf_{j,1}$ whenever $i\leq j$ in $I^\sbf$, and
         \item for any $0<\xi<\pi^\sbf$, there is some $\Por^\sbf_{\Delta^\sbf(\xi),\xi}$-name of a~poset $\Qnm^\sbf_\xi$, with a~maximum element $1_\xi$ living in $V$ (not just a name) such that, for any $i\in I^\sbf$,
     \[
     \Qnm^\sbf_{i,\xi}=
     \begin{cases}
     \Qnm^\sbf_\xi,&\text{if $i\geq\Delta^\sbf(\xi)$,}\\
     \{1_\xi\},&\text{otherwise.}
     \end{cases}
     \]     
     \end{enumerate}
   \end{enumerate}
   According to this notation, $\Por^\sbf_{i,0}$ is the trivial poset and $\Por^\sbf_{i,1}=\Qnm^\sbf_{i,0}$. We often refer to $\la\Por^\sbf_{i,1}:i\in I^\sbf\ra$ as the \emph{base of the coherent system $\sbf$.} 
   Note that~\ref{csIII} implies that $\Por^\sbf_{i,\xi}\subsetdot\Por^\sbf_{j,\xi}$ whenever $i\leq j$ in $I^\sbf$ and $\xi\leq\pi^\sbf$ (see details in~\cite{FFMM}).

   For $j\in I^\sbf$ and $\eta\leq\pi^\sbf$ we write $V^\sbf_{j,\eta}$ for the $\Por^\sbf_{j,\eta}$-generic extensions. Concretely, when $G$ is $\Por^\sbf_{j,\eta}$-generic over~$V$, $V^\sbf_{j,\eta}:=V[G]$ and $V^\sbf_{i,\xi}:=V[\Por^\sbf_{i,\xi}\cap G]$ for all $i\leq j$ in $I^\sbf$ and $\xi\leq\eta$. Note that $V^\sbf_{i,\xi}\subseteq V^\sbf_{j,\eta}$ and $V^\sbf_{i,0}=V$ (see \Cref{Figcohsys}).

   \begin{figure}[ht]
   \begin{center}
     \includegraphics[scale=0.7]{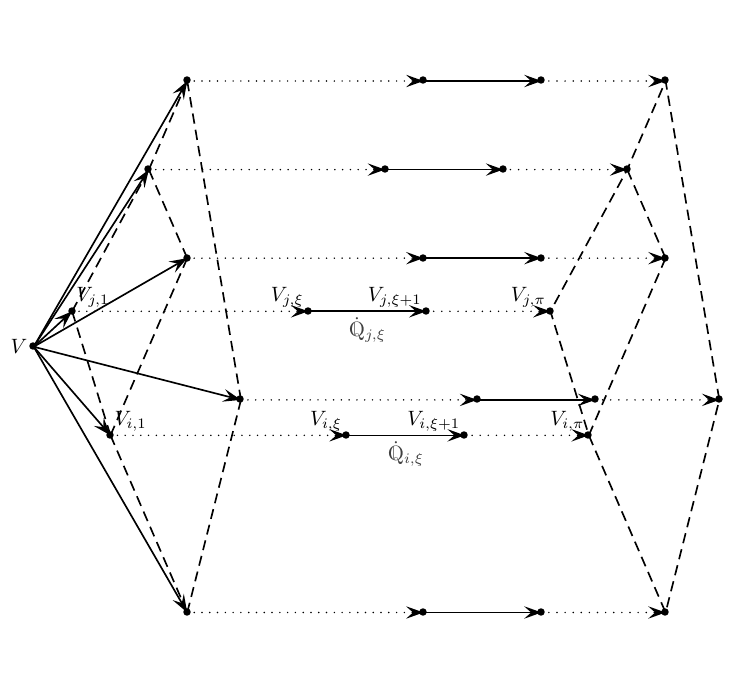}
     \caption{Coherent system of FS iterations. FS iterations of length $\pi$ for each $i\in I$ produce a sequence of generic extensions $\seqn{V_{i,\xi}}{i\in I,\ \xi\leq \pi}$, increasing on $i$ and $\xi$, starting with $V_{i,0}=V$ (the ground model). Each arrow denotes ``$\subseteq$" and the figures in dashed lines represent the `shape' of the partial order $\la I,{\leq}\ra$, i.e.\ $i\leq j$ in $I$ implies $V_{i,\xi}\subseteq V_{j,\xi}$ for all $\xi\leq \pi$.}
     \label{Figcohsys}
   \end{center}
   \end{figure}

   We say that the coherent system $\sbf$ has the \emph{$\theta$-cc} if, additionally, $\Por^\sbf_{i,\xi}$ forces that $\Qnm^\sbf_{i,\xi}$ has the $\theta$-cc for each $i\in I^\sbf$ and $\xi<\pi^\sbf$. This implies that $\Por^\sbf_{i,\xi}$ has the $\theta$-cc for all $i\in I^\sbf$ and $\xi\leq\pi^\sbf$.

   For a~coherent system $\sbf$ and a~set $J\subseteq I^\sbf$, $\sbf|_J$ denotes the coherent system with $I^{\sbf|_J}=J$, $\pi^{\sbf|_J}=\pi^\sbf$ and the FS iterations corresponding to~\ref{csIII} defined as for $\sbf$;\footnote{It could happen that $\Delta(\xi)\notin J$ for some $\xi$, but this is not a problem because, in this case, all iterands at $\xi$ are $\{1_\xi\}$.} if $\eta\leq\pi^\sbf$, $\sbf\frestr\eta$ denotes the coherent system with $I^{\sbf\upharpoonright\eta}=I^\sbf$, $\pi^{\sbf\upharpoonright\eta}=\eta$ and the iterations for~\ref{csIII} defined up to $\eta$ as for $\sbf$. Also, for $i_0\in I$, denote $J_{<i_0}:=\set{i\in J}{i<i_0}$. The set $J_{\leq i_0}$ is defined similarly.

   In particular, the upper indices $\sbf$ are omitted when there is
   no risk of ambiguity.
\end{definition}

Recall that, whenever $\la I,\leq\ra$ is a directed partial order (i.e.,\ for all $i,j\in I$ there is some $j'\in I$ above them) and $\seqn{\Por_i}{i\in I}$ is a sequence of posets such that $\Por_i\subsetdot \Por_j$ whenever $i\leq j$ in $I$, \emph{the direct limit of $\seqn{\Por_i}{i\in I}$} is $\Por:=\bigcup_{i\in I}\Por_i$.

\begin{lemma}[cf.~{\cite[Lem.~2.7]{mejvert}}]\label{smallsupprestr}
   Let $\theta$ be an uncountable regular cardinal 
   and let\/ $\sbf$ be a\/ $\theta$-cc simple coherent system. 
   Assume: 
    \begin{enumerate}[label=\rm(\roman*)] 
      \item $i^*\notin \ran\Delta$, $\bfrak(I^\sbf_{<i^*})\geq \theta$ and
      \item whenever $\pi>0$, $\Por_{i^*,1}$ is the direct limit of\/ $\seqn{\Por_{i,1}}{i\in  I^\sbf_{<i^*}}$.
    \end{enumerate}
   Then, for every $\xi\leq\pi$:
   \begin{enumerate}[label=\rm(\alph*)]
      \item\label{suppr1} $\Por_{i^*,\xi}$ is the direct limit of\/ $\seqn{\Por_{i,\xi}}{i\in I^\sbf_{<i^*}}$ and
      \item\label{suppr2} for any\/ $\Por_{i^*,\xi}$-name of a~function $\dot{x}$ with domain $\gamma<\theta$ into\/ $\bigcup_{i\in I^\sbf_{<i^*}}V_{i,\xi}$, there is some $i\in I^\sbf_{<i^*}$ such that $\dot{x}$ is (forced to be equal to) a\/ $\Por_{i,\xi}$-name.
   \end{enumerate}
\end{lemma}
\begin{proof}
    We first prove~\ref{suppr1}${}\Rightarrow{}$\ref{suppr2}. Let $\dot x$ be a $\Por_{i^*,\xi}$-name as in~\ref{suppr2}. For each $\alpha<\gamma$, there is some maximal antichain $A_\alpha\subseteq \Por_{i^*,\xi}$ such that, for each $p\in A_\alpha$, there are some $i_{\alpha,p}\in I^\sbf_{<i^*}$ and some $\Por_{i_{\alpha,p},\xi}$-name $\dot x_{\alpha,p}$ such that $p\Vdash \dot x(\alpha) = \dot x_{\alpha,p}$. Since $\Por_{i^*,\xi}$ is $\theta$-cc, $|A_\alpha|<\theta$, so $A:=\bigcup_{\alpha<\gamma}A_\alpha$ has size ${<}\theta$ because $\theta$ is regular. Since $\bfrak(I^\sbf_{<i^*})\geq\theta$, by~\ref{suppr1} there is some $i\in I^\sbf_{<i^*}$ such that $A\subseteq \Por_{i,\xi}$ and $i_{\alpha,p} \leq i$ for all $\alpha<\gamma$ and $p\in A_\alpha$. Then, each $A_\alpha$ is a maximal antichain in $\Por_{i,\xi}$ and each $\dot x_{\alpha,p}$ is a $\Por_{i,\xi}$-name. So we can define a~$\Por_{i,\xi}$-name $\dot y$ of a~function with domain $\gamma$ such that, for any $\alpha<\gamma$ and $p\in A_\alpha$, $p\Vdash_{\Por_{i,\xi}} \dot y(\alpha) = \dot x_{\alpha,p}$. Therefore, $\Vdash_{\Por_{i^{\ast},\xi}} \dot x = \dot y$. 

    We now prove~\ref{suppr1} by induction on $\xi\leq\pi$. The case $\xi\in\{0,1\}$ is trivial. For the successor step $\xi\to\xi+1$ with $\xi\geq 1$, assume that~\ref{suppr1} holds for $\xi$, so~\ref{suppr2} is implied for $\xi$. Let $p\in\Por_{i^*,\xi+1}$. Without loss of generality, assume that $\xi\in\dom p$, so $p\frestr\xi\in \Por_{i^*,\xi}$ and $\Vdash_{\Por_{i^*,\xi}} p(\xi)\in\Qnm_\xi =\Qnm_{\Delta(\xi),\xi} \subseteq V_{\Delta(\xi),\xi}$. By~\ref{suppr1} and~\ref{suppr2} for $\xi$, we can find $i_0,i_1\in I^\sbf_{<i^*}$ such that $p\frestr\xi \in \Por_{i_0,\xi}$ and $p(\xi)$ is a $\Por_{i_1,\xi}$-name (for the latter, consider a name for the function with domain $1$ sending $0$ to $p(\xi)$). Since $\bfrak(I^\sbf_{<i^*})$ is infinite, find some $i\in I^\sbf_{<i^*}$ above $\Delta(\xi)$, $i_0$ and~$i_1$. Then, $p\in\Por_{i,\xi}$.

    For the limit step, assume that $\xi$ is limit. If $p\in\Por_{i^*,\xi}$ then $p\in\Por_{i^*,\xi_0}$ for some $\xi_0<\xi$, hence, by induction hypothesis, $p\in\Por_{i,\xi_0}$ for some $i\in I^\sbf_{<i^*}$. Clearly, $p\in\Por_{i,\xi}$.
\end{proof}

\begin{theorem}\label{mainpresunb} 
   Let $\theta$ be an uncountable regular cardinal 
   and let\/ $\sbf$ be a\/ $\theta$-cc simple coherent system. Assume:
    \begin{enumerate}[label=\rm(\roman*)]
      \item $i_0<i_1$ in $I^\sbf$ and $0<\pi$, 
      \item\label{unbi} $I^0\subseteq I^\sbf$ and either $i_0\in I^0$, or\/ $\bfrak(I^0_{<i_0})\geq\theta$ and $i_0\notin\ran\Delta$,
      \item $\Por_{i_0,1}$ is the direct limit of\/ $\seqn{\Por_{i,1}}{i\in I^0_{\leq i_0}}$,\footnote{This is trivial when $i_0\in I^0$.} and
      \item\label{unbii} $\Por_{i_1,1}$ adds a~real $\dot{c}$ such that, for any $i\in I^0_{\leq i_0}$, $\Por_{i_1,1}$ forces that $\dot{c}$ is Cohen over $V_{i,1}$.
   \end{enumerate}
  Then, $\Por_{i_1,\pi}$ forces that $\dot{c}$ is Cohen over $V_{i_0,\pi}$.
\end{theorem}

To proceed with the proof of the theorem, we need to review some facts about the preservation of $\Rbf$-unbounded reals (property defined below) when $\Rbf = \la X,Y,\sqsubset\ra$ is a~\emph{Polish relational system}. We use~\cite[\S4.3]{CM19} as a~reference. We do not define Polish relational systems since we do not require the details, but we just mention that $X$ is a~perfect Polish space, $Y$ is an analytic subset of some Polish space, and $\sqsubset$ is a~very well-defined relation (concretely, $F_\sigma$), which allow many absoluteness arguments. We are interested in using a~Polish relational system $\Mbf$ that is Tukey equivalent with $\Cbf_\Mwf$. There are many examples, one is $\Mbf = \la {}^\omega2, {}^\omega2\times\Ibb,\match\ra$ where $\Ibb$ denotes the collection of all interval partitions $I=\seqn{I_n}{n<\omega}$ of $\omega$ and
\[
x\match(y,I) \text{ iff }(\forall^\infty n<\omega) (\exists \ell\in I_n)\  x(\ell) \neq y(\ell).
\]
A pair $(y,I)$ is typically known as a~\emph{matching real}. The proof of $\Mbf\eqT \Cbf_\Mwf$ can be found in, e.g.\ \cite{blass}.

Given a~relational system $\Rbf = \la X,Y,\sqsubset\ra$ and a~transitive model $N$ of ZFC, a~real $c\in X$ is \emph{$\Rbf$-unbounded over $N$} if $c\not\sqsubset y$ for all $y\in Y^N$. Note that $c\in{}^\omega2$ is $\Mbf$-unbounded over $N$ iff $c$ is a~Cohen real over $N$. 

We say that $\sbf$ is a~\emph{simple coherent pair} if it is a~simple coherent system with $I^\sbf = \{i_0,i_1\}$ and $i_0<i_1$. We use the following results about the preservation of $\Rbf$-unbounded reals for coherent pairs, where $\Rbf=\la X,Y,\sqsubset\ra$ is a~Polish relational system (in particular, $\Mbf$).

\begin{lemma}[{\cite[Lem.~4.29]{CM19}}]\label{CMsucc}
    Let $M\subseteq N$ be transitive models of $\mathrm{ZFC}$ such that\/ $\Rbf$ can be defined in~$M$ (and hence, in $N$). 
    Assume that $c\in X^N$ is\/ $\Rbf$-unbounded over $M$. If\/ $\Por\in M$ is a~poset and $G$ is\/ $\Por$-generic over $N$, then $c$~is\/ $\Rbf$-unbounded over $M[G]$.\qed
\end{lemma}

\begin{lemma}[{\cite[Cor.~4.31]{CM19}}]\label{CMlim}
    Let\/ $\sbf$ be a~simple coherent pair of length a~limit ordinal $\pi$, wlog $I^\sbf = \{0,1\}$. Assume that $\dot c$ is a~$\Por_{1,1}$-name of a~member of $X$ such that, for any $\xi<\pi$, $\Por_{1,\xi}$ forces that $\dot c$ is\/ $\Rbf$-unbounded over $V_{0,\xi}$. Then\/ $\Por_{1,\pi}$ forces that $\dot c$ is\/ $\Rbf$-unbounded over $V_{0,\pi}$.\qed
\end{lemma}

\begin{proof}[Proof of \Cref{mainpresunb}]
    Let $\dot y$ and $\dot I$ be $\Por_{i_0,\pi}$-names of members of ${}^\omega2$ and $\Ibb$, respectively. It suffices to show that $\Por_{i_1,\pi}$ forces $\dot c \match (\dot y,\dot I)$. 
    By hypothesis, in the case $i_0\notin I^0$, we can apply \Cref{smallsupprestr} to $\sbf|_{I^0_{\leq i_0}}$ and get some $i\in I^0_{<i_0}$ such that $\dot y$ and $\dot I$ are $\Por_{i,\pi}$-names. When $i_0\in I^0$, set $i:=i_0$. In any case, $i\in I^0_{\leq i_0}$.
    
    It is enough to show, by induction on $1\leq \eta\leq \pi$, that $\Por_{i_1,\eta}$ forces that $\dot c$ is $\Mbf$-unbounded (i.e., Cohen) over~$V_{i,\eta}$. The case $\eta=1$ is clear by~\ref{unbii}, and the limit step is immediate from \Cref{CMlim}. So we deal with the successor step $\eta=\xi+1>1$. We consider two cases: if $\Delta(\xi)\leq i$ then $\Qnm_{i,\xi}= \Qnm_{i_1,\xi} = \Qnm_\xi$, so we can apply \Cref{CMsucc}; but if $\Delta(\xi)\nleq i$ then $\Qnm_{i,\xi} =\{0\}$ (the trivial poset), so $V_{i,\xi+1} = V_{i,\xi}$ and $\Por_{i_1,\xi}$ already forces that $\dot c$ is $\Mbf$-unbounded over $V_{i,\xi}$ (and so does $\Por_{i_1,\xi+1}$). 
\end{proof}

\begin{theorem}\label{Thm:mvert}
    Let $\theta$ be an uncountable regular cardinal and let\/ $\sbf$ be a\/ $\theta$-cc simple coherent system. Assume: 
    \begin{enumerate}[label=\rm(\roman*)]        
        \item $\ran\Delta\subseteq I^0\subseteq I^\sbf_{<i^*}$ and, for every $i\in I^\sbf\menos I^0$, $\bfrak(I^0_{<i})\geq\theta$, and

        \item For every $i_0\in I^\sbf\menos I^0$, $\Por_{i_0,1}$ is the direct limit of\/ $\seqn{\Por_{i,1}}{i\in I^0_{<i_0}}$.
    \end{enumerate}
     Then: 
     \begin{enumerate}[label = \normalfont (\alph*)]
         \item\label{cohlim:0} For every $i_0\in I^\sbf$, $\Por_{i_0,\pi}$ is the direct limit of\/ $\seqn{\Por_{i,\pi}}{i\in I^0_{\leq i_0}}$.
         
         \item\label{cohlim:a} If\/$\seqn{i_\zeta}{\zeta<\delta}$ is an increasing sequence in $I^\sbf_{<i^*}$ such that, for any $i\in I^0$, there is some $\zeta<\delta$ such that $i\leq i_\zeta$, then          
         $\Por_{i^*,\pi}$ is the direct limit of\/ $\seqn{\Por_{i_\zeta,\pi}}{\zeta<\delta}$.

         \item\label{cohlim:b} Further assume that $\lambda:=|I^\sbf|$, $\nu$ is a~regular cardinal, $\theta\leq\nu\leq \lambda$, $|\Por_{i^*,\pi}|\leq\lambda$, $\lambda^{<\nu}=\lambda$, and there is some strictly increasing sequence\/ $\seqn{i_{\zeta}}{\zeta<\lambda\nu}$ in $I^\sbf_{<i^*}$ as in~\ref{cohlim:a} and such that, for any $\zeta<\lambda\nu$, $\Por_{i_{\zeta+1},1}$~adds a~Cohen real over $V_{i_\zeta,1}$. Then, $\Por_{i^*,\pi}$ forces that there is some maximal ideal $J^*_\nu$ on $\omega$ such that\/ $\slalome(\star,J^*_\nu)=\slalomt^\perp(h,J^*_\nu)=\slalomt(h,J^*_\nu)=\nu$ for any $J^*_\nu$-unbounded $h\in\Baire$. If we remove the assumption $\lambda^{<\nu} = \lambda$, then we can remove\/ $\slalome(\star,J^*_\nu)$ in the result.
     \end{enumerate}
\end{theorem}

\begin{proof}
     \ref{cohlim:0} is a direct consequence of \Cref{smallsupprestr} when $i_0\in I^\sbf\menos I^0$, otherwise it is trivial.~\ref{cohlim:a} follows directly by~\ref{cohlim:0}. 

     To see~\ref{cohlim:b}:      
     By \Cref{mainpresunb}, for any $\zeta<\lambda\nu$, $\Por_{i_{\zeta+1},\pi}$ adds a~Cohen real over $V_{i_\zeta,\pi}$. Then, by~\ref{cohlim:a}, $\Por_{i^*,\pi}$ is the direct limit of $\seqn{\Por_{i_\zeta,\pi}}{\zeta<\lambda\nu}$ and, thus, this sequence satisfies the hypothesis of \Cref{seq:sn}. Hence, $\Por_{i^*,\pi}$ forces that there is some maximal ideal $J^*_\nu$ satisfying $\slalome(\star,J^*_\nu)=\slalomt^\perp(h,J^*_\nu)=\slalomt(h,J^*_\nu) = \nu$ for any $J^*_\nu$-unbounded $h\in\Baire$.
\end{proof}

We now proceed with the applications. We use models established in~\cite{mejvert,BCM} and omit the details that can be found in the references.

\begin{theorem}[cf.~{\cite[Thm.~4.4]{mejvert}}]\label{mejver}
Let $\lambda_0$ be an uncountable regular cardinal, and let $\lambda_3\leq \lambda_4$ be cardinals such that\/ $\cof{[\lambda_3]^{<\lambda_0}}=\lambda_3$ and $\lambda_4=\lambda_4^{<\lambda_0}$. Then there is a~ccc poset forcing\/ $\pfrak=\nonm =\lambda_0$, $\covm = \dfrak =\lambda_3$, $\cofn =\cfrak = \lambda_4$, and:
\begin{enumerate}[label=\normalfont(\alph*)]
    \item\label{mejver:a} For any regular cardinal $\lambda$ such that\/ $\lambda_0\leq\lambda\leq\lambda_3$ and $\lambda_4^{<\lambda}=\lambda_4$, there is some maximal ideal $J$ such that\/ $\slalome(\star,J) = \slalomt^\perp(h,J) = \slalomt(h,J) = \lambda$ for any $J$-unbounded $h\in\Baire$ (see \Cref{figappl:a}). When the assumption $\lambda_4^{<\lambda} = \lambda_4$ is removed, we can remove $\slalome(\star,J)$.
    \item\label{mejver:b} For any regular cardinals $\lambda_1$ and $\lambda_2$ such that $\lambda_0\leq\lambda_1\leq\lambda_2\leq\lambda_3$ and $\lambda_4^{<\lambda_2}=\lambda_4$, there is some ideal~$J'$ satisfying\/ $\slalome(\star,J') = \slalomt^\perp(h,J') = \slalome(h,J') = \lambda_1$ and\/ $\slalomt(\star,J') = \slalome^\perp(h,J') = \slalomt(h,J') = \lambda_2$ for any $h\in\Baire$ such that $\lim^{J'}h =\infty$ (see \Cref{figappl:a}). When the assumption $\lambda_4^{<\lambda_2} = \lambda_4$ is removed, we can remove\/ $\slalome(\star,J')$ and $\slalomt(\star,J')$. 
\end{enumerate}
\begin{figure}[ht]
\centering
\begin{tikzpicture}[scale=0.6,every node/.style={scale=0.6}]
\node (ale) at (-7, -3.5) {$\aleph_1$};
\node (a) at (-5, -3.5) {$\pp$};
\node (as) at (-5, -1) {$\sla{\I,\fin}$};
\node (b) at (-5, 1.5) {$\bb$};
\node (ba) at (-5, 4) {$\nonm$};
\node (aa) at (-2, -3.5) {$\slamg{\J}$};
\node (aas) at (-2, -1) {$\sla{\I,\J}$};
\node (bb) at (-2, 1.5) {$\bb_\J$};
\node (bba) at (-2, 4) {$\sla{h,\J}$};
\node (c) at (1, -3.5) {$\dslaml{\J}$};
\node (cs) at (1, -1) {$\dsla{\I,\J}$};
\node (f) at (1, 1.5) {$\dd_\J$};
\node (csa) at (1, 4) {$\dsla{h,\J}$};
\node (xpf) at (4, 1.5) {$\dd$};
\node (xpc) at (4, -3.5) {$\covm$};
\node (xpcs) at (4, -1) {$\dslago{\I}$};
\node (xpfa) at (4, 4) {$\cofn$};
\node (cont) at (6, 4) {$\cc$};
\foreach \from/\to in {aas/bb,aas/cs,cs/f,aas/cs} \draw [line width=.15cm,
white] (\from) -- (\to);
\foreach \from/\to in {ale/a,a/as, aa/aas, c/cs, b/bb, a/aa, aa/c, bb/f, as/b, aas/bb, cs/f, as/aas, aas/cs,f/xpf,cs/xpcs,c/xpc,xpcs/xpf,xpc/xpcs, b/ba,bb/bba,f/csa,xpf/xpfa,ba/bba,bba/csa,csa/xpfa,xpfa/cont} \draw [->] (\from) -- (\to);

\draw[color=blue,line width=.05cm] (-6,-4)--(-6,5);
\draw[color=blue,line width=.05cm] (-3.5,-4)--(-3.5,5);
\draw[color=blue,line width=.05cm] (2.7,-4)--(2.7,5);
\draw[color=blue,line width=.05cm] (2.7,2.7)--(6,2.7);


\draw[circle, fill=cadmiumorange,color=cadmiumorange] (-4.3,0.2) circle (0.4);
\draw[circle, fill=cadmiumorange,color=cadmiumorange] (5,0.2) circle (0.4);
\draw[circle, fill=cadmiumorange,color=cadmiumorange] (5,3.3) circle (0.4);
\draw[circle,fill=cadmiumorange,color=cadmiumorange] (-0.4,0.2) circle (0.4);
\node at (-0.4,0.2) {$\lambda$};
\node at (-4.3,0.2) {$\lambda_0$};
\node at (5,0.2) {$\lambda_3$};
\node at (5,3.3) {$\lambda_4$};
\end{tikzpicture}
\begin{tikzpicture}[scale=0.6,every node/.style={scale=0.6}]
\node (ale) at (-7, -3.5) {$\aleph_1$};
\node (a) at (-5, -3.5) {$\pp$};
\node (as) at (-5, -1) {$\sla{\I,\fin}$};
\node (b) at (-5, 1.5) {$\bb$};
\node (ba) at (-5, 4) {$\nonm$};
\node (aa) at (-2, -3.5) {$\slamg{\J'}$};
\node (aas) at (-2, -1) {$\sla{\I,\J'}$};
\node (bb) at (-2, 1.5) {$\bb_{\J'}$};
\node (bba) at (-2, 4) {$\sla{h,\J'}$};
\node (c) at (1, -3.5) {$\dslaml{\J'}$};
\node (cs) at (1, -1) {$\dsla{\I,\J'}$};
\node (f) at (1, 1.5) {$\dd_{\J'}$};
\node (csa) at (1, 4) {$\dsla{h,\J'}$};
\node (xpf) at (4, 1.5) {$\dd$};
\node (xpc) at (4, -3.5) {$\covm$};
\node (xpcs) at (4, -1) {$\dslago{\I}$};
\node (xpfa) at (4, 4) {$\cofn$};
\node (cont) at (6, 4) {$\cc$};
\foreach \from/\to in {aas/bb,aas/cs,cs/f,aas/cs} \draw [line width=.15cm,
white] (\from) -- (\to);
\foreach \from/\to in {ale/a,a/as, aa/aas, c/cs, b/bb, a/aa, aa/c, bb/f, as/b, aas/bb, cs/f, as/aas, aas/cs,f/xpf,cs/xpcs,c/xpc,xpcs/xpf,xpc/xpcs, b/ba,bb/bba,f/csa,xpf/xpfa,ba/bba,bba/csa,csa/xpfa,xpfa/cont} \draw [->] (\from) -- (\to);

\draw[color=blue,line width=.05cm] (-6,-4)--(-6,5);
\draw[color=blue,line width=.05cm] (-3.5,-4)--(-3.5,5);
\draw[color=blue,line width=.05cm] (-0.5,-4)--(-0.5,5);
\draw[color=blue,line width=.05cm] (2.7,-4)--(2.7,5);
\draw[color=blue,line width=.05cm] (2.7,2.7)--(6,2.7);


\draw[circle, fill=cadmiumorange,color=cadmiumorange] (-1.2,0.2) circle (0.4);
\draw[circle, fill=cadmiumorange,color=cadmiumorange] (-4.3,0.2) circle (0.4);
\draw[circle, fill=cadmiumorange,color=cadmiumorange] (5,0.2) circle (0.4);
\draw[circle, fill=cadmiumorange,color=cadmiumorange] (5,3.3) circle (0.4);
\draw[circle,fill=cadmiumorange,color=cadmiumorange] (0.4,0.2) circle (0.4);
\node at (0.4,0.2) {$\lambda_2$};
\node at (-1.2,0.2) {$\lambda_1$};
\node at (-4.3,0.2) {$\lambda_0$};
\node at (5,0.2) {$\lambda_3$};
\node at (5,3.3) {$\lambda_4$};
\end{tikzpicture}
\caption{Constellations forced in~\Cref{mejver}.}
\label{figappl:a}
\end{figure}  
\end{theorem}
\begin{proof}

Construct a~simple coherent system on $I^\sbf:=\pts(\lambda_4)$ (ordered by $\subseteq$) of FS iterations of length
$\pi:=\lambda_4\lambda_3$, where $\Por_{A,1}:=\Cor_A$ for all $A\subseteq\lambda_4$. 
To proceed with the construction, we fix a cofinal family $C\subseteq [\lambda_3]^{<\lambda_0}$ of size $\lambda_3$ and a function $t\colon \lambda_3\to C$ such that $|t^{-1}[\{w\}]|=\lambda_3$ for all $w\in C$. Partition $\lambda_4$ into sets $\seqn{S_\zeta}{\zeta<\lambda_3}$ of size $\lambda_4$ and, for $w\subseteq \lambda_3$, set $S^*_w:= \bigcup_{\zeta\in w}S_\zeta$. 
For $0<\alpha<\lambda_3$ and $\rho<\lambda_4$, define $\Delta(\lambda_4\alpha+\rho):=S^*_{t(\alpha)}$ and $\Delta(\rho):=S^*_{t(0)}$, the latter when $\rho>0$.

The iteration is constructed at each interval $[\lambda_4\alpha,\lambda_4(\alpha+1))$ as follows. Using the $\Delta$ defined above, define $\Qnm_{\lambda_4\alpha}:=\Dor^{V_{\Delta(\lambda_4\alpha),\lambda_4\alpha}}$ when $\alpha>0$, where $\Dor$ denotes Hechler forcing. Also allowing $\alpha=0$, enumerate all the nice $\Por_{S^*_{t(\alpha)},\lambda_4\alpha}$-names $\seqn{\Qnm_\xi}{\lambda_4\alpha<\xi<\lambda_4(\alpha+1)}$ of all the $\sigma$-centered posets with domain contained in~$\lambda_4$ of size ${<}\lambda_0$. This is possible by the assumption $\lambda_4^{<\lambda_0}=\lambda_4$, as it is always forced that $\cfrak\leq\lambda_4$. At each $\lambda_4\alpha<\xi<\lambda_4(\alpha+1)$ we use $\Qnm_\xi$ for the successor step (considering the value of $\Delta(\xi)$ as well). This finishes the forcing construction.

Define $I^0:=\set{A\subseteq\lambda_4}{(\exists w\in C)\ A\subseteq S^*_w}$. Notice that $\ran\Delta = \set{S^*_w}{w\in C}$ is cofinal in $I^0$. Also, for any $B\in\pts(\lambda_4)\menos I^0$, $\bfrak(I^0_{\subseteq B}) = \lambda_0$ and $\Por_{B,1} = \Cor_B$ is the direct limit of $\seqn{\Por_{A,1}}{A\in I^0_{\subseteq B}}$. Therefore, \Cref{mainpresunb} can be applied to conclude that, for any $B\subsetneq B'\subseteq\lambda_4$, $\Por_{B',\pi}$ adds a Cohen real over $V_{B,\pi}$. 
On the other hand, for any limit ordinal $\delta<\lambda^+_4$ of cofinality between $\lambda_0$ and $\lambda_3$, we can construct an strictly increasing sequence $\seqn{B^\delta_\zeta}{\zeta<\delta}$ in $\pts(\lambda_4)$ such that any $A\in I^0$ is contained in some $B^\delta_\zeta$. To see this, pick an increasing cofinal sequence $\seqn{\delta_\gamma}{\gamma<\cf(\delta)}$ in $\delta$ and a $\subsetneq$-increasing sequence $\seqn{w_\gamma}{\gamma<\cf(\delta)}$ such that $\bigcup_{\gamma<\cf(\delta)}w_\gamma = \lambda_3$ (the latter is possible because $\cf(\delta)\leq\lambda_3$). For each $\gamma<\delta$, since $|\delta_{\gamma+1} \menos \delta_\gamma| \leq \lambda_4$ (because $\delta<\lambda_4^+$), we can find a $\subsetneq$-increasing sequence $\seqn{B^\delta_\zeta}{\delta_\gamma\leq \zeta<\delta_{\gamma+1}}$ such that $B^\delta_{\delta_\gamma} = S^*_{w_\gamma}$ and $B^\delta_\zeta \subseteq S^*_{w_{\gamma+1}}$. Now, for any $w\in C$, since $\cf(\delta)\geq\lambda_0$, there is some $\gamma<\cf(\delta)$ such that $w\subseteq w_\gamma$. Therefore, $S^*_w \subseteq B^\delta_{\gamma_\delta}$.

Thus, by \Cref{Thm:mvert}, $\Por_{\lambda_4,\pi}$ is the direct limit of $\seqn{\Por_{B^\delta_\zeta,\pi}}{\zeta<\delta}$. 
As a consequence, $\Por_{\lambda_4,\pi}$ forces $\non{\Mwf}\leq\allowbreak\lambda_0$ by using the sequence $\seqn{B^{\lambda_0}_\zeta}{\zeta<\lambda_0}$, and it forces $\lambda_3\leq\cov{\Mwf}$ by using the sequence $\seqn{B^{\lambda_3}_\zeta}{\zeta<\lambda_3}$.

The small $\sigma$-centered iterands ensure that $\Por_{\lambda_4,\pi}$ forces $\lambda_0\leq \pfrak$, while the Hechler posets ensure $\dfrak\leq \lambda_3$. See the cited reference for $\cofn =\cfrak = \lambda_4$.

\ref{mejver:a}: Assume $\lambda_0\leq\lambda\leq\lambda_3$ regular and $\lambda_4^{<\lambda}=\lambda_4$. By considering the sequence $\seqn{B^{\lambda_4\lambda}_\zeta}{\zeta<\lambda_4\lambda}$, we can use~\Cref{seq:sn} to get that $\Por_{\lambda_4,\pi}$ forces that there is a~maximal ideal $J$ such that $\slalome(\star,J) =\slalomt^\perp(h,J)= \slalomt(h,J) = \lambda$.

\ref{mejver:b}: By~\ref{mejver:a} applied to $\lambda_1$ and $\lambda_2$, in $V_{\lambda_4,\pi}$ there are maximal ideals $J_1$ and $J_2$ such that $\slalome(\star,J_1) = \slalomt^\perp(h_1,J_1)= \slalomt(h_1,J_1) = \lambda_1$ and $\slalome(\star,J_2) = \slalomt^\perp(h_2,J_2) =\slalomt(h_2,J_2) = \lambda_2$ for any $J_i$-unbounded $h_i$ and $i\in\{1,2\}$. 
Let $J':=J_1\oplus J_2$. As any $h\colon \omega\oplus \omega\to \omega$ with $\lim^{J'} h=\infty$ has the form $h=h_1\oplus h_2$ with $\lim^{J_1}h_1 = \lim^{J_2}h_2 = \infty$, by \Cref{L3.5} we conclude 
that $\slalome(\star,J') = \slalomt^\perp(h,J') = \slalome(h,J') = \lambda_1$ and $\slalomt(\star,J') = \slalome^\perp(h,J')= \slalomt(h,J') = \lambda_2$.
\end{proof}

\begin{theorem}[cf.~{\cite[Thm.~4.6~(e)]{mejvert}}]\label{mejvert:amob}
Let $\lambda_0\geq \aleph_1$ be a~regular cardinal and $\lambda_3\leq\lambda_4$ cardinals such that\/ $\cof{[\lambda_3]^{<\lambda_0}}=\lambda_3$ and $\lambda_4=\lambda_4^{<\lambda_0}$. Then there is a~ccc poset forcing that\/ $\pfrak=\nonm=\lambda_0$, $\covm=\cofn=\lambda_3$, $\lambda_4 = \cfrak$, and:
\begin{enumerate}[label=\normalfont(\alph*)]
    \item For any regular cardinal $\lambda$ such that $\lambda_0\leq\lambda\leq\lambda_3$ and $\lambda_4^{<\lambda}=\lambda_4$, there is some maximal ideal $J$ such that\/ $\slalome(\star,J) =\slalomt^\perp(h,J) = \slalomt(h,J) = \lambda$ for any $J$-unbounded $h\in\Baire$ (see~\Cref{figappl:b}). We can remove\/ $\slalome(\star,J)$ when the assumption $\lambda_4^{<\lambda} = \lambda_4$ is removed. 

    \item For any regular cardinals $\lambda_1$ and $\lambda_2$ such that $\lambda_0\leq\lambda_1\leq\lambda_2\leq\lambda_3$ and $\lambda_4^{<\lambda_2}=\lambda_4$, there is some ideal~$J'$ satisfying\/ $\slalome(\star,J') = \slalomt^\perp(h,J') = \slalome(h,J') = \lambda_1$ and\/ $\slalomt(\star,J') = \slalome^\perp(h,J') = \slalomt(h,J') = \lambda_2$ for any $h\in\Baire$ such that $\lim^{J'}h = \infty$ (see~\Cref{figappl:b}). 
    We can remove\/ $\slalome(\star,J')$ and\/ $\slalomt(\star,J')$ when the assumption $\lambda_4^{<\lambda_2} = \lambda_4$ is removed. 
\end{enumerate}
\begin{figure}[ht]
\centering
\begin{tikzpicture}[scale=0.6,every node/.style={scale=0.6}]
\node (ale) at (-7, -3.5) {$\aleph_1$};
\node (a) at (-5, -3.5) {$\pp$};
\node (as) at (-5, -1) {$\sla{\I,\fin}$};
\node (b) at (-5, 1.5) {$\bb$};
\node (ba) at (-5, 4) {$\nonm$};
\node (aa) at (-2, -3.5) {$\slamg{\J}$};
\node (aas) at (-2, -1) {$\sla{\I,\J}$};
\node (bb) at (-2, 1.5) {$\bb_\J$};
\node (bba) at (-2, 4) {$\sla{h,\J}$};
\node (c) at (1, -3.5) {$\dslaml{\J}$};
\node (cs) at (1, -1) {$\dsla{\I,\J}$};
\node (f) at (1, 1.5) {$\dd_\J$};
\node (csa) at (1, 4) {$\dsla{h,\J}$};
\node (xpf) at (4, 1.5) {$\dd$};
\node (xpc) at (4, -3.5) {$\covm$};
\node (xpcs) at (4, -1) {$\dslago{\I}$};
\node (xpfa) at (4, 4) {$\cofn$};
\node (cont) at (6, 4) {$\cc$};
\foreach \from/\to in {aas/bb,aas/cs,cs/f,aas/cs} \draw [line width=.15cm,
white] (\from) -- (\to);
\foreach \from/\to in {ale/a,a/as, aa/aas, c/cs, b/bb, a/aa, aa/c, bb/f, as/b, aas/bb, cs/f, as/aas, aas/cs,f/xpf,cs/xpcs,c/xpc,xpcs/xpf,xpc/xpcs, b/ba,bb/bba,f/csa,xpf/xpfa,ba/bba,bba/csa,csa/xpfa,xpfa/cont} \draw [->] (\from) -- (\to);

\draw[color=blue,line width=.05cm] (-6,-4)--(-6,5);
\draw[color=blue,line width=.05cm] (-3.5,-4)--(-3.5,5);
\draw[color=blue,line width=.05cm] (2.7,-4)--(2.7,5);
\draw[color=blue,line width=.05cm] (5.3,-4)--(5.3,5);

\draw[circle, fill=cadmiumorange,color=cadmiumorange] (-4.3,0.2) circle (0.4);
\draw[circle, fill=cadmiumorange,color=cadmiumorange] (3.4,0.2) circle (0.4);
\draw[circle, fill=cadmiumorange,color=cadmiumorange] (6,0.2) circle (0.4);
\draw[circle,fill=cadmiumorange,color=cadmiumorange] (-0.4,0.2) circle (0.4);
\node at (-0.4,0.2) {$\lambda$};
\node at (-4.3,0.2) {$\lambda_0$};
\node at (3.4,0.2) {$\lambda_3$};
\node at (6,0.2) {$\lambda_4$};
\end{tikzpicture}
\begin{tikzpicture}[scale=0.6,every node/.style={scale=0.6}]
\node (ale) at (-7, -3.5) {$\aleph_1$};
\node (a) at (-5, -3.5) {$\pp$};
\node (as) at (-5, -1) {$\sla{\I,\fin}$};
\node (b) at (-5, 1.5) {$\bb$};
\node (ba) at (-5, 4) {$\nonm$};
\node (aa) at (-2, -3.5) {$\slamg{\J'}$};
\node (aas) at (-2, -1) {$\sla{\I,\J'}$};
\node (bb) at (-2, 1.5) {$\bb_{\J'}$};
\node (bba) at (-2, 4) {$\sla{h,\J'}$};
\node (c) at (1, -3.5) {$\dslaml{\J'}$};
\node (cs) at (1, -1) {$\dsla{\I,\J'}$};
\node (f) at (1, 1.5) {$\dd_{\J'}$};
\node (csa) at (1, 4) {$\dsla{h,\J'}$};
\node (xpf) at (4, 1.5) {$\dd$};
\node (xpc) at (4, -3.5) {$\covm$};
\node (xpcs) at (4, -1) {$\dslago{\I}$};
\node (xpfa) at (4, 4) {$\cofn$};
\node (cont) at (6, 4) {$\cc$};
\foreach \from/\to in {aas/bb,aas/cs,cs/f,aas/cs} \draw [line width=.15cm,
white] (\from) -- (\to);
\foreach \from/\to in {ale/a,a/as, aa/aas, c/cs, b/bb, a/aa, aa/c, bb/f, as/b, aas/bb, cs/f, as/aas, aas/cs,f/xpf,cs/xpcs,c/xpc,xpcs/xpf,xpc/xpcs, b/ba,bb/bba,f/csa,xpf/xpfa,ba/bba,bba/csa,csa/xpfa,xpfa/cont} \draw [->] (\from) -- (\to);

\draw[color=blue,line width=.05cm] (-6,-4)--(-6,5);
\draw[color=blue,line width=.05cm] (-3.5,-4)--(-3.5,5);
\draw[color=blue,line width=.05cm] (-0.5,-4)--(-0.5,5);
\draw[color=blue,line width=.05cm] (2.7,-4)--(2.7,5);
\draw[color=blue,line width=.05cm] (5,-4)--(5,5);

\draw[circle, fill=cadmiumorange,color=cadmiumorange] (-1.2,0.2) circle (0.4);
\draw[circle, fill=cadmiumorange,color=cadmiumorange] (3.4,0.2) circle (0.4);
\draw[circle,fill=cadmiumorange,color=cadmiumorange] (0.4,0.2) circle (0.4);
\draw[circle,fill=cadmiumorange,color=cadmiumorange] (-4.3,0.2) circle (0.4);
\draw[circle,fill=cadmiumorange,color=cadmiumorange] (5.65,0.2) circle (0.4);
\node at (0.4,0.2) {$\lambda_2$};
\node at (-1.2,0.2) {$\lambda_1$};
\node at (-4.3,0.2) {$\lambda_0$};
\node at (3.4,0.2) {$\lambda_3$};
\node at (5.65,0.2) {$\lambda_4$};
\end{tikzpicture}
\caption{Constellation forced in~\Cref{mejvert:amob}.}
\label{figappl:b}
\end{figure} 
\end{theorem}
\begin{proof}
Proceed exactly as in the proof of~\Cref{mejver}, but use amoeba forcing instead of $\Dor$ to guarantee that $\cof\Nwf\leq\lambda_3$.
\end{proof}

\begin{theorem}[cf.~{\cite[Thm.~5.3]{BCM}}]\label{bcm}
Let $\lambda_0\leq \lambda_1\leq \lambda_2$ be uncountable regular cardinals and let $\lambda_3$ be a~cardinal such that $\lambda_2\leq\lambda_3=\lambda_3^{{<}\lambda_2}$. 
Then there is some ccc poset forcing that\/ $\pfrak=\bfrak=\lambda_0$, $\nonm=\lambda_1$, $\covm=\lambda_2$, $\dfrak=\cofn=\cfrak=\lambda_3$, and:
\begin{enumerate}[label=\rm(\alph*)]
    \item\label{bcm:a} there is some maximal ideal $\J_1$ such that\/ $\slalome(\star,\J_1) = \slalomt^\perp(h,\J_1)= \slalomt(h,\J_1) = \lambda_1$ for any $J_1$-unbounded $h$ (see~\Cref{bcm:fig}),

    \item\label{bcm:b} there is some maximal ideal $\J_2$ such that\/ $\slalome(\star,\J_2) = \slalomt^\perp(h,\J_2) =\slalomt(h,\J_2) = \lambda_2$ for any $J_2$-unbounded $h$ (see~\Cref{bcm:fig}), and

    \item\label{bcm:c} there is some ideal $J$ such that\/ $\slalome(\star,J) = \slalomt^\perp(h,J)= \slalome(h,J) = \lambda_1$ and\/ $\slalomt(\star,J) = \slalome^\perp(h,J) = \slalomt(h,J) = \lambda_2$ for any $h\in\Baire$ such that\/ $\lim^J h = \infty$ (see~\Cref{bcm:fig2}).
\end{enumerate}
\begin{figure}[ht]
\centering
\begin{tikzpicture}[scale=0.6,every node/.style={scale=0.6}]
\node (ale) at (-7, -3.5) {$\aleph_1$};
\node (a) at (-5, -3.5) {$\pp$};
\node (as) at (-5, -1) {$\sla{\I,\fin}$};
\node (b) at (-5, 1.5) {$\bb$};
\node (ba) at (-5, 4) {$\nonm$};
\node (aa) at (-2, -3.5) {$\slamg{\J_1}$};
\node (aas) at (-2, -1) {$\sla{\I,\J_1}$};
\node (bb) at (-2, 1.5) {$\bb_{\J_1}$};
\node (bba) at (-2, 4) {$\sla{h,\J_1}$};
\node (c) at (1, -3.5) {$\dslaml{\J_1}$};
\node (cs) at (1, -1) {$\dsla{\I,\J_1}$};
\node (f) at (1, 1.5) {$\dd_{\J_1}$};
\node (csa) at (1, 4) {$\dsla{h,\J_1}$};
\node (xpf) at (4, 1.5) {$\dd$};
\node (xpc) at (4, -3.5) {$\covm$};
\node (xpcs) at (4, -1) {$\dslago{\I}$};
\node (xpfa) at (4, 4) {$\cofn$};
\node (cont) at (6, 4) {$\cc$};
\foreach \from/\to in {aas/bb,aas/cs,cs/f,aas/cs} \draw [line width=.15cm,
white] (\from) -- (\to);
\foreach \from/\to in {ale/a,a/as, aa/aas, c/cs, b/bb, a/aa, aa/c, bb/f, as/b, aas/bb, cs/f, as/aas, aas/cs,f/xpf,cs/xpcs,c/xpc,xpcs/xpf,xpc/xpcs, b/ba,bb/bba,f/csa,xpf/xpfa,ba/bba,bba/csa,csa/xpfa,xpfa/cont} \draw [->] (\from) -- (\to);

\draw[color=blue,line width=.05cm] (-6,-4)--(-6,5);
\draw[color=blue,line width=.05cm] (-6,2.7)--(-3.7,2.7);
\draw[color=blue,line width=.05cm] (-3.7,-4)--(-3.7,2.7);
\draw[color=blue,line width=.05cm] (2.7,0.25)--(2.7,5);
\draw[color=blue,dashed,line width=.05cm] (2.7,0.25)--(6,0.25);
\draw[color=blue,dashed,line width=.05cm] (2.7,-2.25)--(6,-2.25);
\draw[color=blue,line width=.05cm] (2.7,-4)--(2.7,0.25);

\draw[circle, fill=cadmiumorange,color=cadmiumorange] (5,-2.85) circle (0.4);
\draw[circle, fill=cadmiumorange,color=cadmiumorange] (5,2.7) circle (0.4);
\draw[circle,fill=cadmiumorange,color=cadmiumorange] (-0.4,0.2) circle (0.4);
\draw[circle,fill=cadmiumorange,color=yellow] (5.5,-0.9) circle (0.4);
\draw[circle,fill=cadmiumorange,color=cadmiumorange] (-4.3,0.9) circle (0.4);
\node at (-0.4,0.2) {$\lambda_1$};
\node at (-4.3,0.9) {$\lambda_0$};
\node at (5,-2.85) {$\lambda_2$};
\node at (5,2.7) {$\lambda_3$};
\node at (5.5,-0.9) {$\textbf{?}$};
\end{tikzpicture}
\begin{tikzpicture}[scale=0.6,every node/.style={scale=0.6}]
\node (ale) at (-7, -3.5) {$\aleph_1$};
\node (a) at (-5, -3.5) {$\pp$};
\node (as) at (-5, -1) {$\sla{\I,\fin}$};
\node (b) at (-5, 1.5) {$\bb$};
\node (ba) at (-5, 4) {$\nonm$};
\node (aa) at (-2, -3.5) {$\slamg{\J_2}$};
\node (aas) at (-2, -1) {$\sla{\I,\J_2}$};
\node (bb) at (-2, 1.5) {$\bb_{\J_2}$};
\node (bba) at (-2, 4) {$\sla{h,\J_2}$};
\node (c) at (1, -3.5) {$\dslaml{\J_2}$};
\node (cs) at (1, -1) {$\dsla{\I,\J_2}$};
\node (f) at (1, 1.5) {$\dd_{\J_2}$};
\node (csa) at (1, 4) {$\dsla{h,\J_2}$};
\node (xpf) at (4, 1.5) {$\dd$};
\node (xpc) at (4, -3.5) {$\covm$};
\node (xpcs) at (4, -1) {$\dslago{\I}$};
\node (xpfa) at (4, 4) {$\cofn$};
\node (cont) at (6, 4) {$\cc$};
\foreach \from/\to in {aas/bb,aas/cs,cs/f,aas/cs} \draw [line width=.15cm,
white] (\from) -- (\to);
\foreach \from/\to in {ale/a,a/as, aa/aas, c/cs, b/bb, a/aa, aa/c, bb/f, as/b, aas/bb, cs/f, as/aas, aas/cs,f/xpf,cs/xpcs,c/xpc,xpcs/xpf,xpc/xpcs, b/ba,bb/bba,f/csa,xpf/xpfa,ba/bba,bba/csa,csa/xpfa,xpfa/cont} \draw [->] (\from) -- (\to);

\draw[color=blue,line width=.05cm] (-6,-4)--(-6,5);
\draw[color=blue,line width=.05cm] (-6,2.7)--(-3.7,2.7);
\draw[color=blue,line width=.05cm] (-3.7,-4)--(-3.7,5);

\draw[color=blue,line width=.05cm] (2.7,0.25)--(2.7,5);
\draw[color=blue,dashed,line width=.05cm] (2.7,0.25)--(6,0.25);
\draw[color=blue,dashed,line width=.05cm] (2.7,-2.25)--(6,-2.25);
\draw[color=blue,line width=.05cm] (2.7,-2.25)--(2.7,0.25);


\draw[circle, fill=cadmiumorange,color=cadmiumorange] (-4.3,3.3) circle (0.4);
\draw[circle, fill=cadmiumorange,color=cadmiumorange] (-4.3,0.9) circle (0.4);
\draw[circle, fill=cadmiumorange,color=cadmiumorange] (5,2.7) circle (0.4);
\draw[circle,fill=cadmiumorange,color=cadmiumorange] (-0.4,0.2) circle (0.4);
\draw[circle,fill=cadmiumorange,color=yellow] (5.5,-0.9) circle (0.4);
\node at (-0.4,0.2) {$\lambda_2$};
\node at (-4.3,3.3) {$\lambda_1$};
\node at (-4.3,0.9) {$\lambda_0$};
\node at (5.5,-0.9) {$\textbf{?}$};
\node at (5,2.7) {$\lambda_3$};
\end{tikzpicture}
\caption{Two separations of the cardinals for the two idelals $\J_1$ and $\J_2$ in \Cref{bcm}. On the left, we have the constellation of~\ref{bcm:a}, and on the right the constellation of~\ref{bcm:b}. The dotted lines indicate that $\lambda_2\leq\slalomt(I,\Fin)\leq\lambda_3$, whose exact value is unclear.} 
\label{bcm:fig}
\end{figure} 

\begin{figure}[h!]
\centering
\begin{tikzpicture}[scale=0.8,every node/.style={scale=0.8}]
\node (ale) at (-7, -3.5) {$\aleph_1$};
\node (a) at (-5, -3.5) {$\pp$};
\node (as) at (-5, -1) {$\sla{\I,\fin}$};
\node (b) at (-5, 1.5) {$\bb$};
\node (ba) at (-5, 4) {$\nonm$};
\node (aa) at (-2, -3.5) {$\slamg{\J}$};
\node (aas) at (-2, -1) {$\sla{\I,\J}$};
\node (bb) at (-2, 1.5) {$\bb_\J$};
\node (bba) at (-2, 4) {$\sla{h,\J}$};
\node (c) at (1, -3.5) {$\dslaml{\J}$};
\node (cs) at (1, -1) {$\dsla{\I,\J}$};
\node (f) at (1, 1.5) {$\dd_\J$};
\node (csa) at (1, 4) {$\dsla{h,\J}$};
\node (xpf) at (4, 1.5) {$\dd$};
\node (xpc) at (4, -3.5) {$\covm$};
\node (xpcs) at (4, -1) {$\dslago{\I}$};
\node (xpfa) at (4, 4) {$\cofn$};
\node (cont) at (6, 4) {$\cc$};
\foreach \from/\to in {aas/bb,aas/cs,cs/f,aas/cs} \draw [line width=.15cm,
white] (\from) -- (\to);
\foreach \from/\to in {ale/a,a/as, aa/aas, c/cs, b/bb, a/aa, aa/c, bb/f, as/b, aas/bb, cs/f, as/aas, aas/cs,f/xpf,cs/xpcs,c/xpc,xpcs/xpf,xpc/xpcs, b/ba,bb/bba,f/csa,xpf/xpfa,ba/bba,bba/csa,csa/xpfa,xpfa/cont} \draw [->] (\from) -- (\to);

\draw[color=blue,line width=.05cm] (-6,-4)--(-6,5);
\draw[color=blue,line width=.05cm] (-6,2.7)--(-3.7,2.7);
\draw[color=blue,line width=.05cm] (-3.7,-4)--(-3.7,2.7);
\draw[color=blue,line width=.05cm] (2.7,0.25)--(2.7,5);
\draw[color=blue,line width=.05cm] (-0.5,-4)--(-0.5,5);
\draw[color=blue,dashed,line width=.05cm] (2.7,0.25)--(6,0.25);
\draw[color=blue,dashed,line width=.05cm] (2.7,-2.25)--(6,-2.25);
\draw[color=blue,line width=.05cm] (2.7,-2.25)--(2.7,0.25);


\draw[circle, fill=cadmiumorange,color=cadmiumorange] (-1.2,0.2) circle (0.4);
\draw[circle, fill=cadmiumorange,color=cadmiumorange] (-4.3,0.9) circle (0.4);
\draw[circle, fill=cadmiumorange,color=cadmiumorange] (5,2.7) circle (0.4);
\draw[circle,fill=cadmiumorange,color=cadmiumorange] (0.4,0.2) circle (0.4);
\draw[circle,fill=cadmiumorange,color=yellow] (5.5,-0.9) circle (0.4);
\node at (0.4,0.2) {$\lambda_2$};
\node at (-1.2,0.2) {$\lambda_1$};
\node at (-4.3,0.9) {$\lambda_0$};
\node at (5,2.7) {$\lambda_3$};
\node at (5.5,-0.9) {$\textbf{?}$};
\end{tikzpicture}
    \caption{Constellation forced in~\Cref{bcm}~\ref{bcm:c}. The dotted lines indicate that $\lambda_2\leq\slalomt(I,\Fin)\leq\lambda_3$, whose exact value is unclear.}
\label{bcm:fig2}
\end{figure}
\end{theorem}
\begin{proof}
Construct a~simple coherent system on $I^\sbf=\lambda_3\lambda_2+1$, ordered by $\leq$, of FS iterations of length $\pi:=\lambda_3\lambda_2\lambda_1$ (ordinal product), where $\Por_{\eta,1}:=\Cor_{\eta}$ for all $\eta\leq\lambda_3\lambda_2$, 
whose iterands for $0<\xi<\pi$ are determined by:
\begin{itemize}
\item all $\sigma$-centered posets of size~${<}\lambda_0$;

\item all $\sigma$-centered subposets of Hechler forcing of size~${<}\lambda_1$; and 

\item $\Qnm_{\xi}:= \Eor^{V_{\Delta(\xi),\xi}}$.
\end{itemize}
The iteration is constructed via book-keeping as in~\cite[Thm.~5.3]{BCM} and the previous proofs. 
Then $\Por_{\lambda_3\lambda_2,\pi}$ forces $\pfrak=\bfrak=\lambda_0$, $\nonm=\lambda_1$, $\covm=\lambda_2$, and $\dfrak=\cofn=\cfrak=\lambda_3$ (details can be found in the cited reference). Since $\Por_{\lambda_3\lambda_2,\pi}$ is obtained by the FS iteration $\la\Por_{\lambda_3\lambda_2,\xi},\Qnm_{\lambda_3\lambda_2,\xi}:\, \xi<\pi\ra$ and $\cf(\pi)=\lambda_1$, by applying~\Cref{FS:sn} we obtain that $\Por_{\lambda_3\lambda_2,\pi}$ forces that there is a~maximal ideal $J_1$ such that $\slalome(\star,\J_1) = \slalomt^\perp(h,\J_1) = \slalomt(h,\J_1) = \lambda_1$ for any $J_1$-unbounded $h\in \Baire$.

\ref{bcm:b}:
For each $\zeta<\lambda_3\lambda_2$, $\Por_{\lambda_3\lambda_2,\zeta+1}$ adds a~Cohen real over $V_{\lambda_3\lambda_2,\zeta+1}$. Then, by~\Cref{Thm:mvert} applied to $\seqn{\zeta}{\zeta<\lambda_3\lambda_2}$, we have that $\Por_{\lambda_3\lambda_2,\pi}$ forces that there is a~maximal ideal $J_2$ on~$\omega$ such that $\slalome(\star,J_2)= \slalomt^\perp(h,\J_2) =\slalomt(h,J_2)=\lambda_2$ for any $J_2$-unbounded $h\in \Baire$.

\ref{bcm:c}:
Use $J=J_1\oplus J_2$ exactly as in the previous results.
\end{proof}

\section{Discussions and open problems}\label{Sdisc_probl}

By \Cref{C2.7} and \Cref{L3.4}, if $J$ has the Baire property then $J$ does not affect the values of many slalom numbers, i.e., a~slalom number with $J$ is equal to the one with $\fin$. For instance, if $h$ is reasonable, then $\slalome(h,J)=\slalome(h,\Fin)$ and $\slalomt(h,J)=\slalomt(h,\Fin)$. However, for two instances of slalom numbers, we were not able to settle such an equality.

\begin{question}
Do we have that $\slalome(I,J) = \slalome(I,\Fin)$ and $\slalome(\star,J) = \slalome(\star,\Fin)$ when $J$ has the Baire property?
\end{question}

From \Cref{aboutbh}, we have a good understanding on when $\slalome(b,h,J)$ is finite or not, and that it can never be $\aleph_0$. However, the situation for $\slalomt(b,h,J)$ is unclear.

\begin{question}
    Do we have examples of $b$, $h$ and $J$ such that $\slalomt(b,h,J) = \aleph_0$?
\end{question}

\begin{question}
    Is there a suitable characterization of $\slalomt(b,h,J) = k$ for any natural number $k\geq 3$?
\end{question}

Concerning \Cref{ftc:oplus}, we wonder about the following problem. It has a positive answer when either $\slalomt(b_0,h_0,J_0)$ or  $\slalomt(b_1,h_1,J_1)$ is infinite, but it is unclear when both are finite (and larger than $1$).
\begin{question}
    Do we have 
    $\slalomt(b_0\oplus b_1,h_0\oplus h_1,J_0\oplus J_1) = \slalomt(b_0,h_0,J_0)\cdot \slalomt(b_1,h_1,J_1)$?
\end{question} 

More open problems about the results of \Cref{sec:sumI} are:

\begin{question}
    Let $J_0$ and $J_1$ be ideals on $\omega$.
    \begin{enumerate}[label = \normalfont (\arabic*)]
        \item For which $h_0\in\Baire$ do we have $\slalome(h_0\oplus h_0,J_0\oplus J_1) = \slalome(h_0,J_0\cap J_1)$? Likewise, we ask when inequalities in \Cref{L3.5}~\ref{L3.5c} are equalities.

        \item Are $\slalome(I_0\oplus I_1,J_0) = \slalome(I_0\cap I_1,J_0) = \max\{\slalome(I_0,J_0),\slalome(I_1,J_0)\}$ and $\pfrak_{\Kat}(I_0\oplus I_1,J_0) = \pfrak_{\Kat}(I_0\cap I_1,J_0) = \max\{\pfrak_{\Kat}(I_0,J_0),\pfrak_{\Kat}(I_1,J_0)\}$ for any ideals $I_0$ and $I_1$ on $\omega$?

        \item Are $\slalome(I,J_0\cap J_1) = \slalome(I,J_0\oplus J_1)$, $\pfrak_{\Kat}(I,J_0\cap J_1) = \pfrak_{\Kat}(I,J_0\oplus J_1)$, $\slalome(\star,J_0\cap J_1) = \slalome(\star,J_0\oplus J_1)$, and $\pfrak_{\Kat}(\star,J_0\cap J_1) = \pfrak_{\Kat}(\star,J_0\oplus J_1)$ for any ideal $I$ on $\omega$?
    \end{enumerate}
\end{question}

By \Cref{critical}, slalom numbers are uniformity numbers of certain selection principles. It is not known whether, in many cases, selection principles with equal critical cardinalities must be equivalent, for instance:

\begin{question}
Is $\mathrm{S}_1(\Gamma_h,\GammaB{{J}})$ equivalent to $\mathrm{S}_1(\Gamma_h,\Gamma)$  when $J$ has the Baire property? The same applies to $\Srm_1(I\rr\Gamma,\LambdaB{J})$, $\Srm_1(\Omega,\LambdaB{J})$, $\Srm_1(\Gamma_h,\LambdaB{J})$, and $\Srm_1(\Gamma,J\rr\Gamma)$.
\end{question}

In \Cref{LOO}, we show that many selection principles with $\OO$ in the second argument are equivalent to those with $\Fin\text{-}\Lambda$ in the second argument. The same applies to their uniformity numbers. However, the following is still not clear.

\begin{question}\label{QGammagO}
Are $\mathrm{S}_1(\Gamma_g,\OO)$ and $\mathrm{S}_1(\Gamma_g,\Fin\text{-}\Lambda)$ equivalent? 
\end{question}

A positive answer to the following question solves \Cref{QGammagO}.

\begin{question}
    Are $\mathrm{S}_1(\Gamma_g,\OO)$ and $\mathrm{S}_1(\Gamma_{g'},\OO)$ equivalent principles for two functions $g$ and $g'$ diverging to infinity? Are $\Srm_1(\Gamma_h,\Gamma)$ and $\Srm_1(\Gamma_1,\Gamma)$ equivalent principles when $h\geq^*1$?
\end{question}

As shown in \Cref{sec:forcing}, many instances of slalom numbers can be distinguished. By \Cref{critical}, the same applies to the corresponding selection principles assuming inequalities between cardinal invariants. On the other hand, we do not know what happens under assumptions compatible with {\rm\bf CH}.

\begin{question}
If {\rm\bf CH} holds, is there an $\schema{\Gammah{h}}{\mathcal{O}}$-space which is not an $\schema{\Gammah{h}}{\Gamma}$-space? The same applies to many pairs of selection principles in \Cref{SjednaABcard} and~\ref{critical_function}. 
\end{question}

Any $\schema{\Gammah{1}}{\mathcal{O}}$-space is totally imperfect by \Cref{CantorShfin}. We may ask about its further topological properties, i.e., properties of an $\schema{\Gammah{1}}{\mathcal{O}}$-space and even an $\schema{\Gammah{1}}{\Gamma}$-space. For instance, by \cite{Comb2}, any $\schema{\Gamma}{\Gamma}$-space $X$ of reals is perfectly meager,\footnote{The notion was discovered in \cite{luzin14} and is called always of the first category as well.} i.e.,  for any perfect set $P$ of reals, the intersection $X\cap P$ is meager in the subspace $P$.

\begin{question}
Is an $\schema{\Gammah{1}}{\Gamma}$-space of reals perfectly meager?   
\end{question}

We also wonder whether we can express other classical cardinal characteristics, like $\addn$, as the critical cardinality of some selection principle or other similar topological property.

Regarding~\Cref{Thm:mixBD}, we ask:

\begin{question}
Is there any model where all four rows of \Cref{BasicDia} are different for some pair $I$,~$J$?
\end{question}

It is possible to force a similar model as in \Cref{Thm:mixBD} but with $\aleph_1<\bfrak<\dfrak<\covn$ (see~\cite[Sec.~5]{GKMScont}). However, we do not know what is the effect on $\bfrak_J$ and $\dfrak_J$ after forcing with a random algebra. On the other hand, Canjar~\cite{Canjar2} has studied the effect on co-initialities of ultrapowers of $\omega$ after forcing with a random algebra.

\begin{question}
    Can we force a constellation like in \Cref{HefollowsRa} but with $\aleph_1 < \bfrak < \bfrak_J < \dfrak_J < \dfrak <\covn$?
\end{question}

We still need to explore the behavior of slalom numbers in generic extensions not adding (too many) Cohen reals. Very few forcing techniques for large continuum work for this, for instance, large products of creature forcing. However, such constructions are $\Baire$-bounding in practice, which force $\dfrak=\aleph_1$ (over a~model of~{\rm\bf CH}). For this reason, this technique could only be used to separate cardinals on the top row of \Cref{BasicDia}. In~\cite{CKM}, continuum many different values were forced for cardinals of the form $\slalomt(b,h,\Fin)$, $\slalome(b,h,\Fin)$, $\slalomt^\perp(b,h,\Fin)$, and $\slalome^\perp(b,h,\Fin)$. We wonder if similar results can be forced for several ideals on $\omega$ instead of $\Fin$.

As a consequence of \Cref{aplcohen}, we can force continuum many cardinals of the form $\slalomt(h,J)$, $\slalome(h,J)$, $\slalomt^\perp(h,J)$ and $\slalome^\perp(h,J)$, even for any fixed $h$ diverging to $\infty$. However, we do not know whether the same is possible for fixed $J$ and varying $h$.


\bibliographystyle{alpha}
\bibliography{liter}

\begin{thebibliography}{BBTFM18}

\bibitem[Bar84]{Ba1984}
Tomek Bartoszy{\'{n}}ski.
\newblock Additivity of measure implies additivity of category.
\newblock {\em Trans. Amer. Math. Soc.}, 281(1):209--213, 1984.

\bibitem[Bar87]{Ba1987}
Tomek Bartoszy{\'{n}}ski.
\newblock Combinatorial aspects of measure and category.
\newblock {\em Fund. Math.}, 127(3):225--239, 1987.

\bibitem[BBFM18]{BBSM}
J{\"{o}}rg Brendle, Andrew Brooke-Taylor, Sy-David Friedman, and Diana~Carolina
  Montoya.
\newblock Cicho\'{n}'s diagram for uncountable cardinals.
\newblock {\em Israel J. Math.}, 225(2):959--1010, 2018.

\bibitem[BCM21]{BCM}
J{\"{o}}rg Brendle, Miguel~A. Cardona, and Diego~A. Mej{\'{\i}}a.
\newblock Filter-linkedness and its effect on preservation of cardinal
  characteristics.
\newblock {\em Ann. Pure Appl. Logic}, 172(1):Paper No. 102856, 30, 2021.

\bibitem[BF12]{BorFar}
Piotr Borodulin-Nadzieja and Barnab{\'{a}}s Farkas.
\newblock Cardinal coefficients associated to certain orders on ideals.
\newblock {\em Arch. Math. Logic}, 51(1-2):187--202, 2012.

\bibitem[BJ95]{BJ}
Tomek Bartoszy{\'{n}}ski and Haim Judah.
\newblock {\em Set theory. On the structure of the real line}.
\newblock A K Peters, Ltd., Wellesley, MA, 1995.

\bibitem[Bla10]{blass}
Andreas Blass.
\newblock Combinatorial cardinal characteristics of the continuum.
\newblock In {\em Handbook of {S}et {T}heory. {V}ols. 1, 2, 3}, pages 395--489.
  Springer, Dordrecht, 2010.

\bibitem[BM99]{BM}
Andreas Blass and Heike Mildenberger.
\newblock On the cofinality of ultrapowers.
\newblock {\em J. Symbolic Logic}, 64(2):727--736, 1999.

\bibitem[BM14]{BrM}
J{\"o}rg Brendle and Diego~Alejandro Mej{\'{\i}}a.
\newblock Rothberger gaps in fragmented ideals.
\newblock {\em Fund. Math.}, 227(1):35--68, 2014.

\bibitem[BS99]{BreShe99}
J{\"{o}}rg Brendle and Saharon Shelah.
\newblock Ultrafilters on {$\omega$}---their ideals and their cardinal
  characteristics.
\newblock {\em Trans. Amer. Math. Soc.}, 351(7):2643--2674, 1999.

\bibitem[BS23]{BS22}
J{\"{o}}rg Brendle and Corey~Bacal Switzer.
\newblock Higher dimensional cardinal characteristics for sets of functions
  {II}.
\newblock {\em J. Symb. Log.}, 88(4):1421--1442, 2023.

\bibitem[B{\v{S}}Z23]{BaSuZd}
Serhii Bardyla, Jaroslav {\v{S}}upina, and Lyubomyr Zdomskyy.
\newblock Ideal approach to convergence in functional spaces.
\newblock {\em Trans. Amer. Math. Soc.}, 376(12):8495--8528, 2023.

\bibitem[Buk11]{BStr}
Lev Bukovsk{\'{y}}.
\newblock {\em {The Structure of the Real Line}}, volume~71 of {\em Instytut
  Matematyczny Polskiej Akademii Nauk. Monografie Matematyczne (New Series)
  [Mathematics Institute of the Polish Academy of Sciences. Mathematical
  Monographs (New Series)]}.
\newblock Birkh\"{a}user/Springer Basel AG, Basel, 2011.

\bibitem[Buk19]{Buky18}
Lev Bukovsk{\'{y}}.
\newblock Selection principle {$\rm S_1$} and combinatorics of open covers.
\newblock {\em Topology Appl.}, 258:239--250, 2019.

\bibitem[Can88]{Canjar2}
Michael Canjar.
\newblock Countable ultraproducts without {CH}.
\newblock {\em Ann. Pure Appl. Logic}, 37(1):1--79, 1988.

\bibitem[Car23]{Car4E}
Miguel~A. Cardona.
\newblock A friendly iteration forcing that the four cardinal characteristics
  of {$\mathcal{E}$} can be pairwise different.
\newblock {\em Colloq. Math.}, 173(1):123--157, 2023.

\bibitem[CKM24]{CKM}
Miguel~A. Cardona, Lukas~Daniel Klausner, and Diego~A. Mej{\'i}a.
\newblock Continuum many different things: localisation, anti-localisation and
  {Y}orioka ideals.
\newblock {\em Ann. Pure Appl. Logic}, 175(7):Paper No. 103453, 58, 2024.

\bibitem[CM19]{CM19}
Miguel~A. Cardona and Diego~A. Mej\'{\i}a.
\newblock On cardinal characteristics of {Y}orioka ideals.
\newblock {\em Math. Log. Q.}, 65(2):170--199, 2019.

\bibitem[CM22]{CM22}
Miguel~A. Cardona and Diego~A. Mej{\'{\i}}a.
\newblock Forcing constellations of {C}icho\'n's diagram by using the {T}ukey
  order.
\newblock {\em Ky\={o}to Daigaku S\=urikaiseki Kenky\=usho K\=oky\=uroku},
  2213:14--47, 2022.
\newblock \href{https://arxiv.org/abs/2203.00615}{arXiv:2203.00615}.

\bibitem[CM23]{CM23}
Miguel~A. Cardona and Diego~A. Mej\'{\i}a.
\newblock Localization and anti-localization cardinals.
\newblock {\em Ky\={o}to Daigaku S\=urikaiseki Kenky\=usho K\=oky\=uroku},
  2261:47--77, 2023.
\newblock \href{https://arxiv.org/abs/2305.03248}{arXiv:2305.03248}.

\bibitem[CM25]{CarMej23}
Miguel~A. Cardona and Diego~A. Mej\'ia.
\newblock More about the cofinality and the covering of the ideal of strong
  measure zero sets.
\newblock {\em Ann. Pure Appl. Logic}, 176(4):Paper No. 103537, 31, 2025.

\bibitem[DKC16]{DaKoCh}
Pratulananda Das, Ljubi\v{s}a D.~R. Ko\v{c}inac, and Debraj Chandra.
\newblock Some remarks on open covers and selection principles using ideals.
\newblock {\em Topology Appl.}, 202:183--193, 2016.

\bibitem[Far00]{farah}
Ilijas Farah.
\newblock Analytic quotients: theory of liftings for quotients over analytic
  ideals on the integers.
\newblock {\em Mem. Amer. Math. Soc.}, 148(702):xvi+177, 2000.

\bibitem[FFMM18]{FFMM}
Vera Fischer, Sy~D. Friedman, Diego~A. Mej\'{\i}a, and Diana~C. Montoya.
\newblock Coherent systems of finite support iterations.
\newblock {\em J. Symb. Log.}, 83(1):208--236, 2018.

\bibitem[FK22]{FI_KW_22}
Rafa\l\ Filip\'{o}w and Adam Kwela.
\newblock Yet another ideal version of the bounding number.
\newblock {\em J. Symb. Log.}, 87(3):1065--1092, 2022.

\bibitem[For10]{foreman}
Matthew Foreman.
\newblock Ideals and generic elementary embeddings.
\newblock In {\em Handbook of set theory. {V}ols. 1, 2, 3}, pages 885--1147.
  Springer, Dordrecht, 2010.

\bibitem[FS09]{FaSo10}
Barnab\'{a}s Farkas and Lajos Soukup.
\newblock More on cardinal invariants of analytic {$P$}-ideals.
\newblock {\em Comment. Math. Univ. Carolin.}, 50(2):281--295, 2009.

\bibitem[GKMS21]{GKMSsplit}
Martin Goldstern, Jakob Kellner, Diego~A. Mej\'{\i}a, and Saharon Shelah.
\newblock Preservation of splitting families and cardinal characteristics of
  the continuum.
\newblock {\em Israel J. Math.}, 246(1):73--129, 2021.

\bibitem[GKMS22]{GKMScont}
Martin Goldstern, Jakob Kellner, Diego~A. Mej\'{\i}a, and Saharon Shelah.
\newblock Controlling classical cardinal characteristics while collapsing
  cardinals.
\newblock {\em Colloq. Math.}, 170(1):115--144, 2022.

\bibitem[GM25]{GaMe}
Viera Gavalová and Diego~Alejandro Mejía.
\newblock Lebesgue measure zero modulo ideals on the natural numbers.
\newblock {\em J. Symb. Log.}, 90(3):1098--1128, 2025.

\bibitem[GN82]{GerNag}
J.~Gerlits and Zs. Nagy.
\newblock Some properties of {$C(X)$}. {I}.
\newblock {\em Topology Appl.}, 14(2):151--161, 1982.

\bibitem[GS93]{GS93}
Martin Goldstern and Saharon Shelah.
\newblock Many simple cardinal invariants.
\newblock {\em Arch. Math. Logic}, 32(3):203--221, 1993.

\bibitem[HH07]{HeHru07}
Fernando Hern{\'{a}}ndez-Hern{\'{a}}ndez and Michael Hru{\v{s}}{\'{a}}k.
\newblock Cardinal invariants of analytic {$P$}-ideals.
\newblock {\em Canad. J. Math.}, 59(3):575--595, 2007.

\bibitem[Hru11]{Hr}
Michael Hru{\v{s}}{\'{a}}k.
\newblock Combinatorics of filters and ideals.
\newblock In {\em Set theory and its applications}, volume 533 of {\em Contemp.
  Math.}, pages 29--69. Amer. Math. Soc., Providence, RI, 2011.

\bibitem[HST22]{Bakke}
Karen~Bakke Haga, David Schrittesser, and Asger T\"{o}rnquist.
\newblock Maximal almost disjoint families, determinacy, and forcing.
\newblock {\em J. Math. Log.}, 22(1):Paper No. 2150026, 42, 2022.

\bibitem[JMSS96]{Comb2}
Winfried Just, Arnold~W. Miller, Marion Scheepers, and Paul~J. Szeptycki.
\newblock The combinatorics of open covers. {II}.
\newblock {\em Topology Appl.}, 73(3):241--266, 1996.

\bibitem[Kel08]{K08}
Jakob Kellner.
\newblock Even more simple cardinal invariants.
\newblock {\em Arch. Math. Logic}, 47(5):503--515, 2008.

\bibitem[KM22]{KM}
Lukas~Daniel Klausner and Diego~Alejandro Mej\'{\i}a.
\newblock Many different uniformity numbers of {Y}orioka ideals.
\newblock {\em Arch. Math. Logic}, 61(5-6):653--683, 2022.

\bibitem[KO14]{KO14}
Shizuo Kamo and Noboru Osuga.
\newblock Many different covering numbers of {Y}orioka's ideals.
\newblock {\em Arch. Math. Logic}, 53(1-2):43--56, 2014.

\bibitem[KS09]{KS09}
Jakob Kellner and Saharon Shelah.
\newblock Decisive creatures and large continuum.
\newblock {\em J. Symbolic Logic}, 74(1):73--104, 2009.

\bibitem[KS12]{KS12}
Jakob Kellner and Saharon Shelah.
\newblock Creature forcing and large continuum: the joy of halving.
\newblock {\em Arch. Math. Logic}, 51(1-2):49--70, 2012.

\bibitem[Kur66]{KuraTop}
K.~Kuratovski\u{\i}.
\newblock {\em {Topology. Volume I}}.
\newblock Izdat. ``Mir'', Moscow, 1966.
\newblock Translated from the English by M. Ja. Antonovski\u{\i}, With a
  preface by P. S. Aleksandrov.

\bibitem[Luz14]{luzin14}
N.N. Luzin.
\newblock Sur un probl\`eme de {M}.\ {B}aire.
\newblock {\em C. R. Acad. Sci. Paris}, 158:1258--1261, 1914.

\bibitem[Mej13]{mejiamatrix}
Diego~Alejandro Mej{\'{\i}}a.
\newblock Matrix iterations and {C}icho\'n's diagram.
\newblock {\em Arch. Math. Logic}, 52(3-4):261--278, 2013.

\bibitem[Mej19]{mejvert}
Diego~A. Mej\'{\i}a.
\newblock Matrix iterations with vertical support restrictions.
\newblock In {\em Proceedings of the 14th and 15th {A}sian {L}ogic
  {C}onferences}, pages 213--248. World Sci. Publ., Hackensack, NJ, 2019.
\newblock \href{https://arxiv.org/abs/1803.05102}{\texttt{arXiv:1803.05102}}.

\bibitem[Mez09]{Mez}
David Meza~Alc{\'a}ntara.
\newblock {\em Ideals and filters on countable sets}.
\newblock PhD thesis, Universidad Nacional Aut\'onoma de M\'exico, 2009.

\bibitem[Mil82]{Mi1982}
Arnold~W. Miller.
\newblock A characterization of the least cardinal for which the {B}aire
  category theorem fails.
\newblock {\em Proc. Amer. Math. Soc.}, 86(3):498--502, 1982.

\bibitem[Mil84]{MiSp}
Arnold~W. Miller.
\newblock Special subsets of the real line.
\newblock In {\em Handbook of set-theoretic topology}, pages 201--233.
  North-Holland, Amsterdam, 1984.

\bibitem[Osi18]{Osip18}
Alexander~V. Osipov.
\newblock Classification of selectors for sequences of dense sets of
  {$C_p(X)$}.
\newblock {\em Topology Appl.}, 242:20--32, 2018.

\bibitem[Paw85]{pawli}
Janusz Pawlikowski.
\newblock Powers of transitive bases of measure and category.
\newblock {\em Proc. Amer. Math. Soc.}, 93(4):719--729, 1985.

\bibitem[Rep21a]{REP21a}
Miroslav Repick\'{y}.
\newblock Spaces not distinguishing ideal convergences of real-valued
  functions.
\newblock {\em Real Anal. Exchange}, 46(2):367--394, 2021.

\bibitem[Rep21b]{REP21b}
Miroslav Repick\'{y}.
\newblock Spaces not distinguishing ideal convergences of real-valued
  functions, {II}.
\newblock {\em Real Anal. Exchange}, 46(2):395--421, 2021.

\bibitem[RS23]{RaSt}
Dilip Raghavan and Juris Stepr\={a}ns.
\newblock The almost disjointness invariant for products of ideals.
\newblock {\em Topology Appl.}, 323:Paper No. 108295, 11, 2023.

\bibitem[Sch96]{Comb1}
Marion Scheepers.
\newblock Combinatorics of open covers. {I}. {R}amsey theory.
\newblock {\em Topology Appl.}, 69(1):31--62, 1996.

\bibitem[{\v S}ot19]{SV19}
Viera {\v S}ottov\'{a}.
\newblock {Cardinal invariant $\lambda (\mathcal{S},\mathcal{J})$}.
\newblock {\em Geyser Math. Cass.}, 1:64--72, 2019.
\newblock (\v Sottov\'a is the second author's maiden name).

\bibitem[{\v S}ot20]{SoDiz}
Viera {\v S}ottov\'{a}.
\newblock {\em The role of ideals in topological selection principles}.
\newblock PhD thesis, Pavol Jozef \v{S}af\'arik University in Ko\v{s}ice,
  Faculty of Science, 2020.

\bibitem[{\v S}{\v S}19]{SoSu}
Viera {\v S}ottov\'{a} and Jaroslav {\v S}upina.
\newblock Principle {${\rm S}_1(\mathcal{P},\mathcal{R})$}: ideals and
  functions.
\newblock {\em Topology Appl.}, 258:282--304, 2019.

\bibitem[{\v S}up16]{SJ}
Jaroslav {\v S}upina.
\newblock Ideal {QN}-spaces.
\newblock {\em J. Math. Anal. Appl.}, 435(1):477--491, 2016.

\bibitem[TZ08]{TsZd08}
Boaz Tsaban and Lyubomyr Zdomskyy.
\newblock Scales, fields, and a problem of {H}urewicz.
\newblock {\em J. Eur. Math. Soc. (JEMS)}, 10(3):837--866, 2008.

\bibitem[vdV25]{TvdV25}
Tristan van~der Vlugt.
\newblock Cardinal characteristics on bounded generalised {B}aire spaces.
\newblock {\em Ann. Pure Appl. Logic}, 176(7):Paper No. 103582, 57, 2025.

\bibitem[\v{S}23]{Su22}
Jaroslav \v{S}upina.
\newblock Pseudointersection numbers, ideal slaloms, topological spaces, and
  cardinal inequalities.
\newblock {\em Arch. Math. Logic}, 62(1-2):87--112, 2023.

\end{thebibliography}

\end{document}